\documentclass[11pt,reqno,abstract=noindent]{scrartcl}
\setkomafont{title}{\normalfont\bfseries\Large}

\setkomafont{author}{\bfseries\small}
\setkomafont{date}{\small}
\usepackage{microtype} 
\usepackage{newpxtext}
\usepackage{newpxmath}
\usepackage{upgreek}
\usepackage[rgb]{xcolor}
\usepackage{caption}

\usepackage[margin=1in]{geometry}

\usepackage{amsmath} 

\usepackage{amsthm}
\usepackage{mathtools}
\usepackage{etoolbox}

\usepackage{algpseudocode}
\usepackage{algorithm}
\usepackage{algorithmicx}

\usepackage{tocloft}

\usepackage{graphicx}
\usepackage{caption}
\usepackage{subcaption}
\usepackage{makecell}
\usepackage{longtable}

\usepackage{tikz}
\usetikzlibrary{arrows.meta, decorations.pathmorphing, decorations.markings, patterns,calc,arrows}

\usepackage{cite}
\usepackage[colorlinks=true, pdfstartview=FitV, linkcolor=blue,citecolor=blue, urlcolor=blue]{hyperref}
\usepackage[nameinlink,capitalise]{cleveref}
\hypersetup{pdftitle={Classical Euler flows generate the Guderley imploding shock wave}, pdfauthor={Giorgio Cialdea, Steve Shkoller, Vlad Vicol}}

\newtheorem{theorem}{Theorem}[section]
\newtheorem{lemma}[theorem]{Lemma}
\newtheorem{prop}[theorem]{Proposition}
\newtheorem{corollary}[theorem]{Corollary}
\newtheorem{definition}[theorem]{Definition}

\theoremstyle{remark}
\newtheorem{remark}[theorem]{Remark}

\newcommand{\p}{\partial}
\newcommand{\R}{\mathbb{R}}
\newcommand{\s}{\mathsf{s}}
\newcommand{\ds}{\dot{\mathsf{s}}}
\newcommand{\ts}{\mathrm{T}_*}
\newcommand{\ti}{\mathrm{T_{in}}}
\newcommand{\tf}{\mathrm{T_{fin}}}
\newcommand{\rW}{\mathring{W}}
\newcommand{\rZ}{\mathring{Z}}
\newcommand{\rK}{\mathring{S}}
\newcommand{\rS}{\mathring{S}}
\newcommand{\rw}{\mathring{w}}
\newcommand{\rz}{\mathring{z}}

\newcommand{\rb}{\mathring{b}}

\renewcommand{\sp}[1]{#1|_\s^+}
\newcommand{\sm}[1]{#1|_\s^-}
\newcommand{\vsp}{v{\scriptstyle |_{\s}^+}}
\newcommand{\usp}{u{\scriptstyle|_{\s}^+}}
\newcommand{\csp}{c{\scriptstyle|_{\s}^+}}
\newcommand{\psp}{p{\scriptstyle|_{\s}^+}}
\newcommand{\rsp}{\rho{\scriptstyle|_{\s}^+}}
\newcommand{\vsm}{v{\scriptstyle|_{\s}^-}}
\newcommand{\psm}{p{\scriptstyle|_{\s}^-}}
\newcommand{\rsm}{\rho{\scriptstyle|_{\s}^-}}
\newcommand{\wsp}{w{\scriptstyle|_{\s}^+}}
\newcommand{\zsp}{z{\scriptstyle|_{\s}^+}}

\newcommand{\wsm}{w{\scriptstyle|_{\s}^-}}
\newcommand{\zsm}{z{\scriptstyle|_{\s}^-}}
\newcommand{\bsm}{b{\scriptstyle|_{\s}^-}}
\newcommand{\usm}{u{\scriptstyle|_{\s}^-}}
\newcommand{\csm}{c{\scriptstyle|_{\s}^-}}
\newcommand{\ssm}{\sigma{\scriptstyle|_{\s}^-}}
\newcommand{\ssp}{\sigma{\scriptstyle|_{\s}^+}}

\newcommand{\ltsm}{\lambda_3{\scriptstyle|_{\s}^-}}
\newcommand{\ltsp}{\lambda_3{\scriptstyle|_{\s}^+}}
\newcommand{\losm}{\lambda_1{\scriptstyle|_{\s}^-}}
\newcommand{\losp}{\lambda_1{\scriptstyle|_{\s}^+}}
\newcommand{\lwsm}{\lambda_2{\scriptstyle|_{\s}^-}}
\newcommand{\lwsp}{\lambda_2{\scriptstyle|_{\s}^+}}

\newcommand{\rwsm}{\rw{\scriptstyle|_{\s}^-}}
\newcommand{\rzsm}{\rz{\scriptstyle|_{\s}^-}}
\newcommand{\rbsm}{\rb{\scriptstyle|_{\s}^-}}

\newcommand{\h}{\mathsf{h}}

\renewcommand{\div}{\makebox{div}}
\newcommand{\g}{\mathsf{g}}

\newcommand{\tsh}{\mathrm{T}_{\! \circ}}
\newcommand{\tfl}{\mathrm{T}_{\! \flat}}
\newcommand{\m}{\mathsf{m}}
\newcommand{\mo}{\mathsf{m}_1}
\renewcommand{\l}{ {\mathsf{c}} }
\newcommand{\dsh}{\delta_{\! \circ}}
\newcommand{\dfl}{\delta_{\! \flat}}

\def\jump#1{{[\hspace{-1.5pt}[#1]\hspace{-1.5pt}]}}
\def\mean#1{{\langle\hspace{-2.5pt}\langle#1\rangle\hspace{-2.5pt}\rangle}}
\definecolor{shcu}{RGB}{255,0,0}
\definecolor{fakeshock}{RGB}{218,165,32}

\numberwithin{equation}{section}

\setcapindent{0pt}
\deffootnote{0em}{1.6em}{\textsuperscript{\thefootnotemark}\,}

\AtBeginEnvironment{abstract}{\small}

\title{Classical Euler flows generate the strong Guderley imploding shock wave}
\author{
  Giorgio Cialdea\thanks{Courant Institute of Mathematical Sciences, New York University, New York, NY 10012, \href{mailto:gc2987@nyu.edu}{gc2987@nyu.edu}}
  \and
  Steve Shkoller\thanks{Department of Mathematics, University of California, Davis, CA 95616, \href{mailto:shkoller@math.ucdavis.edu}{shkoller@math.ucdavis.edu}}
  \and
  Vlad Vicol\thanks{Courant Institute of Mathematical Sciences, New York University, New York, NY 10012, \href{mailto:vcv224@nyu.edu}{vcv224@nyu.edu}}
}
\date{}

\begin{document}
\maketitle

\begin{abstract}
\noindent
\textbf{Abstract.} We prove that Guderley's self-similar imploding shock solution for the compressible Euler equations with ideal--gas law ($\gamma>1$) arises from classical, radially symmetric, shock--free data. For such data prescribed at initial time $\ti < 0$, we prove that the flow remains classical up to a first singular time 
$t=\ts \in (\ti, 0)$, where a preshock forms with a $C^{\frac{1}{3}}$ cusp in the fast acoustic  variable. 
From this preshock a unique, initially weak, regular shock is born, whose strength can be made arbitrarily large on a controlled time interval; the front then deforms onto the Guderley shock and implodes at the origin at the collapse time $t=0$. There exists a matching time $t=\tf \in (\ts,0)$ such that on $[\tf,0)$ the solution \emph{coincides exactly} with the classical Guderley self--similar profile, and at $t=\tf$ the shock trajectory matches the self--similar front to all orders. 
As $t {\scriptscriptstyle \nearrow} 0$, the Euler solution implodes at the center, and continues for
 $t>0$ as a reflected  blast wave, providing a global-in-time unique Euler solution which evolves from regular initial conditions.
\end{abstract}

\tableofcontents

\section{Introduction}
The study of powerful implosions in compressible fluids is dominated by a few canonical solutions that capture the essential dynamics of focusing waves. 
Among the most significant is the Guderley imploding shock \cite{Guderley1942}, a self-similar solution to the Euler equations that describes a strong, 
symmetric shock wave converging to a point. Beyond its historical importance (as the first analytical solution of Euler describing an imploding shock), the Guderley solution is widely regarded in the physics and numerical 
communities as a \textit{universal attractor}\footnote{The modern concept of universal attractor is discussed by Barenblatt~\cite{Barenblatt1996} in terms of the stability of invariant solutions. One linearizes the (autonomous) similarity dynamical system about the self‑similar profile to get a spectral problem for perturbations. If all perturbation eigenvalues have negative real part except the neutral one corresponding to shifting the focusing time (a symmetry), the similarity solution is stable. In Barenblatt’s language, that meant that the flow forgets initial/boundary details and approaches the self‑similar profile in an intermediate‑asymptotic sense—the modern “universal attractor” phrasing.}. This concept is a classic example of \textit{intermediate asymptotics}, a powerful theory of self-similarity 
developed in the foundational works of Barenblatt~\cite{Barenblatt1996}, Zel'dovich and Raizer~\cite{ZeldovichRaizer}, Sedov~\cite{Sedov}, and 
Stanyukovich~\cite{Stanyukovich}.\footnote{Intermediate asymptotics describe the behavior of a physical system at a stage that is late enough for the solution to have ``forgotten'' the fine details of the initial conditions, but early enough that it is not yet influenced by the ultimate boundaries of the domain. In this stage, the system's evolution often follows a universal, self-similar form that depends only on global conserved quantities and physical parameters of the problem (such as the adiabatic exponent $\gamma$ and the geometry). Other canonical examples governed by this principle include the Taylor-von Neumann-Sedov blast wave from a point explosion \cite{Sedov1946, vonNeumann1947, Taylor1950,Taylor1950b}
 and the propagation of a thermal wave from an instantaneous source in a porous medium (the Barenblatt-Pattle solution)~\cite{Barenblatt1996, ZeldovichRaizer}.} 
The theory posits that a wide range of different strong, converging shock waves will naturally evolve towards the unique Guderley self-similar profile, a 
notion that has been confirmed in modern numerical studies~\cite{LaRi1977,Lazarus1981,Ponchaut2006,Hornung2008,Ramsey2010test, RaKaBo2012,RaKa2017,RaBa2019, GiBaKr2023}.

This notion of universal attraction, however, is accompanied by a subtle paradox: the Guderley solution itself is known to be \textit{linearly unstable} to 
(non-radial) perturbations \cite{Brushlinskii1963}. This dual nature, where the solution acts as both an attractor and an unstable separatrix, is illustrated schematically in 
Figure~\ref{fig:attractor}. Any small deviation from the perfect self-similar profile is expected to be amplified as the shock approaches the center. This 
suggests that the Guderley profile acts as an \textit{intermediate asymptotic}: physical solutions are first attracted towards it, and then, as the 
collapse becomes imminent, they are driven away from it by the growth of inherent instabilities.

\begin{figure}[h!]
    \centering
    \begin{tikzpicture}[font=\small]
        \draw[{Stealth}-] (4.5,-0.2) node[below, anchor=north east, xshift=10pt] {Progression of Shock Collapse} -- (10,-.2);
        \node[blue!60!black] at (3.0,-0.8) {(Radius decreases $\leftarrow$)};
        \node[rotate=90, anchor=center] at (11, 0) {Deviation from Guderley Profile};
        
        \draw[black, thick] (0,0) -- (10,0) node[pos=0.5, above=10pt, black] {Guderley Solution};
        
        \node at (7.5, 2.3) {\textbf{Attractor Phase}};
        \node at (7.5, 1.9) {(Forgetting Initial Conditions)};
        \draw[blue!60!black, thick] (10, 1.5) to[out=180, in=20] (6, 0.05);
        \draw[blue!60!black, thick] (10, -1.5) to[out=180, in=-20] (6, -0.05);
        \draw[blue!60!black, thick] (10, 0.8) to[out=180, in=10] (6, 0.02);
        \node[blue!60!black] at (8.0, -1.8) {Various Initial Conditions};

        \draw[dashed, gray] (5,-2.2) -- (5,2.2);
        
        \node at (2.5, 2.3) {\textbf{Instability Phase}};
        \node at (2.5, 1.9) { (Amplification of Perturbations)};
        \draw[red, thick] (4.9, 0) .. controls (4, 0.1) and (3, -0.2) .. (2, 0.5) .. controls (1, 1.5) .. (0.2, 2);
        \node[red, anchor=east] at (1.5, 1) {Perturbation Growth};
    \end{tikzpicture}
    \caption{A schematic illustrating the dual nature of the Guderley solution. At early stages (right), various initial conditions are attracted towards the universal self-similar solution. At late stages (left), as the shock approaches collapse, the inherent instability of this solution amplifies any small residual perturbation, causing the flow to diverge from the ideal path.}
    \label{fig:attractor}
\end{figure}
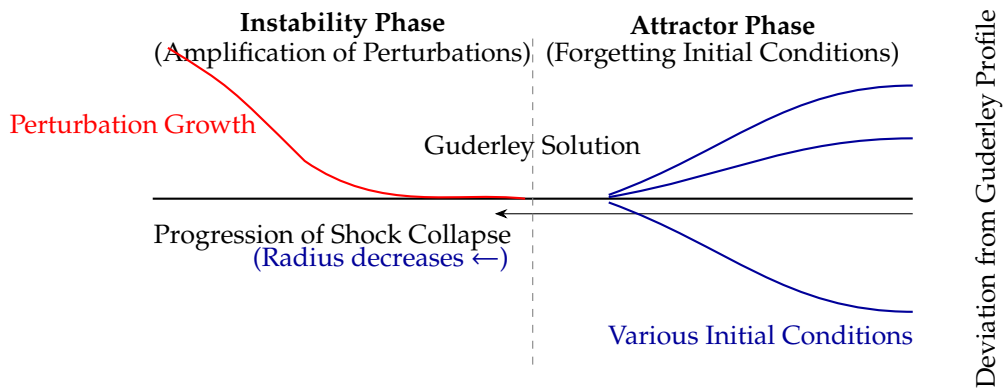

While this physical picture is well-established, it has remained largely a concept of numerical and heuristic asymptotic analysis. A central open problem in the 
mathematical theory of shock waves has been to rigorously establish the basin of attraction for the Guderley state. That is,  

\begin{quote}
\textit{Prove from first principles that 
there exists a set of regular, shock-free initial data whose corresponding unique regular solutions evolve to match the Guderley implosion profile exactly.}
\end{quote}

This work provides an answer to this question; we provide an explicit construction of solutions to the compressible Euler equations that evolve from 
\textit{classical, shock-free initial data} into the \textit{strong shock regime}, culminating in a state that is identical to the classical Guderley self-similar shock 
profile. In doing so, we rigorously prove that the basin of attraction for the Guderley solution is non-empty and contains a large set of initial conditions, 
thereby validating  its role as an attractor.

Our result also overcomes the limitations of prior mathematical constructions for shock development from smooth data, such as those in \cite{Lebaud1994,Yi2004,ChLi2016,Ch2019,BuDrShVi2022}, which were confined to the \textit{weak shock regime}. In that regime, a preshock forms with a $C^{\frac{1}{3}}$ cusp in the solution at the first time of singularity
$t=\ts$, and the subsequent shock strength grows slowly, proportional to $(t-\ts)_ {+}^{\frac{1}{2}}$ for times $t>\ts$ such that $|t-\ts|\ll 1$.
Our  constructed solution successfully navigates the transition from this weak shock state into the fully nonlinear strong shock dynamics described by Guderley.

Furthermore, in the process of proving that smooth data can evolve towards the Guderley imploding shock, we have established a remarkable and complete evolutionary pathway. We present the first known solution to the Euler equations that:
\begin{enumerate}
    \item[(i)] Emanates from regular, shock-free initial data.
    \item[(ii)] Forms a gradient catastrophe (a preshock).
    \item[(iii)] Develops a weak shock.
    \item[(iv)] Transitions into a strong shock.
    \item[(v)] Evolves to \textit{exactly match} the classical Guderley self-similar shock profile at a finite time.
    \item[(vi)] Continues as the Guderley solution to a final implosion at the origin.
\end{enumerate}
See \cref{fig:global:1:tikz} below for a schematic representation of this evolutionary pathway of the Euler solution from initial time $t=\ti$ to the implosion
time $t=0$. For times $t>0$, the Guderley solution exists as a reflected strong shock, modeling an explosive blast wave \cite{Guderley1942, Lazarus1981}.
By combining the steps (i)-(vi) with the rigorous global-in-time  existence theorem \cite{JaJiSc2024} for Guderley's reflected shock solution, we obtain the following corollary to our main: a global-in-time solution for the multi-dimensional Euler equations starting from classical initial data.

We note also that complementary to the strong-shock formation and focusing considered here, there is now a rigorous theory for \emph{smooth} (shock‑free) imploding flows \cite{MeRaRoSz2022a, BuCaGo2025,ShWeWaZh2025} as well as their stability \cite{Bi2021, BuCaGo2025, ChCiShVi2024, CLGSShSt2023,MeRaRoSz2022b}. See also \cite{Je2025,JeJo2024,JeTs2023} for shock-free finite regularity imploding radial Euler solutions.

\subsection{The Euler equations}
In conservation law form, the Euler equations are given by
\begin{subequations}
 \label{euler:md}
\begin{align}
\p_t (\rho u) + \div(\rho u \otimes u + pI) &= 0\,,   \label{eq:momentum}\\
\p_t \rho + \div( \rho u) &=0\,, \label{eq:mass} \\
\p_t E + \div( (p+E)u) &=0\,, \label{eq:energy}
\end{align}
\end{subequations}
where \eqref{eq:momentum} is the conservation of momentum, \eqref{eq:mass} is the conservation of mass, and \eqref{eq:energy} is the conservation of energy.  We shall consider multi-dimensional flows on the spatial domain $ \mathbb{R}^d$ with dimension $d$ either $2$ or $3$. The primitive variables are 
the velocity vector field
$u: \R^d \times \R \to \R^d$, the positive density function $\rho:\R^d \times \R \to \R_+$, the total energy function $E: \R^d \times \R \to \R$ , and
the pressure function $p:\R^d \times \R \to \R$.  The momentum $\rho u$ is the product of density and velocity.   To close the system of conservation laws \eqref{euler:md}, we employ the ideal gas law as the equation
of state and require the pressure function $p=p(u, \rho, E)$ to satisfy
 \begin{equation} \label{eq:EOS}
p = (\gamma-1) ( E - \tfrac{1}{2} \rho |u|^2) \,,
\end{equation}
where $\gamma>1$ denotes the adiabatic exponent.

\subsection{Regular shock solutions}
From classical initial data prescribed at time $t=\ti < 0$, we construct solutions that form a preshock at time $t=\ts \in (\ti, 0)$ and subsequently develop a shock discontinuity. In our construction, this smoothly evolving shock front is not described by a single smooth function but is a composite curve defined piecewise over the time interval $(\ts, \tf]$ for some $\tf < 0$. Let us denote this global shock surface by $\mathcal{S}\subset \R^d \times (\ts, \tf]$.

For the main phase of its evolution, $t \in [\tsh, \tf]$ (where $\tsh$ is carefully chosen to satisfy $\ts < \tsh \ll \tf$), the shock front is given by a curve $r = \s(t)$. This curve is chosen to smoothly 
connect the strong, self-similar Guderley shock at time $\tf$ to a weak shock state at time $\tsh$. However, a direct backwards extension of this curve to the 
preshock time $\ts$ is not viable. The process of extinguishing a strong shock's strength down to zero would force the shock acceleration $\ddot{\s}(t)$ to 
become singular as $t \searrow \ts$ (see \ref{prop:weak2pre} below).
 This would, in turn, create an unphysical, asymmetric preshock state—
one that is smooth on the exterior ($r > \s(\ts)$) but only Hölder continuous on the interior ($r < \s(\ts)$)—which cannot be generated from $C^1$  
initial data.

To resolve this  \textit{strong shock}-to-\textit{weak shock}-to-\textit{preshock} transition, and to ensure the formation of a physically correct singularity, we introduce a second, highly regular curve $r = \l(t)$ on the short, initial time 
interval $t \in (\ts, \tsh]$. This curve provides a $C^\infty$ transition from the weak shock at $\tsh$ to a  $C^{\frac{1}{3}}$ cusp at the preshock time $\ts$. 
The complete shock surface is thus the union of these two pieces:
\[
\mathcal{S} = \big\{ (x,t) : r = \l(t) \text{ for } t \in (\ts, \tsh] \big\} \cup \big\{ (x,t) : r = \s(t) \text{ for } t \in (\tsh, \tf] \big\}.
\]
Across this orientable space-time hypersurface, the fluid variables $(u, \rho, E)$ experience a jump discontinuity. The shock speed and the jumps must satisfy the Rankine-Hugoniot conditions, which state that
\begin{subequations} \label{RHg}
\begin{align*}
\dot{\mathfrak{s}} \jump{\rho u_n}&= \jump{\rho u_n^2 + p} |n| \,, \\
\dot{\mathfrak{s}} \jump{\rho}&= \jump{\rho u_n} |n| \,, \\
\dot{\mathfrak{s}} \jump{E} &= \jump{ (p+E)u_n} |n| \,,\end{align*}
\end{subequations}
where $\dot{\mathfrak{s}}$ denotes the shock speed ($\dot{\l}$ or $\dot{\s}$), and $n$ is the  spatial (outward) unit normal  to the shock front. Our constructed solution satisfies the following standard definitions.

\begin{definition}[\textsl{Regular shock solution \cite{BuDrShVi2022}}] \label{def:regular:shock}
We say that~$(u, \rho, E, \mathcal{S})$ is a regular shock solution on~$\R^d \times (\ts, \tf]$ if:
\begin{itemize}
    \item $(u, \rho, E)$ is a weak solution of~\eqref{euler:md} with strictly positive density.
    \item The shock front~$\mathcal{S}$ is a codimension-$1$ orientable hypersurface.
    \item The solution $(u, \rho, E)$ is $C^1$ smooth in space and time away from $\mathcal{S}$.
    \item The discontinuities across $\mathcal{S}$ satisfy the Rankine-Hugoniot conditions~\eqref{RHg}.
    \item The Lax geometric (entropy) conditions \eqref{lax} are satisfied.
\end{itemize}\end{definition}

\begin{definition}[\textsl{Regular shock solution emanating from a preshock}] \label{def:regular:shock:pre}
A regular shock solution is said to emanate from a preshock if the state variables $(u,\rho, E)$ are continuous across the entire spatial domain at the initial instant of shock formation, $t=\ts$.\end{definition}

\subsection{A symmetric form of the Euler equations}
Let   $S$ denote the specific entropy and let us denote the \textit{square-root of pseudo entropy} by 
\begin{equation*} 
b = e^{\frac{S}{2}}  \,.\end{equation*}
 Then  the ideal gas law  \eqref{eq:EOS} can be re-expressed as 
\begin{equation}
p = \tfrac{1}{\gamma} \rho^\gamma b^2\,. \label{eq:EOS2}
\end{equation}
The sound speed $c$ and the rescaled sound speed $\sigma$ are given by
\begin{equation*}
\sigma = \tfrac{1}{\alpha} 	\rho^\alpha b, \quad c= \alpha \sigma = 	\rho^\alpha b \,,
\qquad \mbox{where} \qquad
\alpha=\tfrac{\gamma-1}{2}.\end{equation*}
It follows that \eqref{eq:EOS2} can be written as
\begin{equation*}
p = \tfrac{\alpha^2}{\gamma} \sigma^2 \rho = \tfrac{1}{\gamma} c^2 \rho = \tfrac{1}{\gamma} \rho^\gamma b^2.
\end{equation*}
Furthermore, the introduction of the new adiabatic exponent $ \alpha >0$ and the rescaled sound speed $\sigma$
allow us to write the Euler equations \eqref{euler:md} in a symmetric form given by
\begin{subequations} \label{euler:sound}
\begin{align}
\p_t u + u\cdot \nabla u + \alpha \sigma \nabla \sigma &= \tfrac{1}{\alpha \gamma} \rho^{2\alpha} b \nabla b \,, \\
\p_t \sigma +u\cdot \nabla \sigma + \alpha \sigma \div u &=0 \,,  \\
\p_t b+ u\cdot \nabla b &=0 \,.  \label{euler:sound:S}
\end{align}
\end{subequations}

\subsection{The Euler equation in radial symmetry}

We will consider radially symmetric solutions of the Euler equations. To this end,  we introduce $r=|x|$ and $e_r = \tfrac{x}{|x|}$ and we assume that
\begin{equation*}
\rho = \rho(r,t), \quad u = u(r,t) e_r, \quad b = b(r,t), \quad \sigma = \sigma (r,t) .
\end{equation*}
In radial symmetry,
the symmetric form of the Euler equations \eqref{euler:sound} is given as
\begin{subequations} \label{euler:r}
\begin{align*}
\p_t u + u \p_r u + \alpha \sigma \p_r \sigma& = \tfrac{1}{\alpha \gamma} \rho^{2\alpha} b \p_r b, \\
\p_t \sigma + u \p_r \sigma + \alpha \sigma \left( \p_r u + \tfrac{d-1}{r} u \right) & = 0, \\
\p_t b + u \p_r b &= 0.\end{align*}
\end{subequations}

We introduce the Riemann variables
 \begin{equation*}
  z = u - \sigma\,, \qquad b \,, \qquad  w = u + \sigma\,, 
 \end{equation*}
 which are associated to the three wave speeds
  
 \label{wave:sp}
\begin{align*} 
 \lambda_1  &= u- \alpha \sigma = \tfrac{1+\alpha}{2} z + \tfrac{1-\alpha}{2} w \,,   \\
 \lambda_2 &= u = \tfrac{1}{2} (z+w) \,,   \\
 \lambda_3 &= u+\alpha\sigma  =  \tfrac{1-\alpha}{2} z + \tfrac{1+\alpha}{2} w \,. \end{align*}

In the case of radial collapse, in which the fluid velocity  $u\leq 0$, the wave speed
$\lambda_1$ represents the fast acoustic wave speed and  $\lambda_3$ is the slow acoustic wave speed.  In such a flow regime, $z$ represents the dominant
Riemann variable, and $w$ is the subdominant Riemann variable.  
The Euler equations \eqref{euler:r} can then be written as
\begin{subequations} 
 \label{euler:rv}
\begin{align}
\p_t z + \lambda_1 \p_r z - \tfrac{\alpha(d-1)}{4r}( w^2-z^2) &= \tfrac{ 1}{\alpha\gamma }\rho^{2\alpha}b\p_r b  \,,   \\
\p_t w + \lambda_3 \p_r w + \tfrac{\alpha(d-1)}{4r}( w^2-z^2)& = \tfrac{1}{\alpha \gamma }\rho^{2\alpha} b \p_r b \,,  \\
\p_t b + \lambda_2 \p_r b &= 0 \,.
\end{align}
\end{subequations}
For non-isentropic dynamics, characteristic variables are traditionally obtained in differential form.  As such, we employ 
 the \textit{Differentiated Riemann Variables} (DRV), defined by
\begin{equation} 
\label{ring:v}
\mathring{z} = \p_r z + \tfrac{1}{\gamma \alpha} \rho^\alpha \p_r b\,, 
\qquad
\mathring{b} = \p_r b\,, 
\qquad 
\mathring{w} = \p_r w - \tfrac{1}{\gamma \alpha} \rho^\alpha \p_rb \,.
\end{equation}
The evolution equations for the system of DRVs is obtained from \cref{euler:rv,ring:v}; the DRV system is written as
\begin{subequations} \label{euler-DRV}
\begin{align}
\p_t \mathring{z} + \lambda_1 \p_r \rz 
& = - \rz \left(\tfrac{1+\alpha}{2}\rz + \tfrac{1-\alpha}{2}\rw - \tfrac{1}{\gamma} \rho^\alpha \rb \right) - \tfrac{1}{4 \gamma} \rho^\alpha \rb (\rw+\rz)
\notag \\
& \qquad \qquad
+ (d-1)\tfrac{ \gamma \alpha(\rw w - \rz z)+ (w+z)\rho^\alpha \p_r b }{2 \gamma r} - 2(d-1)\alpha \tfrac{ w^2 - z^2}{r^2}  \label{DRV-z}
\,,  \\
\p_t \mathring{w} + \lambda_3 \p_r \rw 
& = - \rw\left(\tfrac{1+\alpha}{2}\rw + \tfrac{1-\alpha}{2}\rz + \tfrac{1}{\gamma} \rho^\alpha \rb \right)+ \tfrac{1}{4 \gamma} \rho^\alpha \rb (\rw+\rz) 
 \notag \\
& \qquad \qquad
-(d-1)\tfrac{ \gamma \alpha(\rw w - \rz z)+ (w+z)\rho^\alpha \p_r b }{ 2 \gamma r} + 2(d-1)\alpha \tfrac{ w^2 - z^2}{r^2} 
\,, \\
\p_t \rb + \lambda_2 \p_r \rb 
&= - \rb \left(\tfrac{1}{2}\rw + \tfrac{1}{2}\rz \right) \,,
\end{align}
\end{subequations}
and these equations are coupled with \eqref{euler:rv}.
Similar DRV evolution equations  were used in \cite{NSV2025-CAM}, \cite{NSV2025-JLMS}, and \cite{ShVi2024}.

\subsection{Guderley's self-similar imploding shock solution}
\label{sec:Guderley}
The focusing strong shock solution of  Guderley \cite{Guderley1942} was the first self-similar solution of the second kind, 
requiring the similarity exponent to be determined by solving a nonlinear eigenvalue problem\footnote{%
In the self-similar formulation \eqref{guderley:JLS}–\eqref{ODE:nd}, set
$\mathcal D(\xi)=(1+U(\xi)/\xi)^{2}-(C(\xi)/\xi)^{2}.$
A sonic (critical) point \(\xi_s\) is characterized by \(\mathcal D(\xi_s)=0\), i.e. $\xi_s+U(\xi_s)=\pm\,C(\xi_s)$
(the imploding branch typically takes the minus sign). Because \(\mathcal D\)
appears in the denominators of \eqref{ODE}, regularity of the solution at
\(\xi=\xi_s\) requires the numerators to vanish simultaneously:
\[
  G\!\Big(\tfrac{U(\xi_s)}{\xi_s},\,\tfrac{C(\xi_s)}{\xi_s}\Big)=0,
  \qquad
  F\!\Big(\tfrac{U(\xi_s)}{\xi_s},\,\tfrac{C(\xi_s)}{\xi_s}\Big)=0,
\]
with $F,G$ given in \eqref{ODE:nd}. Together with the Rankine--Hugoniot
shock data at $\xi=1$, these algebraic compatibility conditions are satisfied
only for a (typically unique) value of the similarity exponent
$\lambda=\lambda(\gamma,d)$. Since $\lambda$ enters $F$ and $G$, Barenblatt has termed this a so-called
\emph{nonlinear eigenvalue problem}: one shoots from $\xi=1^+$ using the jump
conditions and adjusts $\lambda$ until the conditions at $\xi_s$ are met and
the continuation remains regular; the density \(R\) then follows from
\eqref{id:rho}. See also~\cite{HiGr2001} for a modern perspective on the computation of $\lambda$.} 
rather than being determined by dimensional analysis arguments and scaling.\footnote{The modern method of self-similarity emerged in the 1940s, seeking special solutions for which the spatial profile of the flow remains the same shape over time, only changing in scale according to a power law. This simplifies the complex system of partial differential equations into a more manageable system of ordinary differential equations. Two of the most famous and physically significant self-similar solutions from this era describe opposite phenomena.
The \textit{Guderley  implosion} describes a strong, symmetric shock wave converging to a point and the subsequent reflected shock that propagates outward after collapse. It was first developed by K. G. Guderley in Germany in 1942 \cite{Guderley1942}. Due to later, independent work, it is often referred to as the Guderley–Landau–Stanyukovich problem. The 
 \textit{Taylor–von Neumann–Sedov  blast wave}  describes a strong, symmetric shock wave expanding from a point explosion. It was famously solved independently by G. I. Taylor in the United Kingdom and John von Neumann in the United States during their work on the atomic bomb, and by L. I. Sedov in the Soviet Union.
The complete theories of these solutions and their underlying mathematical frameworks are detailed in the foundational texts of the field by Landau \& Lifshitz \cite{LandauLifshitz}, Zel'dovich \& Raizer \cite{ZeldovichRaizer}, Stanyukovich \cite{Stanyukovich}, and Sedov \cite{Sedov}.}

The focusing phase of the solution is assumed to occur for time $t<0$ with the physical radial coordinate $r \in [0, \infty )$,  and the similarity variables are chosen so that the implosion occurs at the center $r=0$,  precisely at time $t=0$.  The  self-similar variable $\xi$  assumes the form
\begin{equation*}
\xi = \tfrac{\  \ r^\lambda}{-t} \,.
\end{equation*}
We refer to $\lambda>1$ as the \textit{similarity exponent}.
For $r\ge0$ and $t<0$, the primitive variables have the self-similar ansatz
\begin{subequations} 
\label{guderley:JLS}
\begin{align}
u(r, t) &= 0, \quad c(r,t ) = 0, \quad \rho(r,t) = 1, \quad &0 \le r^\lambda < -t, \\
u(r,t) &= \tfrac{1}{\lambda}r^{1-\lambda} U\big(\tfrac{r^\lambda}{-t} \big), \quad &-t<r^\lambda, \\
c(r,t)& = \tfrac{1}{\lambda} r^{1-\lambda} C\big(\tfrac{r^\lambda}{-t} \big), \quad &-t<r^\lambda, \\
\rho(r,t) &= R\big(\tfrac{r^\lambda}{-t} \big), \quad &-t<r^\lambda ,
\end{align}
\end{subequations}
where the self-similar profiles $U, C: [1,+\infty] \to \R$ are solutions of
\begin{subequations} 
\label{ODE}
\begin{align}
\xi\p_\xi \Big( \frac{U}{\xi}\Big) = -\tfrac{1}{\lambda } \tfrac{G\big(\tfrac{U}{\xi}, \tfrac{C}{\xi} \big)}{D\big(\tfrac{U}{\xi}, \tfrac{C}{\xi}\big)} \,, \\
\xi\p_\xi \Big(\frac{C}{\xi} \Big) = -\tfrac{1}{\lambda } \tfrac{F\big(\tfrac{U}{\xi}, \tfrac{C}{\xi}\big)}{D\big(\tfrac{U}{\xi}, \tfrac{C}{\xi}\big)} \,,
\end{align}
\end{subequations}
and
\begin{subequations} 
\label{ODE:nd}
\begin{align}
 D(U,C) =& (1+U)^2 - C^2 \,,\\
 G(U,C) = &C^2 ( d\, U + \tfrac{2(\lambda-1)}{\gamma} ) - U(1+U)(\lambda+U) \,, \\
 F(U,C) = &C \Big( C^2 ( 1+ \tfrac{\lambda-1}{\gamma}) - \big( 1+ \tfrac{(d-1)(\gamma-1)}{2} \big)(1+U)^2\\ \notag
 & \qquad + \tfrac{(d-1)(\gamma-1) + (\lambda-1)(\gamma-3)}{2}(1+U) -(\lambda-1)(\gamma-1)\Big) \,. 
 \end{align}
For $\xi\in[1,+\infty]$,  the density profile $R(\xi)$  is then computed using the algebraic identity
\begin{equation} 
\label{id:rho}
R(\xi)^{ \tfrac{2(\lambda-1)}{d} + 1 - \gamma} C(\xi)^2 \big(1 + \tfrac{U(\xi)}{\xi}\big)^{\frac{2(\lambda-1)}{d}} = \tfrac{2 \gamma (\gamma-1)^{1-\gamma} }{(1+\gamma)^{1+\gamma}} .
\end{equation}
The identity~\eqref{id:rho} is a simple consequence of the transport of specific entropy~\eqref{euler:sound:S} and the self-similar ansatz. 

In the self-similar coordinate $\xi$, the shock front is fixed at $\xi=1$, and the  Rankine-Hugoniot jump conditions\footnote{
For the Guderley solution, the gas ahead of the shock is assumed to be cold and at rest. It is assumed that $u_- = 0$,  $p_- = 0$,  and $\rho_- =1$.
The assumption of zero pressure implies that the sound speed in the quiescent region, $c_- = \sqrt{\gamma p_- / \rho_-}$, is also zero.
The speed of the shock front is $\dot{\g} = d\g/dt$. The Mach number is defined as the ratio of the shock speed to the sound speed of the medium into which it is propagating: $ M = \frac{|\dot\g|}{c_-} \,. $
Given that $c_- = 0$ in the Guderley problem, the Mach number is technically infinite from the outset.  For this reason, the problem is described as a \textit{strong} shock. To see the relationship clearly, we first derive the jump conditions for a general finite Mach number and then take the limit.
In the frame of reference moving with the shock, the conservation laws across the discontinuity can be written as
\begin{align*}
\rho_- |\dot\g| &= \rho_+ (|\dot\g| - u_+) \quad &\text{(Mass)}  \,,  \\
p_- + \rho_- \dot\g^2 &= p_+ + \rho_+ (|\dot\g| - u_+)^2 \quad &\text{(Momentum)}  \,, \\
\tfrac{\gamma}{\gamma-1}\tfrac{p_-}{\rho_-} + \tfrac{1}{2}\dot\g^2 &= \tfrac{\gamma}{\gamma-1}\tfrac{p_+}{\rho_+} + \tfrac{1}{2}(|\dot\g| - u_+)^2 \quad &\text{(Energy)} \,.
\end{align*}
Solving this system for a general $p_-$ and $c_- = \sqrt{\gamma p_- / \rho_-}$ yields the jump conditions in terms of the Mach number $M = |\dot\g|/c_-$:
\begin{align*}
\tfrac{\rho_+}{\rho_-} = \tfrac{(\gamma+1)M^2}{(\gamma-1)M^2 + 2} \,,  \quad
\tfrac{p_+}{p_-} = \tfrac{2\gamma M^2 - (\gamma-1)}{\gamma+1}  \,,  \quad
\tfrac{u_+}{|\dot\g|} = \tfrac{2}{\gamma+1}\left(1 - \tfrac{1}{M^2}\right) \,. 
\end{align*}
Taking the limit as $M \to \infty$ gives the conditions used for the Guderley solution as
\begin{align*}
\tfrac{\rho_+}{\rho_-} \to \tfrac{\gamma+1}{\gamma-1} \,,  \quad
p_+ \to \tfrac{2\rho_- \dot\g^2}{\gamma+1}  \,, \quad
u_+ \to \tfrac{2|\dot\g|}{\gamma+1} \,.
\end{align*}
}
 (which we will detail below)
provide the boundary conditions for
the system of ODE in \eqref{ODE}, which we write as
\begin{equation} 
U(1) =-\tfrac{2}{\gamma+1}, \quad C(1) = \tfrac{\sqrt{\gamma(\gamma-1)}}{\gamma+1}, \quad R(1) = \tfrac{\gamma+1}{\gamma-1}.
\end{equation}
\end{subequations}
Denoting the Guderley shock curve by $\{r=\g(t)\}$ with $\g(t) = (-t)^{\frac{1}{\lambda}}$ (see Figure \ref{fig:shock_schematic}), the restriction (or trace) of the Guderley shock solution \eqref{guderley:JLS} to both
the $^-$ and $^+$ sides of $\g(t)$ (in physical spacetime coordinates $(r,t)$) are given by
\begin{subequations} 
\begin{align}
u(r,t)|_{\g}^- &=0, \quad c(r,t)|_{\g}^- =0, \quad \rho(r,t)|_{\g}^- =1, \\
u(r,t)|_{\g}^+ &= -\tfrac{2}{(\gamma+1) \lambda} (-t)^{^{-\frac{\lambda-1}{\lambda}}}, 
\quad c|_{\g}^+ = \tfrac{ \sqrt{ 2\gamma (\gamma-1)}}{(\gamma+1)\lambda} (-t)^{^{-\frac{\lambda-1}{\lambda}}}, \quad \rho|_{\g}^+ = \tfrac{\gamma+1}{(\gamma-1)}, \label{u-c-rho-plus}
\end{align}
\end{subequations} 
which show that
\begin{equation*} 
\jump{u} \sim ( -t)^{^{-\frac{\lambda-1}{\lambda}}}, \quad \jump{c} \sim ( -t)^{^{-\frac{\lambda-1}{\lambda}}}, \quad \jump{\rho}\sim 1.\end{equation*}
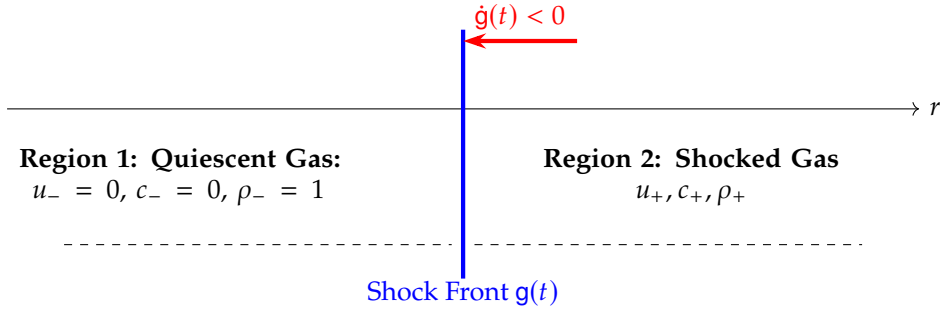
\begin{figure}[h!]
    \centering
    \begin{tikzpicture}[font=\small,scale=1.5]
        \draw[->] (-1,0) -- (7,0) node[right]{$r$};
        \draw[ultra thick, blue] (3,-1.5) -- (3,0.7) node[below=90pt, blue]{Shock Front $\g(t)$};
        \draw[ultra thick, -{Stealth[length=3mm, width=2mm]}, red] (4,0.6) -- (3.,0.6) node[midway, above]{$\dot{\g}(t) < 0$};
        
        \node[text width=5cm, align=center] at (0.5, -0.6) {
            \textbf{Region 1: Quiescent Gas: } 
            $u_-=0$, $c_-=0$, $\rho_-=1$
        };
        \draw[dashed] (-0.5,-1.2) -- (2.9,-1.2);
        
        \node[text width=5cm, align=center] at (5, -0.6) {
            \textbf{Region 2: Shocked Gas} \\
            $u_+, c_+, \rho_+$
        };
        \draw[dashed] (3.1,-1.2) -- (6.5,-1.2);
    \end{tikzpicture}
    \caption{A schematic of the converging shock front. The shock moves left into the quiescent Region 1 (the core), leaving behind the compressed, hot gas of Region 2 (the outer layer it has passed through).}
    \label{fig:shock_schematic}
\end{figure}

We observe, that even if $\jump{u} \to 0 $ and  $\jump{c} \to 0$  as $t \to - \infty,$ the local Mach number remains a constant of size $O(1)$ and  
\begin{equation*}
\tfrac{\jump{u}}{\mean{c}} = \tfrac{2}{\sqrt{2 \gamma(\gamma-1)}}, \quad \tfrac{\jump{c}}{\mean{c}} = 2\,,
\end{equation*}
where $\mean{\cdot }$ denotes the average across the shock.
Note that for any fixed $\kappa>0$, there exists a time $\tf<0$ such that
\begin{align*}
u|_{\g}^- (\tf)&=0, \quad c|_{\g}^-(\tf) =0, \quad \rho|_{\g}^-(\tf) =1, \\
\big|u|_{\g}^+(\tf)\big| & \le \kappa, \quad c|_{\g}^+(\tf) \le \kappa \quad \rho|_{\g}^+(\tf) = \tfrac{\gamma+1}{\gamma-1}.
\end{align*}
We have chosen $\kappa$ to be proportional to $ (-\tf)^{\frac{1-\lambda}{\lambda}} $.

Guderley's seminal work \cite{Guderley1942} formally derived the self-similar laws governing an imploding shock and numerically identified the critical similarity exponents. In the following decades, related analyses by others, such as Hunter \cite{Hunter1960}, further elucidated the self-similar ODE reductions. Despite this long history, these foundational results remained at a formal or numerical level, leaving key mathematical questions unanswered. In particular, a rigorous proof for the existence of analytic solutions, the justification for continuing these solutions across the sonic degeneracy, and the verification that the resulting profiles constitute genuine weak entropy solutions of the full non-isentropic Euler equations were all missing.

These long-standing challenges were recently resolved by Jang-Liu-Schrecker \cite{JaLiSc2025}, where the existence of self-similar converging shock solutions was
established for the full relevant range of adiabatic exponents, $\gamma \in (1,3]$. The  analysis in  \cite{JaLiSc2025} rigorously proves that the smooth regions of the flow are real-analytic, crucially resolves the analytic continuation across the sonic point, and demonstrate that the constructed profiles satisfy the Rankine--Hugoniot jump conditions as well as the necessary entropy conditions. The results of   \cite{JaLiSc2025} contain the first complete 
PDE-theoretic justification of the self-similar implosion scenario. 
We shall henceforth fix one adiabatic exponent $\gamma \in (1, 3]$, and one profile~$(U, C, R)$ from \cite{JaLiSc2025}. All constants will be allowed to depend on this chosen profile and~$\gamma$.

\subsection{Statement of the main theorem}
\label{sec:main:result}
\begin{theorem}
For fixed spatial dimension $d \in \{2,3\}$, and for any final matching time $\tf < 0$, there exist an initial time $\ti$ and a preshock time $\ts$ satisfying $\ti < \ts < \tf$, and a corresponding radial initial data set $(w_{\mathrm{in}}, z_{\mathrm{in}},b_{\mathrm{in}})$ at time $t=\ti$ which has the following properties:
\begin{itemize}
    \item \textsl{Regularity:} The data belongs to the space $C^{1,{\frac{1}{3}}}(\R^+)$.
    \item \textsl{Quiescent Core:} The data represents a fluid at rest in the core region, with $(u_{\mathrm{in}},\rho_{\mathrm{in}},b_{\mathrm{in}}) \equiv (0,1,0)$ for all radii $r < (-\tf)^{\frac{1}{\lambda}}$.
\end{itemize}
This initial data generates a unique radial solution to the Euler equations \eqref{euler:rv} whose evolution unfolds in the following stages:

\begin{enumerate}
    \item \textsl{Smooth Evolution:} For $t \in [\ti, \ts)$, the solution $(w,z,b)$ evolves smoothly and remains in $C^{1,{\frac{1}{3}}}(\R^+)$.
    
    \item \textsl{Preshock Formation:} At the precise time $t=\ts$, a preshock singularity forms at a single spatial location $r_* > 0$. The dominant Riemann variable $z$ develops a $C^{\frac{1}{3}}$ cusp, while $w$ and $b$ retain higher regularity. The solution exhibits the local asymptotic behavior:
\begin{subequations} 
 \label{exp:xs}
\begin{align}
z(r,\ts) &= z(r_*, \ts) + \mathsf{a} (r-r_*)^{\frac{1}{3}} + \mathsf{b}(r-r_*)^{\frac{2}{3}} + O(|r-r_*|), \\
w(r,\ts) &= w( r_*, \ts) + \mathsf{c}_w (r-r_*) + O(|r-r_*|^{\frac{4}{3}}), \\
b(r,\ts) &= b(r_*,\ts) + \mathsf{c}_b (r-r_*) + O(|r-r_*|^{\frac{4}{3}}),
\end{align}
\end{subequations}
for suitable constants $\mathsf{a}, \mathsf{b}, \mathsf{c}_w, \mathsf{c}_b \in \R$.

    \item \textsl{Shock Development:} For $t \in (\ts, \tf]$, a discontinuous shock emerges from the preshock point, propagating along a front $\{ r = \s(t) \}$ with $\s(\ts)=r_*$. The flow evolves as a unique ``regular shock solution'' to \eqref{euler:md}, remaining $C^\infty$ on either side of the shock front.

    \item \textsl{Matching with Guderley Solution:} At time $t=\tf$, the solution $(\rho,u,E)$ (satisfying the Lax entropy conditions)  and the shock front $\s(\tf)$ seamlessly match the classical Guderley self-similar imploding-shock solution (described in \S~\ref{sec:Guderley}). For all subsequent times $t \in [\tf, 0)$, the evolution is identical to the Guderley solution.

    \item \textsl{Final Implosion:} At $t=0$, the evolution culminates in an implosion singularity at the origin $r=0$, where the velocity $u \to -\infty$ and the energy $E \to +\infty$, while the density $\rho$ remains bounded.
\end{enumerate}\end{theorem}

\begin{figure}[htb!]
\centering
\begin{tikzpicture}[
 	font=\sffamily\small,
 	scale=1.2,
 	>=latex
]
 	\coordinate (P) at ({6/pow(10,0.7) * pow(10-4,0.7)}, 4);
 	\coordinate (Q) at ({6/pow(10,0.7) * pow(10-3,0.7)}, 3);
 	\coordinate (R) at ({6/pow(10,0.7) * pow(10-8,0.7)}, 8);
 	
 	\coordinate (rin3) at (5.25, 1); 	
 	\coordinate (rstar) at (4.7, 1); 
 	\coordinate (rin2) at (5.75, 1); 	
 	\coordinate (rin1) at (6.25, 1); 	
 	\coordinate (r01) at (7.0, 1); 	

 	\fill[cyan!20, opacity=0.8] (R |- 0,1) -- (R) -- plot[domain=8:9.99, samples=50, variable=\t] ({6/((10)^0.7) * (10-\t)^0.7}, {\t}) -- (0,10) -- (0,1) -- cycle;

 	\fill[white] (R |- 0,1) -- (rin1) -- (Q) -- (Q |- R) -- cycle;

 	\begin{scope}
 	 	\clip (rin1) -- (Q) -- plot[domain=3:9.99, samples=100, variable=\t] ({6/((10)^0.7) * (10-\t)^0.7}, {\t}) -- (8.5, 10) -- (8.5, 1) -- cycle;
 	 	\fill[yellow!20, opacity=0.8] (0,1) rectangle (8,10);
 	\end{scope}

 	\fill[white] (Q) -- (P) -- (r01) -- (rin1) -- cycle;

 	
 	\draw[->, thick] (0, 1) -- (8.5, 1);
 	\node at (8.5, 0.8) {Space r};
 	\draw[->, thick] (0, 1) -- (0, 10.5) node[above] {Time t};

 	\draw[orange, ultra thick] plot[domain=3:8, samples=100, variable=\t] ({6/((10)^0.7) * (10-\t)^0.7}, {\t});
 	\draw[red!80!black, ultra thick] plot[domain=8:9.99, samples=50, variable=\t] ({6/((10)^0.7) * (10-\t)^0.7}, {\t});
 	
 	\coordinate (g_label_pos) at ({6/pow(10,0.7) * pow(10-9.5,0.7)},9.5);
 	\node[anchor=west, xshift=4pt, text=red!80!black] at (g_label_pos) {$\g(t)$};
 	\coordinate (s_label_pos) at ({6/pow(10,0.7) * pow(10-5.5,0.7)-0.1},5.5);
 	\node[anchor=west, xshift=4pt, text=red!80!black] at (s_label_pos) {$\s(t)$};
 	\coordinate (n_label_pos) at ({6/pow(10,0.7) * pow(10-3.5,0.7)},3.5);
 	\node[anchor=west, xshift=-22pt, text=red!80!black] at (n_label_pos) {$\l(t)$};
 	
 	
 	\node[blue!60!black, align=center] at (1.1, 6.5) {Rest State \\ $(u, \rho, c) \hspace{1.25em}$ \\ $= (0,1,0)$};
 	\node[orange!80!black, align=center] at (6.5, 6) {Guderley's \\ Self-Similar \\ Solution};

 	\node[rotate=90, anchor=center, red!60!black, align=center] at (-0.5, 9) {\scriptsize Guderley\\ \scriptsize Solution};
 	
 	\draw[dotted, thick] (0, 8) -- (8, 8);
 	\node[anchor=east] at (0, 8) {$\tf$};
 	\node[rotate=90, red!60!black, align=center, anchor=center] at (-0.5, 6) {\scriptsize Shock Development\\ \scriptsize Weak to Strong };

 	\draw[dotted, thick] (0, 10) -- (8, 10);
 	\node[anchor=east] at (0, 10) {$0$};	

 	\draw[dotted, thick] (0, 4) -- (8, 4);
 	\node[anchor=east] at (0, 4) {$\tsh$};
 	\node[rotate=90, red!60!black, align=center, anchor=center] at (-0.5, 3.5) {\scriptsize Weak \\ \scriptsize Shock};

 	\draw[dotted, thick] (0, 3) -- (8, 3);
 	\node[anchor=east] at (-0.1, 3) {$\ts$};
 	
 	\node[anchor=east] at (0, 1) {$\ti$};
 	\node[rotate=90, red!60!black, align=center, anchor=center] at (-0.5, 2) {\scriptsize Shock\\ \scriptsize Formation};

 	\draw[dotted, thick] (R) -- (R |- 0,1) node[below] {$r_{\mathsf{fin}}$};
 	
 	\fill (P) circle (2pt);
 	\fill (Q) circle (2pt);
 	\fill (R) circle (2pt);

 	\path[draw,blue!70!black,%
 	 	 	decoration={%
 	 	 	 	markings,%
 	 	 	 	mark=at position 0.5 	with {\arrow[scale=1.25]{stealth}},%
 	 	 	},%
 	 	 	postaction=decorate] (rin3) -- (Q) node[pos=0.2, left=-0.2em, blue!70!black] {$\eta$};
 	\path[draw,purple!80!black,%
 	 	 	decoration={%
 	 	 	 	markings,%
 	 	 	 	mark=at position 0.5 	with {\arrow[scale=1.25]{stealth}},%
 	 	 	},%
 	 	 	postaction=decorate] (rin2) -- (Q) node[pos=0.25, right=-0.1em, purple!80!black] {$\phi$};
 	\path[draw,green!60!black,%
 	 	 	decoration={%
 	 	 	 	markings,%
 	 	 	 	mark=at position 0.5 	with {\arrow[scale=1.25]{stealth}},%
 	 	 	},%
 	 	 	postaction=decorate] (rin1) -- (Q) node[pos=0.5, right, green!60!black] {$\psi$}; 	 	
 	\path[draw,green!60!black,%
 	 	 	decoration={%
 	 	 	 	markings,%
 	 	 	 	mark=at position 0.5 	with {\arrow[scale=1.25]{stealth}},%
 	 	 	},%
 	 	 	postaction=decorate] (r01) -- (P) node[pos=0.5, right, green!60!black] {$\psi$}; 	 	 	
 	\draw[dotted, black, thick] (rstar) -- (Q) node[pos=0.5, above, black] {$ $};

 	\node[below] at (rin3) {$r_{\mathsf{in}}^{(3)}$};
 	\node[below] at (rstar) {$r_*$};
 	\node[below] at (rin2) {$r_{\mathsf{in}}^{(2)}$};
 	\node[below] at (rin1) {$r_{\mathsf{in}}^{(1)}$};
 	\node[below] at (r01) {$r_0^{(1)}$};
 
\end{tikzpicture}

\caption{\footnotesize The global picture showing the main result. The \emph{initial data} is specified at time $t=\ti$; this data is $C^2$-smooth except for the points $\{r_{\mathrm{in}}^{(i)} \}_{i=2}^{3}$, which are the backwards-in-time images of the preshock along the wave speeds $\{\lambda_i\}_{i=2}^3$; at these locations, the initial data has $C^{1,{\frac{1}{3}}}$ regularity. \emph{Shock formation} occurs on the time interval $[\ti,\ts)$. At time $t=\ts$ and position $r=r_*$ a $C^{\frac{1}{3}}$ preshock develops for the Riemann-variable $z$; the Riemann-variables $w$ and $b$ remain $C^{1,{\frac{1}{3}}}$-smooth. \emph{Shock development} occurs on the time interval $(\ts,0)$; this interval is sub-divided intro three pieces. On $(\ts,\tsh)$ we transition from the preshock into a weak-shock, for all three Riemann variables; on either side of the shock surface $\{r = \s(t)\}$ the solution is $C^\infty$ smooth in space. On $(\tsh,\tf)$ we transition from a weak-shock into the exact Guderley shock state (the self-similar profile evaluated at $\tf$). On $(\tf,0)$ the dynamics is precisely described by the exact Guderley imploding shock. We emphasize that in the light blue region of spacetime, the solution is at ``rest'', meaning that velocity and sound speed equal vanish identically, while the density is a constant. In the yellow region of spacetime, the solution exactly coincides with the Guderley solution, when restricted to that spacetime. We also emphasize that the shock surface $\{r=\s(t) \colon t\in (\tsh,\tf)\} \cup \{ r = \l(t) \colon t \in (\ts, \tsh]\}$ cannot be determined solely by knowledge of the Guderley data along the time-slice$\{t = \tf\}$. Hence, we \emph{choose} the shock surface $ \{ r=\s(t) \colon t \in (\tsh, \tf)\}\cup \{ r = \l(t) \colon t \in (\ts, \tsh]\}$ and the trace $\lambda_1(\l(t)^+, t)$ for~$t \in (\ts, \tsh)$. This gives us all the data required in order to \emph{uniquely} solve the Euler equations in the white region.}
\label{fig:global:1:tikz}
\end{figure}
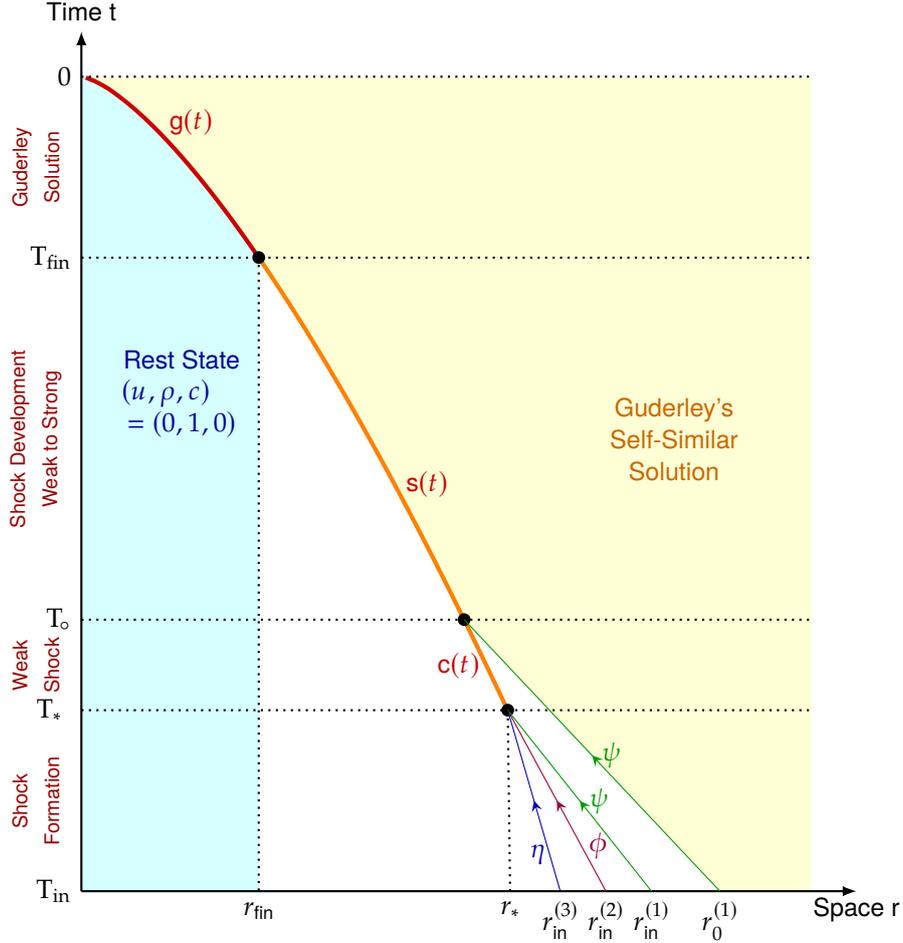

\begin{corollary}[Global-in-time Euler evolution from classical data]
\label{cor:global}
For spatial dimension $d \in \{2,3\}$, there exists an initial time $\ti < 0$ and radially symmetric initial data $(u_{\mathrm{in}}, \rho_{\mathrm{in}}, E_{\mathrm{in}})$ belonging to the space $C^{1,{\frac{1}{3}}}(\R^+)$, such that the corresponding unique solution $(u,\rho,E)$ to the Euler equations \eqref{euler:md} exists globally in time on the interval $[\ti, \infty)$.

This solution evolves through the complete sequence of physical phenomena described in Theorem~\ref{sec:main:result}: smooth evolution, shock formation, transition to the strong Guderley implosion at $t=0$, and subsequent reflection outward as the strong, self-similar Guderley blast wave for all $t>0$.
\end{corollary}

\begin{proof}[Proof of \cref{cor:global}]
The proof relies on combining the construction in Theorem~\ref{sec:main:result} with the established theory for the reflected Guderley shock.

\begin{enumerate}
    \item \textsl{Implosion Phase ($t \in [\ti, 0)$):} Theorem~\ref{sec:main:result} establishes the existence of a unique solution emanating from $C^{1,{\frac{1}{3}}}$ initial data at $t=\ti$. This solution evolves to exactly match the Guderley self-similar imploding profile on an interval $[\tf, 0)$, culminating in the implosion singularity at $t=0$.

    \item \textsl{Explosion Phase ($t \ge 0$):} The Guderley solution possesses a unique continuation past the implosion time $t=0$ as a reflected, diverging strong shock wave (a blast wave). The rigorous existence, uniqueness, and analyticity of this reflected self-similar solution for $t\ge 0$ were established in \cite{JaJiSc2024}.
\end{enumerate}
By concatenating the solution constructed in Theorem~\ref{sec:main:result} on the time interval $[\ti, 0)$ with the unique reflected Guderley solution on $[0, \infty)$, we obtain the desired global-in-time solution.
\end{proof}

\begin{remark}[\textsl{Sharp Regularity of the Initial Data: Finitely Differentiable Structure}]
The initial data constructed in Theorem~\ref{sec:main:result} is globally $C^{1,{\frac{1}{3}}}(\R^+)$. This guarantees that the first derivatives of the physical variables are bounded and ${\frac{1}{3}}$-H\"older continuous everywhere. However, the higher-order regularity profile is structured and distinguishes between regions outside and inside the backward acoustic cone emanating from the preshock (the interval $[r_{\mathrm{in}}^{(3)}, r_0^{(1)}]$ in Figure~\ref{fig:global:1:tikz}).

\begin{enumerate}
    \item \textsl{Exterior Regions ($C^\infty$):} Outside the acoustic cone ($r < r_{\mathsf{in}}^{(3)}$ and $r > r_{\mathsf{in}}^{(1)}$), the solution is $C^\infty$. These regions include the quiescent core ($r<r_{\mathsf{fin}}$) and the smooth Guderley profile ($r>r_{0}^{(1)})$.

    \item \textsl{Interior Region (Finitely Differentiable, not $C^\infty$):} Inside the acoustic cone, the solution is only finitely differentiable. It is not $C^\infty$. This arises from two complementary perspectives:
    
    \begin{itemize}
        \item[a)] \textsl{Construction Methodology:} The backward evolution (Section 5) utilizes a coordinate transformation (Eq. \eqref{choice:psi:smooth}) designed to regularize the $C^{1/3}$ preshock cusp. This transformation maps the physical preshock profile (Eq. \eqref{exp:xs}) into a state that is generally only $C^2$ at $t=\ts$. This occurs because the $O(r-r_*)$ term in the expansion \eqref{exp:xs} generally possesses asymmetry (different coefficients on the left and right of $r_*$), which translates into a jump in the third derivative of the transformed variables. The backward evolution propagates this finite regularity, resulting in initial data that is generally $C^2$ (locally $W^{2,\infty}$), but not $C^\infty$.
        
        \item[b)] \textsl{Nonlinear Contamination:} For nonlinear hyperbolic systems, localized lack of smoothness can affect the entire domain of dependence. The coefficients of the Euler system (the wave speeds $\lambda_i$) depend on the solution; since the solution is globally only $C^{1,1/3}$ within the cone (see Item 3a below), the coefficients are also only $C^{1,1/3}$, which generally prevents the solution from attaining $C^\infty$ regularity.
    \end{itemize}

    \item \textsl{Localized Singularities:} The regularity is further reduced at specific points corresponding to the characteristics emanating from the preshock:

    \begin{enumerate}
        \item \textsl{Slow/Entropy Characteristics ($r_{\mathrm{in}}^{(2)}, r_{\mathrm{in}}^{(3)}$) —  $C^{1,{\frac{1}{3}}}$:}
        The singularity propagates backward along the $\phi$ (entropy) and $\eta$ (slow acoustic) characteristics. The backward regularization yields  $C^{1,{\frac{1}{3}}}$ regularity at $r_{\mathrm{in}}^{(2)}$ and $r_{\mathrm{in}}^{(3)}$.
        
        While the first derivatives are bounded, the second derivatives are unbounded near these points. The analysis provides the upper bound: $|\partial_r^2 u(r)| = O\left(|r-r_{\mathrm{in}}^{(i)}|^{-{\frac{2}{3}}}\right)$ as $r \to r_{\mathrm{in}}^{(i)}$. Thus, the solution is not $W^{2,\infty}$ globally inside the cone.

        \item \textsl{Fast Characteristic ($r_{\mathrm{in}}^{(1)}$) —  $C^2$:}
        A special cancellation occurs along the fast characteristic $\psi$, due to the engineered symmetry of the preshock profile up to the $(r-r_*)^{{\frac{2}{3}}}$ order (Lemma 4.17). This ensures the second derivatives remain bounded and continuous ($C^2$) at $r_{\mathrm{in}}^{(1)}$.
        
        However, the higher-order asymmetry mentioned in Item 2a results in a jump discontinuity in the third derivative at $r_{\mathrm{in}}^{(1)}$. The solution here is precisely $C^2$, but generally fails to be $C^3$.
    \end{enumerate}
\end{enumerate}
    
In summary, the initial data is characterized by $C^\infty$ exterior regions connected by a finitely differentiable interior region ($C^2$ almost everywhere), with localized singularities where the regularity drops precisely to $C^{1,{\frac{1}{3}}}$. This structured lack of global smoothness is necessary to generate the precise shock formation described in the theorem.
\end{remark}

\begin{remark}[\textsl{Cancellation of Weak Characteristic Singularities}]
A crucial subtlety of this construction is that the initial data cannot be generic and smooth. If one were to start with generic $C^2$ initial data at $t=\ti$ that 
evolves to a preshock, then for times $t>\ts$, weak characteristic singularities would necessarily emanate from the preshock point $(r_*, \ts)$. Specifically, 
one would expect a \textit{weak contact discontinuity} to propagate along the particle path characteristic $\phi$ and a \textit{weak rarefaction wave} to 
propagate along the slow acoustic characteristic $\eta$.

The presence of these additional singular waves would render the solution non-smooth on either side of the main shock front, making it structurally 
inconsistent with the target Guderley solution, which consists only of a single, clean shock front collapsing into a quiescent core. Therefore, a key 
feature of the backwards-in-time proof is that it constructs a very specific, non-generic initial data set. The precise $C^{1,{\frac{1}{3}}}$ nature of this data is 
engineered to create a delicate cancellation at the preshock instant, which is necessary to prevent the formation of these unwanted weak characteristic 
singularities and ensure the resulting solution is smooth away from the main shock.
\end{remark}

\begin{remark}[\textsl{Non-uniqueness of the initial data}]
While the forward-in-time evolution from any single constructed initial data set is unique, the initial data itself is not. Our backwards-in-time proof generates 
an entire family of infinitely many valid initial data sets, each of which evolves to match the Guderley solution perfectly at time $\tf$.

This non-uniqueness is a direct consequence of the degrees of freedom available when prescribing the shock front's history for $t < \tf$. The 
Rankine-Hugoniot conditions constrain the relationship between the fluid states and the shock speed, but they do not uniquely determine the shock path 
backwards in time. Our construction leverages this freedom by making specific choices for the shock curve from admissible open sets of functions. 
Specifically, the infinite dimensionality of the set of initial data stems from:
\begin{itemize}
    \item The choice of the shock curve $\s(t)$ on the time interval $[\tsh, \tf]$, which transitions the solution from the strong Guderley shock to a weak shock.
    \item The subsequent choices of a refined shock curve $\l(t)$ and the wavespeed trace $\lambda_1|_{\l}^+$ on the interval $[\ts, \tsh]$, which transition the weak shock into the desired $C^{\frac{1}{3}}$ preshock cusp.
\end{itemize}
Each distinct choice of this shock history generates a different (but still valid) preshock state at time $\ts$, which in turn corresponds to a unique member of the infinite family of initial data sets at time $\ti$.
\end{remark}

\begin{remark}[\textsl{Forward-in-time uniqueness}]
In contrast to the non-uniqueness of the initial data, the forward-in-time evolution from any single data set $(u_{\mathrm{in}}, \rho_{\mathrm{in}}, b_{\mathrm{in}})$ constructed in our proof is unique. This uniqueness is established piece-by-piece across the three distinct phases of the evolution:

\begin{itemize}
    \item \textsl{Classical Flow ($t \in [\ti, \ts]$):} In the interval leading up to the singularity, the solution is classical. Standard well-posedness theories \cite{Friedrichs1948,Kato1975} are not directly applicable here, as the sound speed and pressure both vanish and hence ensure a failure of the strict hyperbolicity required by such existence theorems.  Nevertheless, for the specific case of radial symmetry, local well-posedness in $C^{1, \alpha }$, $\alpha \in (0,1)$, is guaranteed by the method of characteristics.

    \item \textsl{From Preshock to Weak Shock ($t \in (\ts, \tsh]$):} Once the preshock forms at $t=\ts$, uniqueness continues to hold within the specific class of radial regular shock solutions emanating from the preshock. The precise statement of this  uniqueness, which relies on appropriate a priori bounds, is given in Lemma~\ref{lemma:uniqueness:preshock} in Appendix~\ref{app:uniqueness}.

    \item \textsl{Regular Shock Evolution ($t \in (\tsh, 0)$):} For the subsequent evolution, where the solution is a regular strong shock that eventually matches the Guderley profile, uniqueness is guaranteed by the classical theory for multi-dimensional shock fronts, as established by Majda \cite{Ma1983a,Ma1983b}.
\end{itemize}
Thus, a unique evolutionary path is established at each stage, from the classical initial data to the final implosion.
\end{remark}

\begin{remark}[\textsl{Qualitative nature of the initial data}]
It is important to note that our theorem is a constructive existence proof; the initial data at $t=\ti$ is not given by an explicit formula. Its existence is guaranteed by our backwards-in-time construction, which involves evolving the complex  Guderley solution backward through the nonlinear Euler equations. An analytical 
representation of this process is intractable.

Nevertheless, we can describe the essential qualitative features that this initial data must possess. Referring to Figure~\ref{fig:global:1:tikz}, the profiles of the Riemann variables are highly asymmetric around their respective points of non-smoothness 
($r_{\mathrm{in}}^{(1)}, r_{\mathrm{in}}^{(2)}, r_{\mathrm{in}}^{(3)}$). For the dominant variable $z$, the profile must be significantly steeper on the left side (smaller radii) of $r_{\mathrm{in}}^{(1)}$. This carefully engineered gradient is precisely what is required to ensure that the characteristics originating from a 
wide spatial interval at $t=\ti$ will all focus and converge at the single preshock point $(r_*, \ts)$ at the later time. A similar asymmetric steepness is required 
for the $w$ and $b$ profiles to achieve this focusing.

A crucial aspect of this construction is the avoidance of additional, spurious shocks. One might expect that the large initial gradients required by the 
construction would themselves lead to new, independent gradient catastrophes. This is prevented by the global dynamics of the solution. The primary shock 
front, which travels inward along the path $r=\s(t)$, moves with sufficient speed to overtake these steep regions, effectively consuming them before they have 
the time to form new singularities. The dynamics are entirely dominated by the single, coherent implosion front that evolves into the Guderley shock.\end{remark}

\begin{remark}[\textsl{Construction within the Guderley Acoustic Cone}]
It is an important feature of our result that the entire constructed evolution, from the initial data at $t=\ti$ to the matching time $t=\tf$, takes place at radii 
smaller than the trajectory of the Guderley sonic point. This trajectory can be understood as the boundary of the acoustic cone emanating backward-in-time 
from the final implosion event at $(r=0, t=0)$.

Therefore, our initial data, the formation of the preshock, and the entire shock development process are constructed to exist strictly within the region of 
spacetime that is in causal contact with the final singularity. This avoids any complexities of the flow at larger radii (beyond the sonic point) and shows that the 
mechanism for generating the Guderley shock from smooth data can be contained entirely within its core acoustic domain.
\end{remark}

\begin{remark}[\textsl{Regularization from preshock to initial data}]
At $t=\ts$ the dominant Riemann variable $z$ has a $C^{\frac{1}{3}}$ cusp at $r=r_*$; see the expansion \eqref{taylor:rs}. The improvement to $C^{1,{\frac{1}{3}}}$ when evolving backward to $t=\ti$ is \emph{not} due to any dissipative mechanism. It is purely kinematic, coming from a degenerate but monotone reparametrization of space, tailored to the fast--acoustic characteristics.

The construction fixes the fast--acoustic ``Lagrangian'' chart $(x,s)$ by solving
\[
\partial_s \psi(x,s)=\lambda_1\!\circ\!\psi(x,s)\,, \qquad 
\psi(x,\ts)=r_*+\tfrac{x^3}{3} \,;
\]
see \eqref{choice:psi:smooth}. The terminal condition forces the Jacobian $J=\partial_x\psi(\cdot,\ts)=x^2$ to vanish quadratically at $x=0$, matching the preshock geometry. Composing the data with $\psi$ removes the cusp and we have that
\[
Z(x,\ts):=z(\ts,\psi(x,\ts))=z(\ts,r_*)+\mathsf{a}\,x+\mathsf{b}\,x^2+O(x^3),
\]
and similarly $W(\cdot,\ts)$ and $S(\cdot,\ts)$ become smooth in $x$ (substitute $r-r_*={\frac{x^3}{3}} $ in \eqref{taylor:rs}).\footnote{Here, we are using the convention that the
capitalized version of a  symbol (or letter) represents the composite function,  with the composition taken with the fast acoustic flow map $\psi$ as defined below in \eqref{ring:def}.}
Thus the ``bad” factor $(r-r_*)^{-\frac{2}{3}}$ in $\partial_r z$ is exactly canceled by $J=x^2$ via the chain rule.

In these variables (see \eqref{ring:def}) the composition with the fast acoustic flow $\psi$ produces a first–order system with weights $J$ and $\Sigma$ (e.g. \eqref{euler:bp:lag}, \eqref{eq:J:sl}–\eqref{eq:JrZ:sl}). Derivatives propagate along the $x$–trajectories $\Upsilon_3,\Upsilon_2$ defined in \eqref{ups:1}–\eqref{ups:2}, and the bootstrap bounds \eqref{bt:s}–\eqref{bt:e} give
\[
|\partial_s \mathring{W}(x,s)| \lesssim (t-\Upsilon_3(x))^{-\frac{2}{3}}, 
\qquad
|\partial_s  \mathring{S}(x,s)| \lesssim (t-\Upsilon_2(x))^{-\frac{2}{3}}.
\]
The corresponding ``weights” are integrable with a $\frac{1}{3}$-derivative gain (see \eqref{bounds:weight}–\eqref{bounds:weight:3}), so integrating in $s$ yields $\partial_x(W,Z , \mathsf{S})(\cdot,t)\in C^{0,{\frac{1}{3}}}_x$ for every $t<\ts$. Since $J(\cdot,t)=\partial_x\psi(\cdot,t)$ is strictly positive and $C^1$ for $t<\ts$, the map $x\mapsto r=\psi(x,t)$ is a $C^2$ diffeomorphism. Pulling back to the physical coordinate $r$ then gives
\[
(w,z,b)(\cdot,t)\in C^{1,{\frac{1}{3}}}_r \quad \text{for every } t\in[\ti,\ts),
\]
which is the regularity asserted in Proposition~\ref{prop:pre2smooth}.

In short,  the cubic terminal labeling $\psi(\cdot,\ts)=r_*+\tfrac{x^3}{3}$ \emph{flattens} the preshock cusp; the characteristic system carries this smoothness backward; and invertibility of $\psi(\cdot,t)$ for $t<\ts$ transfers it back to $r$, producing $C^{1,{\frac{1}{3}}}$ initial data.
\end{remark}

\section{Strategy of the proof}
\label{sec:strategy}

Our main theorem asserts the existence of a specific, non-generic set of initial data at an initial time $t=\ti$,  which generates a unique solution to the Euler equations that undergoes shock formation and development, culminating in an exact match with the Guderley self-similar imploding shock at a later time $t=\tf$.

To prove this forward-in-time existence claim, our strategy is to reverse the problem. Instead of attempting to evolve an unknown initial state forward into a known final state, we begin with the well-understood Guderley solution at the terminal time $t=\tf$ and solve the Euler equations backwards in time. The objective of this backwards construction is to explicitly identify a valid preshock state at an intermediate time $t=\ts$ and, ultimately, the required $C^{1,{\frac{1}{3}}}$ initial data at $t=\ti$. The proof is organized into the following three main steps.

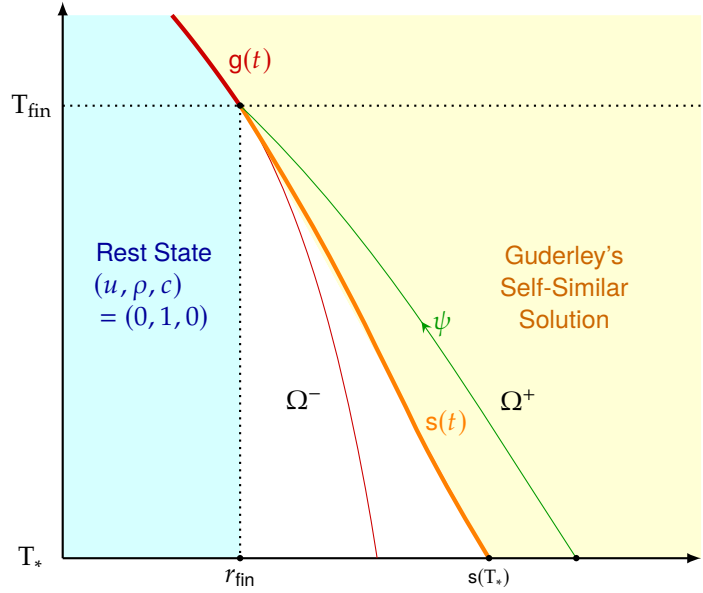
\begin{figure}[htb!]
\centering
\begin{tikzpicture}[
 	font=\sffamily\small,
 	scale=1.2,
 	>=latex
]
 	
 	\coordinate (Q) at ({6/pow(10,0.7) * pow(10-3,0.7)}, 3);
 	
 	\coordinate (R) at ({6/pow(10,0.7) * pow(10-8,0.7)}, 8);
 	\coordinate (rright) at ({6/pow(10,0.7) * pow(10-8,0.7)}, 3);
 	\coordinate (rfake) at ($(rright)!0.55!(Q)$);
 	\coordinate (rleft) at ($(Q)!-0.35!(rright)$);
 	\coordinate (a_point) at 	({6/pow(10,0.7) * pow(10-6,0.7)-0.055}, 6);
 	\coordinate (eta_3_anchor) at ($(rright)!0.25!(Q)$);
 	\coordinate (eta_2_anchor) at ($(rright)!0.5!(Q)$);
 	\coordinate (eta_1_anchor) at ($(rright)!0.75!(Q)$);

 	\fill[cyan!20, opacity=0.8] (R |- 0,3) -- (R) -- plot[domain=8:9, samples=50, variable=\t] ({6/((10)^0.7) * (10-\t)^0.7}, {\t}) -- (0,9) -- (0,3) -- cycle;

 	\begin{scope}

 	 	\clip (7,3)-- (Q) -- plot[smooth, tension = 1] coordinates {(R) (2.55,7) (3.1,6) (3.6, 5) (4.1,4) (Q)} -- plot[domain=8:9, samples=50, variable=\t] ({6/((10)^0.7) * (10-\t)^0.7}, {\t}) -- ({6/((10)^0.7) * (10-9)^0.7}, 9) -- (7, 9) -- cycle;
 	 	
 	 	\fill[yellow!20, opacity=0.8] (0,1) rectangle (8,10);
 	\end{scope}
 
%
 	\draw[->, thick] (0, 3) -- (7, 3);
 	\draw[->, thick] (0, 3) -- (0, 9.15);
%
 	\draw [red!80!black, thin] plot [smooth, tension = 1] coordinates {(R) (2.8, 6) (rfake)};
 	\draw[orange, ultra thick] plot[smooth, tension = 1] coordinates {(R) (2.55,7) (3.1,6) (3.6, 5) (4.1,4) (Q)};
 	\draw[red!80!black, ultra thick] plot[domain=8:9, samples=50, variable=\t] ({6/((10)^0.7) * (10-\t)^0.7}, {\t});
%
 	\coordinate (g_label_pos) at ({6/pow(10,0.7) * pow(10-8.5,0.7)},8.5);
 	\node[anchor=west, xshift=4pt, text=red!80!black] at (g_label_pos) {$\g(t)$};

 	\coordinate (s_label_pos) at ({6/pow(10,0.7) * pow(10-4.5,0.7)-0.2},4.5);
 	\node[anchor=west, xshift=4pt, text=orange] at (s_label_pos) {$\s(t)$};

 	\node[blue!60!black, align=center] at (1, 6) {Rest State \\ $(u, \rho, c) \hspace{1.25em}$ \\ $= (0,1,0)$};
 	\node[orange!80!black, align=center] at (5.5, 6) {Guderley's \\ Self-Similar \\ Solution};

 	\draw[dotted, thick] (0, 8) -- (7, 8);
 	\node[anchor=east] at (0, 8) {$\tf$};	
 	\node[anchor=east] at (-0.1, 3) {$\ts$};
 
 	\path[draw,green!60!black,%
 	 	 	decoration={%
 	 	 	 	markings,%
 	 	 	 	mark=at position 0.5 	with {\arrow[scale=1.25]{stealth}},%
 	 	 	},%
 	 	 	postaction=decorate] (rleft) .. controls (4.5,4.75) and (3.5,6.5) .. (R) node[pos=0.5, right, green!60!black] {$\psi$}; 	 	
 	
 	\draw[dotted, thick] (R) -- (R |- 0,3) node[below] {$r_{\mathsf{fin}}$};
 	\node at (5,4.75) {$\Omega^+$};
 	\node at (2.65,4.75) {$\Omega^-$};
 	\fill (Q) circle (1pt) node[below] {$\scriptstyle \s(\ts)$};
 	\fill (R) circle (1pt);
 	\fill (rright) circle (1pt);
 	\fill (rleft) circle (1pt);

\end{tikzpicture}
\caption{{\footnotesize In thick red, we have displayed the classical Guderley shock curve~$\g$ for $t > \tf$. In thin red, we have continued the Guderley shock curve for $t \in [\ts,\tf]$, in order to emphasize the relative placement of the new, \emph{chosen}, shock curve~$\s$, which we have drawn in orange. The green $\psi$ characteristic is the backwards-in-time fast-acoustic characteristic emanating from $(\g(\tf)^+,\tf) = (\s(\tf)^+,\tf)$; the intersection of this characteristic curve with $\{t = \ts\}$ imposes an upper bound for the choice of $\s(\ts)$. In this step, we solve the Euler equations in the white shaded region $\Omega^-$, by tracing back characteristics backwards-in-time to $\{t = \ts\}$.}}
\label{fig:step:1:outline}
\end{figure}

\subsection{Step 1: From the Guderley Shock to a Weak Preshock}

The construction begins at the chosen final time $t=\tf$, where our solution is defined to be identical to the classical Guderley solution. The first step of the proof is to solve the Euler equations backwards in time on an interval $[\ts, \tf]$ to find a state at $t=\ts$ that corresponds to a preshock. The key is to prescribe a shock trajectory, $r=\s(t)$, that connects the known Guderley shock at $t=\tf$ to this yet-to-be-determined preshock state.

This prescribed shock curve $\s(t)$ is not arbitrary; it is carefully engineered to satisfy several critical properties. First, it must seamlessly match the Guderley shock curve $\g(t)$ at the handover time $t=\tf$ to a high order of smoothness. Second, for the entire interval, it must satisfy the \textsl{Lax geometric entropy conditions}~\eqref{lax}. These inequalities ensure the shock is compressive: the fast family converges on the shock front from both sides, while the remaining families are outrun by the shock ahead of it. For this shock, which moves with a speed near $\lambda_1$, this specifically means that $\lambda_1|_{\s}^+ < \ds(t) < \lambda_1|_{\s}^-$, while the other characteristic speeds remain slower than the shock, i.e. $\dot\s(t) < \lambda_j^\pm(t)$ for $j=2,3$.

Most importantly, the curve is designed such that the shock strength gradually vanishes as time evolves backward to $\ts$. This is achieved by forcing the shock speed to approach the fast acoustic wavespeed from the right, satisfying the preshock condition $\ds(\ts) = \lambda_1|_{\s}^+(\ts)$. This controlled approach dictates the rate at which the jumps across the shock decay, with $\jump{z} \sim (\ds - \lambda_1|_{\s}^+)$ as $t \searrow \ts$, where the tilde symbol ($\sim$) indicates asymptotic proportionality (i.e., the ratio of the two quantities tends to a non-zero constant).

With the shock path $\s(t)$ fixed, in the spacetime $\Omega^+=\left\{ (r,t) \in \R_+ \times (\ts, \tf] : r> \s(t)\right\}$, we let the Euler solution exactly equal the Guderley one. We then use the Rankine-Hugoniot jump conditions and the Guderley traces $(w,z,b)|_\s^+$, to obtain values for all unknowns on the ``left'' side of $\s$, namely $(w,z,b)|_\s^-$. This yields boundary conditions for the Euler equations on the future temporal boundary of the spacetime $\Omega^-=\left\{ (r,t) \in \R_+ \times (\ts, \tf] : r < \s(t)\right\}$. In the region~$\Omega^-$, we \emph{uniquely} solve the Euler equations by tracing all wave families backwards in time, up to the time slice $\{t = \ts\}$. See Figure~\ref{fig:step:1:outline}.

\subsection{Step 2: Crafting a Well-Posed Preshock} \label{strategy:3}

The construction in Step 1 successfully reduces a strong shock to a continuous state at $t=\ts$, but it comes at a cost. The required shock curve $\s(t)$ must have a singular acceleration, with $\ddot{\s}(t) \to +\infty$ as $t \searrow \ts$. This creates an ``imbalanced preshock" with mismatched regularity: the solution is $C^\infty$ on the exterior of the preshock sphere (matching the Guderley data) but only Hölder continuous on the interior. Such an asymmetric singularity cannot be generated by initial data with $C^k$ regularity for $k\ge 1$.

To resolve this technical issue, we refine the construction on a short time interval $[\ts, \tsh)$, where $\tsh$ is chosen close enough to $\ts$ that the shock from Step 1 is already very weak. We discard the singular curve $\s(t)$ on this interval and instead construct a new, highly regular shock curve, $r=\l(t)$, which is specifically engineered to produce a physically realistic preshock. The goal is to create a symmetric $C^{\frac{1}{3}}$ cusp in the dominant Riemann variable $z$. This is achieved by prescribing the precise rate at which the shock speed approaches the characteristic speed, ensuring that $\dot{\l}(t) - \lambda_1|_{\l}^+(t) \sim (t-\ts)^{\frac{1}{2}}$.

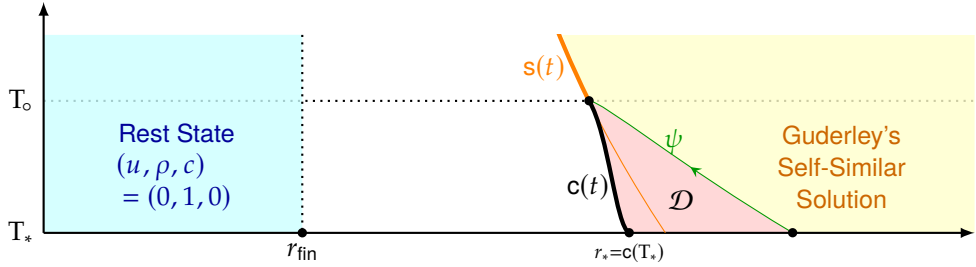
\begin{figure}[htb!]
\centering

\begin{tikzpicture}[
 	font=\sffamily\small,
 	scale=1.75,
 	>=latex
]
 	\coordinate (Q) at ({6/pow(10,0.7) * pow(10-3,0.7)}, 3);
 	\coordinate (R) at ({6/pow(10,0.7) * pow(10-5,0.7)}, 5);
 	\coordinate (R_new) at (4.1, 4);
 	\coordinate (rright) at ({6/pow(10,0.7) * pow(10-8,0.7)}, 3);
 	\coordinate (rright_top) at ({6/pow(10,0.7) * pow(10-8,0.7)}, 4.5);
 	\coordinate (rleft) at ($(Q)!-0.35!(rright)$);
 	\coordinate (n_anchor) at ($(rright)!0.9!(Q)$);
 	\coordinate (n_left) at ($(rright)!0.6!(Q)$);

 	\draw[dotted, thick] (0, 4) -- (7, 4);
 	\node[anchor=east] at (0, 4) {$\tsh$};

 	\fill[cyan!20, opacity=0.8] (rright) -- (rright_top) -- (0,4.5) -- (0,3) -- cycle;

 	\begin{scope}
 	 	\clip (7,3)-- (Q) -- plot (rleft) .. controls (4.8,3.5) and (4.2,4) .. (R_new) -- plot[smooth, tension = 1] coordinates {(R) (4.1,4) (R_new)} -- (R) -- (7, 4.5) -- cycle;
 	 	\fill[yellow!20, opacity=0.8] (rright) rectangle (7,4.5);
 	\end{scope}
 	
 	 	\begin{scope}
 	 	\clip (R_new) -- (R_new).. controls (4.25,3.75) and (4.3,3.05) .. (n_anchor) -- (n_anchor) -- (rleft) -- (rleft) .. controls (4.8,3.5) and (4.2,4) .. (R_new) -- cycle;
 	 	\fill[pink, opacity=0.6] (rright) rectangle (6,4);
 	\end{scope}

 	\draw[->, thick] (0, 3) -- (7, 3);
 	\draw[->, thick] (0, 3) -- (0, 4.75);

 \begin{scope}
 	\clip (0,4) rectangle (7,4.5);
 	\draw[orange, ultra thick] plot[smooth, tension = 1] coordinates {(R) (4.1,4) (Q)};
\end{scope}

 \begin{scope}
 	\clip (0,3) rectangle (7,4);
 	\draw[orange, thin] plot[smooth, tension = 1] coordinates {(R) (4.1,4) (Q)};
\end{scope}

 	\coordinate (n_label_pos) at (4.3,3.35);
 	\node[anchor=west, xshift=-22pt, text=black] at (n_label_pos) {$\l(t)$};
 	\draw[black, ultra thick] (R_new).. controls (4.25,3.75) and (4.3,3.05) .. (n_anchor);
 	\node[below, text=black] at (n_anchor) {$\scriptstyle r_* = \l(\ts)$};
 
 	\coordinate (s_label_pos) at ({6/pow(10,0.7) * pow(10-4.6,0.7)-0.45},4.25);
 	\node[anchor=west, xshift=4pt, text=orange] at (s_label_pos) {$\s(t)$};
 	
 	\node[blue!60!black, align=center] at (1, 3.5) {Rest State \\ $(u, \rho, c) \hspace{1.25em}$ \\ $= (0,1,0)$};
 	\node[orange!80!black, align=center] at (6, 3.5) {Guderley's \\ Self-Similar \\ Solution};
 	
 	\node[anchor=east] at (0, 3) {$\ts$};
 
 	\path[draw,green!60!black,%
 	 	 	decoration={%
 	 	 	 	markings,%
 	 	 	 	mark=at position 0.5 	with {\arrow[scale=1.25]{stealth}},%
 	 	 	},%
 	 	 	postaction=decorate] (rleft) .. controls (4.8,3.5) and (4.2,4) .. (R_new) node[pos=0.5, right, green!60!black] {$\psi$};

 	\draw[dotted, thick] ({6/pow(10,0.7) * pow(10-8,0.7)},4.5) -- ({6/pow(10,0.7) * pow(10-8,0.7)},3) node[below] {$r_{\mathsf{fin}}$};
 	\node at (4.8,3.25) {${\mathcal D}$};
 	\fill (R_new) circle (1pt);
 	\fill (rright) circle (1pt);
 	\fill (rleft) circle (1pt);

 	\fill (n_anchor) circle (1pt);
\end{tikzpicture}
\caption{{\footnotesize We have drawn in black the new shock curve~$\l$, while in orange we have displayed the old shock curve~$\s$. In this step we solve the equations in the light-pink shadowed region~$\mathcal{D}$. The boundary value for~$\lambda_1$ is given along the black~$\l$ curve, while the data for~$w, b$ is given along the green~$\psi$ characteristic in green.}}
\label{fig:step:2:a:outline}
\end{figure}

This setup defines a well-posed, backwards-in-time Goursat problem.\footnote{ In this context, well-posedness means that a solution is guaranteed to exist, is unique, and depends continuously on the prescribed boundary data—a small perturbation of the data results in a correspondingly small perturbation of the solution. 
}
 As shown in Figure~\ref{fig:step:2:a:outline}, boundary data is prescribed on two intersecting curves:
\begin{itemize}
    \item On the newly chosen shock front $r=\l(t)$, we prescribe the trace of the fast acoustic wavespeed, $\lambda_1|_{\l}^+(t)$.
    \item On the fast-acoustic characteristic $\psi$ emanating backward-in-time from the handover point $(\s(\tsh)^+, \tsh)$, we inherit the data for the subdominant variables, $w$ and $b$, from the solution constructed in Step 1.
\end{itemize}
This mixed boundary data is sufficient to uniquely determine the solution in the spacetime region $\mathcal{D}$ between the characteristic and the new shock front. Subsequently, the Rankine-Hugoniot conditions along $\l(t)$ provide the necessary data to solve for the flow on the interior side of the shock, in region $\mathcal{L}$ (see Figure~\ref{fig:step:2:b:outline}).
\begin{figure}[htb!]
\centering
\begin{tikzpicture}[
 	font=\sffamily\small,
 	scale=1.75,
 	>=latex
]
 	\coordinate (Q) at ({6/pow(10,0.7) * pow(10-3,0.7)}, 3);
 	\coordinate (R) at ({6/pow(10,0.7) * pow(10-5,0.7)}, 5);
 	\coordinate (R_new) at (4.1, 4);
 	\coordinate (rright) at ({6/pow(10,0.7) * pow(10-8,0.7)}, 3);
 	\coordinate (rright_top) at ({6/pow(10,0.7) * pow(10-8,0.7)}, 4.5);
 	\coordinate (rleft) at ($(Q)!-0.35!(rright)$);
 	\coordinate (n_anchor) at ($(rright)!0.9!(Q)$);
 	\coordinate (n_left) at ($(rright)!0.6!(Q)$);

 	\draw[dotted, thick] (0, 4) -- (7, 4);
 	\node[anchor=east] at (0, 4) {$\tsh$};

 	\fill[cyan!20, opacity=0.8] (rright) -- (rright_top) -- (0,4.5) -- (0,3) -- cycle;

 	\begin{scope}
 	 	\clip (7,3)-- (Q) -- plot (rleft) .. controls (4.8,3.5) and (4.2,4) .. (R_new) -- plot[smooth, tension = 1] coordinates {(R) (4.1,4) (R_new)} -- (R) -- (7, 4.5) -- cycle;
 	 	\fill[yellow!20, opacity=0.8] (rright) rectangle (7,4.5);
 	\end{scope}
 	
 	 	\begin{scope}
 	 	\clip (R_new) -- (R_new).. controls (4.25,3.75) and (4.3,3.05) .. (n_anchor) -- (n_anchor) -- (rleft) -- (rleft) .. controls (4.8,3.5) and (4.2,4) .. (R_new) -- cycle;
 	 	
 	 	\fill[pink, opacity=0.6] (rright) rectangle (6,4);
 	\end{scope}

 	\begin{scope}
 	 	\clip (R_new) -- (n_left) .. controls (3.8,3.5) and (4.2,4.2) .. (R_new) -- (n_left) -- (n_anchor) -- (R_new).. controls (4.25,3.75) and (4.3,3.05) .. (n_anchor) -- cycle;
 	 	\fill[brown, opacity=0.25] (rright) rectangle (6,4);
 	\end{scope}
 	
 	 	\begin{scope}
 	 	\clip ({6/pow(10,0.7) * pow(10-8,0.7)},3) -- ({6/pow(10,0.7) * pow(10-8,0.7)},4.5) -- ({6/pow(10,0.7) * pow(10-8,0.7)},4.5) -- (R_new) -- (R_new) .. controls (4.2,4.2) and (3.8,3.5) .. (n_left) -- (n_left) -- ({6/pow(10,0.7) * pow(10-8,0.7)},3) -- ({6/pow(10,0.7) * pow(10-8,0.7)},3) -- cycle;
 	 	\fill[green,opacity=0.15] (1,3) rectangle (6,4);
 	\end{scope}
 
%
 	\draw[->, thick] (0, 3) -- (7, 3);
 	\draw[->, thick] (0, 3) -- (0, 4.75);

 \begin{scope}
 	\clip (0,4) rectangle (7,4.5);
 	\draw[orange, ultra thick] plot[smooth, tension = 1] coordinates {(R) (4.1,4) (Q)};
\end{scope}

 	\coordinate (n_label_pos) at (4.65,3.55);
 	\node[anchor=west, xshift=-22pt, text=black] at (n_label_pos) {$\l(t)$};
 	\draw[black, ultra thick] (R_new).. controls (4.25,3.75) and (4.3,3.05) .. (n_anchor);
 	\node[below, text=black] at (n_anchor) {$\scriptstyle r_* = \l(\ts)$};
 
 	\coordinate (s_label_pos) at ({6/pow(10,0.7) * pow(10-4.6,0.7)-0.45},4.25);
 	\node[anchor=west, xshift=4pt, text=orange] at (s_label_pos) {$\s(t)$};
 	
 	\node[green!60!black, align=center] at (3, 3.5) {Previously\\ solved \\ in Step $1$};
 	\node[blue!60!black, align=center] at (1, 3.5) {Rest State \\ $(u, \rho, c) \hspace{1.25em}$ \\ $= (0,1,0)$};
 	\node[orange!80!black, align=center] at (6, 3.5) {Guderley's \\ Self-Similar \\ Solution};
 	
 	\node[anchor=east] at (0, 3) {$\ts$};
 
 	\path[draw,green!60!black,%
 	 	 	decoration={%
 	 	 	 	markings,%
 	 	 	 	mark=at position 0.5 	with {\arrow[scale=1.25]{stealth}},%
 	 	 	},%
 	 	 	postaction=decorate] (rleft) .. controls (4.8,3.5) and (4.2,4) .. (R_new) node[pos=0.5, right, green!60!black] {$\psi$}; 	 	
 	 	 	
 	\path[draw,blue!60!black,%
 	 	 	decoration={%
 	 	 	 	markings,%
 	 	 	 	mark=at position 0.5 	with {\arrow[scale=1.25]{stealth}},%
 	 	 	},%
 	 	 	postaction=decorate] (n_left) .. controls (3.8,3.5) and (4.2,4.2) .. (R_new) node[pos=0.35, left , blue!60!black] {$\eta$};

 	\draw[dotted, thick] ({6/pow(10,0.7) * pow(10-8,0.7)},4.5) -- ({6/pow(10,0.7) * pow(10-8,0.7)},3) node[below] {$r_{\mathsf{fin}}$};
 	\node at (4.8,3.25) {${\mathcal D}$};
 	\node at (4,3.25) {${\mathcal L}$};
 	\fill (R_new) circle (1pt);
 	\fill (rright) circle (1pt);
 	\fill (rleft) circle (1pt);
 	\fill (n_left) circle (1pt);
 	\fill (n_anchor) circle (1pt);
\end{tikzpicture}

\caption{{\footnotesize In this step we solve the equations in the light-brown shadowed region~$\mathcal{L}$. To do so, we use that along the non-characteristic new shock curve $\l$ the Rankine-Hugoniot jump conditions provide sufficient boundary conditions for Euler solvability within $\mathcal{L}$. In the green region to the left of region $\mathcal{L}$ the solution was previously obtained in Step $1$.}}
\label{fig:step:2:b:outline}
\end{figure}
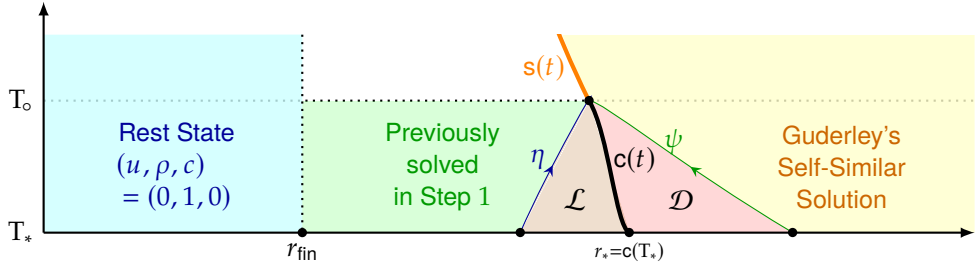

The outcome of this step is a regular shock solution on the interval $[\ts,\tsh)$ that connects the weak shock state at $t=\tsh$ to a genuine $C^{\frac{1}{3}}$-preshock at the initial time $t=\ts$. At this preshock instant, occurring at the location $r_* := \l(\ts)$, the solution exhibits a clear separation of regularity between the dominant and subdominant variables:
\begin{itemize}
    \item The dominant Riemann variable $z$ forms the characteristic cusp singularity, with the fractional series expansion
    $$
    z(r, \ts) = z(r_*, \ts) + \mathsf{a}(r-r_*)^{\frac{1}{3}} + \mathsf{b}(r-r_*)^{\frac{2}{3}} + O(|r-r_*|).
    $$
    \item In contrast, the subdominant variables $w$ and $b$ are smoother, remaining Lipschitz continuous up to the preshock point. Their higher-order behavior is given by
    \begin{align*}
    w(r, \ts) &= w(r_*, \ts) + \mathsf{c}_w (r-r_*) + O( |r-r_*|^{\frac{4}{3}}), \\
    b(r, \ts) &= b(r_*, \ts) + \mathsf{c}_b (r-r_*) + O( |r-r_*|^{\frac{4}{3}}).
    \end{align*}
\end{itemize}
The constants in these expansions are determined by the details of the construction, as established in Lemma~\ref{lemma:expansion:0}.

\subsection{Step 3: From Preshock to Smooth Initial Data}

This final step of the construction completes the backwards-in-time journey by evolving the preshock state to find the required initial data. The well-posed $C^{\frac{1}{3}}$-preshock data, constructed at time $t=\ts$ in the previous step, now serves as the Cauchy data for a backwards-in-time initial value problem for the Euler equations.

Unlike the preceding steps, which involved choosing shock curves from admissible sets, this stage of the proof is deterministic. With the data at $\ts$ fully specified, the evolution backwards to an initial time $\ti$ (for $\ts-\ti$ sufficiently small) is unique, leaving no remaining degrees of freedom in the construction.

The main result of this step is the demonstration of a ``backwards regularization effect," which produces a $C^{1,{\frac{1}{3}}}$ solution at the initial time $t=\ti$. More precisely, the non-smoothness in this initial data is confined to three distinct spatial locations. As illustrated in Figure~\ref{fig:step:3:outline}, these points correspond to the intersection of $\{t=\ti\}$ with the slow-acoustic, entropy, and fast-acoustic characteristics emanating backwards-in-time from the single preshock point $(r_*,\ts)$.

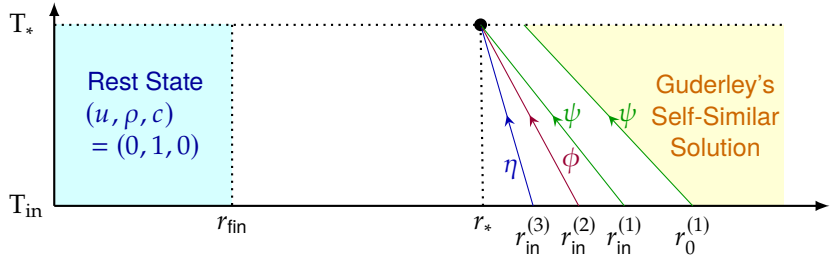
\begin{figure}[htb!]
\centering

\begin{tikzpicture}[
 	font=\sffamily\small,
 	scale=1.2,
 	>=latex
]
 	\coordinate (P) at ({6/pow(10,0.7) * pow(10-4,0.7)}, 4);
 	\coordinate (Q) at ({6/pow(10,0.7) * pow(10-3,0.7)}, 3);
 	\coordinate (R) at ({6/pow(10,0.7) * pow(10-8,0.7)}, 8);
 	
 	\coordinate (rin3) at (5.25, 1); 	
 	\coordinate (rstar) at (4.7, 1); 
 	\coordinate (rin2) at (5.75, 1); 	
 	\coordinate (rin1) at (6.25, 1); 	
 	\coordinate (r01) at (7.0, 1);

 	\fill[cyan!20, opacity=0.8] (0,1) -- (0,3) -- (1.95,3) -- (1.95,1) -- (0,1) -- cycle;

 	\begin{scope}
 	 	\clip (r01) -- (5.15,3) -- (8.5, 3) -- (8.5, 1) -- (r01) -- cycle;
 	 	\fill[yellow!20, opacity=0.8] (0,1) rectangle (8,10);
 	\end{scope}

 	
 	\draw[->, thick] (0, 1) -- (8.5, 1);
 	\draw[->, thick] (0, 1) -- (0, 3.25);

 	
 	\node[blue!60!black, align=center] at (1, 2) {Rest State \\ $(u, \rho, c) \hspace{1.25em}$ \\ $= (0,1,0)$};
 	\node[orange!80!black, align=center] at (7.25, 2) {Guderley's \\ Self-Similar \\ Solution};


 	\draw[dotted, thick] (0, 3) -- (8, 3);
 	\node[anchor=east] at (-0.1, 3) {$\ts$};
 	
 	\node[anchor=east] at (0, 1) {$\ti$};
 
 	\draw[dotted, thick] (1.95,3) -- (1.95,1) node[below] {$r_{\mathsf{fin}}$};
 	
 	\fill (Q) circle (2pt);

 	\path[draw,blue!70!black,%
 	 	 	decoration={%
 	 	 	 	markings,%
 	 	 	 	mark=at position 0.5 	with {\arrow[scale=1.25]{stealth}},%
 	 	 	},%
 	 	 	postaction=decorate] (rin3) -- (Q) node[pos=0.2, left=-0.2em, blue!70!black] {$\eta$};
 	\path[draw,purple!80!black,%
 	 	 	decoration={%
 	 	 	 	markings,%
 	 	 	 	mark=at position 0.5 	with {\arrow[scale=1.25]{stealth}},%
 	 	 	},%
 	 	 	postaction=decorate] (rin2) -- (Q) node[pos=0.25, right=-0.1em, purple!80!black] {$\phi$};
 	\path[draw,green!60!black,%
 	 	 	decoration={%
 	 	 	 	markings,%
 	 	 	 	mark=at position 0.5 	with {\arrow[scale=1.25]{stealth}},%
 	 	 	},%
 	 	 	postaction=decorate] (rin1) -- (Q) node[pos=0.5, right, green!60!black] {$\psi$}; 	 	
 	\path[draw,green!60!black,%
 	 	 	decoration={%
 	 	 	 	markings,%
 	 	 	 	mark=at position 0.5 	with {\arrow[scale=1.25]{stealth}},%
 	 	 	},%
 	 	 	postaction=decorate] (r01) -- (5.15,3) node[pos=0.5, right, green!60!black] {$\psi$}; 	 	 	
 	\draw[dotted, black, thick] (rstar) -- (Q) node[pos=0.5, above, black] {$ $};

 	\node[below] at (rin3) {$r_{\mathsf{in}}^{(3)}$};
 	\node[below] at (rstar) {$r_*$};
 	\node[below] at (rin2) {$r_{\mathsf{in}}^{(2)}$};
 	\node[below] at (rin1) {$r_{\mathsf{in}}^{(1)}$};
 	\node[below] at (r01) {$r_0^{(1)}$};
 
\end{tikzpicture}
\caption{{\footnotesize In Step 3 we solve the Euler system backwards-in-time on $[\ti,\ts)$ from Cauchy data given by the $C^{\frac{1}{3}}$-preshock at time $\{t = \ts\}$, which was obtained at the end of Step 2.}}
\label{fig:step:3:outline}
\end{figure}

 \begin{remark}[\textsl{A tunable family of preshock cusps}]
Our theorem is stated for a $C^{\frac{1}{3}}$ preshock, but the mechanism yields a one–parameter family of cusps. The exponent is set by how fast the shock speed approaches the $1$–characteristic as $t\uparrow\ts$. If
\[
\dot{\l}(t)-\lambda_1(\s(t)^+,t)\sim (\ts-t)^{\frac{\beta}{1-\beta}} \qquad \text{for some }\beta\in(0,1),
\]
then the dominant Riemann variable satisfies
\[
z(r, \ts)=z(r_*, \ts)+\mathsf a\,(r-r_*)^\beta+\mathsf b\,(r-r_*)^{2\beta}+O(|r-r_*|),
\]
so $z(\cdot,\ts)\in C^\beta$ and $\partial_r z\sim |r-r_*|^{\beta-1}$. The arguments of Section~\ref{sec:weak2pre} are robust for any such $\beta$, and the backward regularization in Section~\ref{sec:pre2smooth} requires only adjusting the terminal labeling of the fast–acoustic flow to
\[
\psi(x,\ts)=r_*+x\,|x|^{\frac{1}{1-\beta}},
\]
which enforces $r-r_*\propto |x|^{\frac{1}{\beta}}$ so that $(r-r_*)^\beta$ becomes $|x|$ in the $x$–coordinate, removing the blow‑up in $z_r$ by the chain rule. Evolving backward in this chart yields
\[
(w,z,b)(\cdot,t)\in C^{1,\beta}_r \quad \text{for every } t\in[\ti,\ts).
\]
The case $\beta=\tfrac13$ corresponds to the choice $\psi(x,\ts)=r_*+\tfrac{x^3}{3}$ and produces specially crafted $C^{\frac{1}{3}}$ preshock, matching
the $C^{\frac{1}{3}}$ preshock which is produced by \textit{stable, smooth} and \textit{generic} initial data.
\end{remark}

\section{From the Guderley Shock to a Weak Preshock}
\label{sec:large2weak}

The main result of this section is the following.

\begin{prop}[\textsl{From a strong Guderley shock to a weak preshock}]
\label{prop:large2weak}
We take the Guderley self–similar solution as the target final state and fix any final matching time $\tf<0$.
Then there exist a preshock time $\ts<\tf$ with $|\ts-\tf|\ll|\tf|$, an exponent $\varepsilon>0$,
and a shock trajectory $\s:[\ts,\tf]\to\R_+$ such that there is a radial solution $(w,z,b)$ to \eqref{euler:rv} (equivalently, $(u,\rho,E)$ to \eqref{euler:md}) on $[\ts,\tf]$ with the following properties:
\begin{enumerate}
\item \textsl{Regular shock and quiescent core.} \label{prop:large2weak:1}
We set $\mathcal{S}=\{(x,t)\in\R^d\times[\ts,\tf]:\,|x|=\s(t)\}$. Then $(u,\rho,E,\mathcal{S})$ is a regular shock solution emanating from a preshock (Definition~\ref{def:regular:shock:pre}), is $C^\infty$ on $\R^d\times[\ts,\tf]\setminus\mathcal{S}$, and its traces $(\sm{u},\sm{\rho},\sm{E})$, $(\sp{u},\sp{\rho},\sp{E})$ satisfy the Lax entropy conditions~\eqref{lax}. Moreover,
\[
\mathcal{H}=\{(r,t):\,0<r<(-\tf)^{\frac{1}{\lambda}},\ t\in[\ts,\tf]\}
\]
is a quiescent core: $(u,\rho,E)=(0,1,0)$ on $\mathcal{H}$.

\item \textsl{Exact exterior matching.} \label{prop:large2weak:2}
For every $t\in(\ts,\tf]$, in the exterior region
\[
\Omega^+(t)=\{(r,t)\in\R_+\times(\ts,\tf]:\ r>\s(t)\}
\]
the restriction $(u^+,\rho^+,E^+)=(u,\rho,E)|_{\Omega^+(t)}$ coincides \emph{exactly} with the Guderley solution at the same spacetime location $(r,t)$.
In particular, $(u^+,\rho^+,E^+)|_{t=\ts}$ is smooth up to $r=\s(\ts)$. The interior state $(u^-,\rho^-,E^-)=(u,\rho,E)|_{\Omega^-(t)}$ is smooth for $t>\ts$, but the trace $(u^-,\rho^-,E^-)|_{t=\ts}$ retains only $C^\frac{\varepsilon}{1+\varepsilon}$ regularity as $r\to\s(\ts)^-$.

\item \textsl{Global agreement at the final time.} \label{prop:large2weak:3}
At $t=\tf$ the solution equals the Guderley self–similar imploding shock, and the shock curves match to $C^\infty$ order:
\[
\tfrac{d^k}{dt^k}\s(\tf)=\tfrac{d^k}{dt^k}\g(\tf)\qquad\text{for all }k\ge0.
\]

\item \textsl{Regularity of the shock path.} \label{prop:large2weak:4}
$\s\in C^\infty((\ts,\tf])\cap C^{1,\varepsilon}([\ts,\tf])$.

\item \textsl{Vanishing strength as $t\downarrow\ts$.} \label{prop:large2weak:5}
There exists $C_1>0$ such that, as $t\to\ts^+$,
\begin{align*}
\dot{\s}(t)-\sp{\lambda_1}(t) &= C_1\,|t-\ts|^{\varepsilon}+O\!\big(|t-\ts|^{2\varepsilon}\big), \\
\jump{z}(t) &= \tfrac{4}{1+\alpha} C_1\,|t-\ts|^{\varepsilon}+O\!\big(|t-\ts|^{2\varepsilon}\big), \\
\jump{w}(t) &= O\!\big(|t-\ts|^{3\varepsilon}\big), \\
\jump{S}(t) &= O\!\big(|t-\ts|^{3\varepsilon}\big).\end{align*}

\end{enumerate}
\end{prop}

\subsection{The region $\R_+ \times [\tf,0)$}
We fix a final matching time $\tf<0$. On the entire time interval $[\tf,0)$ we impose the Guderley imploding shock in self–similar form. With
\[
\xi := \tfrac{r^{\lambda}}{-t},
\]
the state is
\begin{equation}\label{gud:sol}
(u,c,\rho)(r,t)=
\begin{cases}
(0,0,1), & 0\le r^{\lambda}< -t,\\[2mm]
\big(\tfrac1{\lambda}\,r^{1-\lambda}U(\xi),\ \tfrac1{\lambda}\,r^{1-\lambda}C(\xi),\ R(\xi)\big), & r^{\lambda}>-t,
\end{cases}
\qquad t\in[\tf,0),
\end{equation}
where $U,C,R$ are the imploding self–similar profiles built in~\cite{JaLiSc2025}. The shock is located at $\xi=1$, i.e. along $r^{\lambda}=-t$; the interior core $0\le r^{\lambda}<-t$ is quiescent.  
As established in \cite{JaLiSc2025},   the profiles $U, C, R$ are smooth on $[1, +\infty)$, and hence there exists a constant $\m>0$ 
(depending only on the chosen  $(\gamma,\lambda)$) such that for every $\xi \in [1, +\infty)$, 
\begin{subequations}  \label{bounds:profile}
\begin{align}
 |U(\xi)| &\le \m, \quad  0<C(\xi) \le \m, \quad 1 \le | R(\xi)| \le \m,  \\
  |\p_\xi U(\xi)| &\le \m, \quad  |\p_\xi C(\xi)| \le \m, \quad | \p_\xi R(\xi)| \le \m.
 \end{align}
\end{subequations}

For later estimates we parameterize the final time by $\kappa>0$ via 
\begin{equation*}
\kappa^{\frac{\lambda}{\lambda-1}}  = (-\tf)^{-1}
\qquad\Longleftrightarrow\qquad
\kappa = (-\tf)^{\frac{1-\lambda}{\lambda}} .
\end{equation*}
This temporal scaling is encoded in the definition of the exterior velocity and sound speed in \eqref{u-c-rho-plus}.
By the self–similar scaling in \eqref{gud:sol} and the bounds \eqref{bounds:profile}, there exists $\m>0$ (depending only on the chosen $(\gamma,\lambda)$) such that
\begin{equation}
\|u(\cdot,\tf)\|_{\infty}\le \m\,\kappa,\qquad
\|c(\cdot,\tf)\|_{\infty}\le \m\,\kappa,\qquad
\|b(\cdot,\tf)\|_{\infty}\le \m\,\kappa.
\label{ucb-bounds}
\end{equation}

\subsection{The region $\R_+ \times [\ts, \tf]$} \label{sub:sec:r+}
Our objective is the construction of a regular shock solution to the Euler equations \emph{emanating from a preshock} in $\R^d \times [\ts,\tf]$, with data on the future slice $\R_+\times\{\tf\}$ given by \eqref{gud:sol}. At this stage $\ts$ is free; we set $\delta:=\tf-\ts$. The parameters $\delta$ and $\kappa$ will be chosen later to close the bootstrap assumptions (see Proposition~\ref{prop:choice:g} and Lemma~\ref{lemma:om}).

We construct a shock curve $\s:[\ts,\tf]\to\R_+$ (see \eqref{h:ode} and Proposition~\ref{prop:choice:g}) such that $\s(t)\ge \g(t)$. 
In the exterior region $\Omega^+ = \{(r,t): r>\s(t),\ t\in[\ts,\tf]\}$ the Euler solution equals the Guderley solution \eqref{gud:sol}; since $\Omega^+\subset\{(r,t): r>\g(t)\}$, the exterior flow is $C^\infty$.

From the Guderley solution on the right of the shock, we obtain the right traces of $(u,c,\rho)$ for $t\in[\ts,\tf]$:
\begin{subequations} 
\label{shock:rs}
\begin{align}
\usp(t) &:= u(\s(t)^+, t) = \tfrac{\s(t)^{1-\lambda}}{\lambda}\, U\!\left( \tfrac{\s(t)^\lambda}{-t}\right), \\
\csp(t) &:= c(\s(t)^+, t) = \tfrac{\s(t)^{1-\lambda}}{\lambda}\, C\!\left( \tfrac{\s(t)^\lambda}{-t}\right), \\
\rsp(t) &:= \rho(\s(t)^+, t) = R\!\left( \tfrac{\s(t)^\lambda}{-t}\right).
\end{align}
\end{subequations}
Inverting the Rankine–Hugoniot conditions (Lemma~\ref{lemma:rh:inversion}) yields the left traces:
\begin{align*} 
(\usm(t), \csm(t), \rsm(t)) 
&:=( u(\s(t)^-, t), c(\s(t)^-, t), \rho(\s(t)^-, t) )  \\
&= \mbox{  Explicit function of }\big( \usp(t), \csp(t), \rsp(t), \ds(t)\big).
\end{align*}
 
To complete the construction, we solve the Euler equations in $\Omega^- = \{(r,t): 0<r<\s(t),\ t\in[\ts,\tf]\}$ with data on the future boundary of $\Omega^-$ given by
\begin{align*}
u(r, \tf) &= 0, \qquad c(r, \tf) = 0, \qquad \rho(r, \tf) = 1, \qquad 0 \le r< \s(\tf), \\
u(\s(t)^-, t) &= \usm(t), \qquad
c(\s(t)^-, t) = \csm(t), \qquad
\rho(\s(t)^-, t) = \rsm(t), \qquad t \in [\ts, \tf).
\end{align*}
We follow the three characteristic wave families backward-in-time to the slice $\{t=\ts\}$ using a bootstrap argument (Lemma~\ref{lemma:om}).

To close the bootstrap we introduce three parameters\footnote{Because of self‑similarity, the only actual free combination is $\kappa^{\frac{\lambda}{\lambda-1}}\delta$.} $(\delta,\varepsilon,\kappa)$, and we determine them in the order
\begin{equation*} 
\varepsilon \ \to\ \kappa \ \to\ \delta.\end{equation*}
Here $\varepsilon$ is chosen small in Lemma~\ref{lemma:om}, while $\kappa^{\frac{\lambda}{\lambda-1}}\delta$ is chosen small relative to $\varepsilon$ in Proposition~\ref{prop:choice:g}. These constants depend only on the self‑similar Guderley profiles. We let $\m>1$ denote a constant that may change from line to line and depends only on the profile; similarly, we let $\mo\in(0,1)$ denote a profile‑dependent constant that may also change from line to line.

\subsection{The Rankine--Hugoniot jump conditions in spherical symmetry}
Because the solution is radially symmetric, the shock surface is spherical. We take $\mathcal{S} = \{(x,t): \s(t)-|x|=0\}$; the spacetime normal is $(-\ds, e_r)$. We set $u_n=u\cdot n=u\cdot e_r$. Since there is no tangential component, we write $u$ for $u_n$. The Rankine--Hugoniot conditions \eqref{RHg} reduce to the system
\begin{subequations} 
\label{RH}
\begin{align}
\ds \jump{\rho u}&= \jump{\rho u^2 + p}, \\
\ds \jump{\rho}&= \jump{\rho u}, \\
\ds \jump{E} &= \jump{ (p+E)u}.
\end{align}
\end{subequations}
We denote $\Omega^-=\{(r,t)\in\R_+\times(\ts,\tf]: r<\s(t)\}$ and $\Omega^+=\{(r,t)\in\R_+\times(\ts,\tf]: r>\s(t)\}$.

\paragraph{The Lax entropy condition}
The Lax entropy conditions for a $1$-shock are
\begin{equation} \label{lax}
(u-c)|_\s^{+} < \ds< ( u - c)|_\s^{-}, \qquad \ds< u|_\s^+ \le (u + c)|_\s^+ , \qquad \ds< u|_\s^- \le (u + c)|_\s^-.
\end{equation}

\begin{lemma} \label{lemma:rh:inversion}
Let $(\usp,\rsp,\psp,\ds)$ satisfy
\begin{equation}
 \rsp>0, \qquad 0 \le \psp \le \tfrac{(\vsp)^2 \rsp}{\alpha}, \qquad \usp > \ds, \qquad \ds > \usp - \alpha\,\ssp,
 \label{cond:shock}
\end{equation}
where $v=u-\ds$. Then there exists a unique left state $(\usm,\rsm,\psm)$ solving the Rankine--Hugoniot system \eqref{RH} such that
\begin{equation} \label{cond:lax}
\rsm>0, \qquad \psm\ge0, \qquad \usm - \alpha\,\ssm > \ds,
\end{equation}
and it is given explicitly by
\begin{equation} \label{RH:re}
\vsm = \tfrac{ \gamma \psp + \alpha (\vsp)^2 \rsp}{(1+\alpha)\,\vsp\,\rsp}, \qquad
 \rsm = \tfrac{ (1+\alpha) (\vsp)^2 (\rsp)^2}{\gamma \psp + \alpha (\vsp)^2 \rsp}, \qquad
 \psm = \tfrac{ -\alpha \psp + (\vsp)^2 \rsp}{1+\alpha}.
\end{equation}
Moreover, there exist constants $\mo,\m>0$, depending only on $\gamma$, such that
\begin{subequations} 
\label{ratio:bounds}
\begin{align}
\mo \le \tfrac{\ds-\losp}{\losm-\ds} \le \m,\\
\mo \le \tfrac{\lwsp-\ds}{\lwsm-\ds} \le \m,\\
\mo \le \tfrac{\ltsp-\ds}{\ltsm-\ds} \le \m.
\end{align}
\end{subequations}

In the near‑preshock regime where $\sp{\lambda_1}-\ds \searrow 0$ while $\sp{\lambda_3}-\ds$ stays bounded away from $0$ and $(\usp,\rsp,\psp)$ remain in a compact set, the jump vector admits the Taylor expansions
\begin{subequations} \label{RH:taylor}
\begin{align}
\jump{z} &= -\tfrac{4(\sp{\lambda_1}-\ds)}{1+\alpha}
           + \tfrac{4\,(\sp{\lambda_1}-\ds)^2}{(1+\alpha)\,(\sp{\lambda_3}-\ds)}
           + O\!\big((\sp{\lambda_1}-\ds)^3\big),\\
\jump{w} &= \tfrac{4\,(\sp{\lambda_1}-\ds)^3}{(1+\alpha)\,(\sp{\lambda_3}-\ds)^2}
           + O\!\big((\sp{\lambda_1}-\ds)^4\big), \\
\jump{S} &= \tfrac{64\,\alpha\gamma\,(\sp{\lambda_1}-\ds)^3}{3(1+\alpha)^2\,(\sp{\lambda_3}-\ds)^3}
           + O\!\big((\sp{\lambda_1}-\ds)^4\big), \\
\tfrac{\sm{\lambda_1}-\ds}{\ds-\sp{\lambda_1}} &= 1 + O\!\big(\ds-\sp{\lambda_1}\big).
\label{RH:taylor:l1}
\end{align}
\end{subequations}
with constants in the $O(\cdot)$ terms depending only on $\gamma$ and the chosen compact set.\end{lemma}

\begin{proof}[Proof of Lemma~\ref{lemma:rh:inversion}]
We set $v= u -\ds$. A short calculation shows that \eqref{RH} is equivalent to
\begin{subequations} \label{RH:t}
\begin{align}
\jump{ \rho v^2 + p} = 0, \label{RH:t1}\\
\jump{\rho v}=0, \label{RH:t2}\\
\jump{v^2+\tfrac{2\gamma}{\gamma+1}\tfrac{p}{\rho}} =0.
\label{RH:t3}
\end{align}
\end{subequations}
Substituting \eqref{RH:t2} and \eqref{RH:t3} into \eqref{RH:t1}, we find that $\rsm$ solves
\begin{equation} \label{eq:rhop}
f(\rsm) = \rsp (\vsp)^2 + \psp,
\end{equation}
where
\[
f(x)= \tfrac{\gamma+1}{2\gamma} \tfrac{(\vsp \rsp)^2}{x} + \tfrac{\alpha}{\gamma} \Big( (\vsp)^2 + \tfrac{\gamma}{\alpha} \tfrac{\psp}{\rsp} \Big) x, \qquad x>0.
\]
Differentiating,
\[
f'(\bar{x})=0 \iff \bar{x} = \sqrt{ \tfrac{\gamma+1}{\gamma-1} \tfrac{ (\vsp \rsp)^2 }{ (\vsp)^2 + \tfrac{\gamma}{\alpha} \tfrac{\psp}{\rsp} }}.
\]
Because $\rsp>0$, the Lax condition in \eqref{cond:shock} implies that
\begin{align*} 
\rsp> \bar{x}
\iff
\gamma \psp > \rsp (\vsp)^2
\iff
\sqrt{ \tfrac{\gamma \psp}{\rsp}} > \vsp
\iff
\vsp - \alpha \ssp <0,\end{align*}
where we used that $\vsp>0$. Since $f(\rsp) = \rsp (\vsp)^2 + \psp$ and $f$ is convex with $\lim_{x\to0^+}f(x)=\lim_{x\to\infty}f(x)=\infty$, there are exactly 
two positive solutions of \eqref{eq:rhop}. One is $x=\rsp$, which fails \eqref{cond:lax}; thus the admissible solution is the unique $\rsm\in(0,\bar{x})$. Then 
$\vsm$ and $\psm$ follow from \eqref{RH:t2}–\eqref{RH:t3}, giving \eqref{RH:re}. To prove \eqref{ratio:bounds}, we note that \eqref{cond:shock} implies
that
\begin{equation} \label{ineq:ratio:bounds}
(\vsp)^2 \le (\csp)^2 \le \tfrac{\gamma}{\alpha} (\vsp)^2.
\end{equation}
The quoted bounds follow from \eqref{ineq:ratio:bounds} by elementary estimates.

By writing $\vsp = \tfrac12(\ltsp+\losp-2\ds)$ and $\csp=\tfrac12(\ltsp-\losp)$ in \eqref{RH:re}, we obtain that
\begin{align*}
\jump{z} &= \vsm-\vsp -\tfrac{1}{\alpha} ( \csm-\csp) = F_z( \ds-\losp, \ds-\ltsp) \,, \\
\jump{w}& = \vsm-\vsp +\tfrac{1}{\alpha} ( \csm-\csp) = F_w( \ds-\losp, \ds-\ltsp) \,,
\end{align*}
where
\begin{subequations} \label{ugly:rh}
\begin{align}
 &F_z(x_1, x_2) = \tfrac{8 x_1 x_2^3}{(x_1+x_2) \left(4 \alpha x_1 x_2 +(1+\alpha)x_1^2-(1+\alpha)x_2^2 - \sqrt{\left((x_1-x_2)^2+\alpha(x_1+x_2)^2 \right)\left( -\alpha(x_1+x_2)^2 +\gamma(x_1-x_2)^2 \right)} \right)} \,,   \label{eq:Fz} \\
 & F_w(x_1, x_2) = \tfrac{8 x_1^3 x_2}{(x_1+x_2) \left(4 \alpha x_1 x_2 -(1+\alpha)x_1^2+(1+\alpha)x_2^2 + \sqrt{\left((x_1-x_2)^2+\alpha(x_1+x_2)^2 \right)\left( -\alpha(x_1+x_2)^2 +\gamma(x_1-x_2)^2 \right)} \right)} \,. 
\end{align}
\end{subequations}
Since, for any $\bar{x}_2 >0$, we have $F_z, F_w \in C^\infty\left( \left(\tfrac{-\bar{x}_2}{2}, 0 \right) \times ( \bar{x}_2, +\infty) \right)$, we can compute the Taylor series of $F_z$ and $F_w$ around $x_1=0$ for any fixed $x_2>\bar{x}_2$, and find that
\begin{equation} \label{taylor:fz:fw}
 F_z(x_1, x_2) = -\tfrac{4 x_1}{1+\alpha} + \tfrac{4 x_1^2}{(1+\alpha)x_2} + O( x_1^3) \,, \quad
 F_w(x_1, x_2) = \tfrac{4 x_1^3}{(1+\alpha) x_2^2} + O(x_1^4) \,,
\end{equation}
where the constants in the terms $O(\cdot)$ depend only on $\bar{x}_2$ and an upper bound for $x_2$. 
The Taylor series for $ F_z$ and $ F_w$ in  \eqref{taylor:fz:fw} then yield the Taylor series expansions for $\jump{z}$ and $\jump{w}$ stated
in \eqref{RH:taylor}.   A similar calculation yields the Taylor series for $\jump{S}=\jump{2\log(\rho^{-\alpha}c)}$ and $\frac{\losm-\ds}{\ds-\losp}$.
\end{proof}

Next we compute analogues of \eqref{RH:taylor} for the time derivatives of the jumps.

\begin{lemma} \label{lemma:RH:taylor:pt}
We consider $(\usp(t),\rsp(t),\psp(t),\ds(t))\in C^2(\ts,\tf]$ satisfying the hypotheses of Lemma~\ref{lemma:rh:inversion} for every $t\in(\ts,\tf]$. We assume $\losp(\ts)-\ds(\ts)\searrow 0$ as $t\searrow\ts$, and that there exist constants $\rho_0>0$, $\sigma_0>0$, $\m$ with
\[
\rho_0 < \rsp < \m, \qquad \sigma_0 \le \ssp \le \m, \qquad |\usp|\le \m, \qquad \bigl|\tfrac{d}{dt}\wsp\bigr|\le \m, \qquad \bigl|\tfrac{d}{dt}\sp{S}\bigr|\le \m.
\]
Then, there exists $\mathcal{C}>0$, depending only on $\sigma_0,\rho_0,\m$, such that
\begin{subequations} 
\label{RH:taylor:pt}
\begin{align}
\Bigl|\tfrac{d}{dt} \Bigl( \jump{z} +\tfrac{4}{1+\alpha} (\sp{\lambda_1}-\ds)\Bigr) \Bigr| &\le \mathcal{C}\, \bigl( (\sp{\lambda_1}-\ds)\, \tfrac{d}{dt}(\sp{\lambda_1}-\ds) \bigr), \label{taylor-dt-jumpz} \\
\Bigl|\tfrac{d}{dt} \jump{w} \Bigr|& \le \mathcal{C}\, \bigl( (\sp{\lambda_1}-\ds)\, \tfrac{d}{dt}(\sp{\lambda_1}-\ds)^2 \bigr) \,,  \label{taylor-dt-jumpw}\\
\Bigl|\tfrac{d}{dt} \jump{S} \Bigr| &\le \mathcal{C}\, \bigl( (\sp{\lambda_1}-\ds)\, \tfrac{d}{dt}(\sp{\lambda_1}-\ds)^2 \bigr) \,.  \label{taylor-dt-jumps}
\end{align}
\end{subequations}
\end{lemma}
\begin{proof}[Proof of Lemma~\ref{lemma:RH:taylor:pt}] 
From the identity \eqref{eq:Fz}, we obtain a function $\tilde{F}_z$ with bounded derivatives (whose bounds depend on $\sigma_0$ and $\m$) such that
\begin{equation}\label{jumpz-temp}
\jump{z} +\tfrac{4}{1+\alpha} (\sp{\lambda_1}-\ds) = (\losp-\ds)^2\, \tilde{F}_z(\losp-\ds,\ltsp-\ds) \,. 
\end{equation} 
Differentiating \eqref{jumpz-temp} in time provides  the festimate \eqref{taylor-dt-jumpz}. 
The same argument applies to $\jump{w}$ and $\jump{S}$, 
Employing the identities
\[
\jump{w} = (\losp-\ds)^3 \tilde{F}_w(\losp-\ds,\ltsp-\ds), \qquad
\jump{S} = (\losp-\ds)^3 \tilde{F}_S(\losp-\ds,\ltsp-\ds),
\]
and their time-derivatives, the identical argument as used for \eqref{taylor-dt-jumpz} also gives the bounds \eqref{taylor-dt-jumpw} and
\eqref{taylor-dt-jumps}.
\end{proof}

\subsection{Estimates from the  Rankine--Hugoniot jump conditions}
From \eqref{RH:re}, we obtain that
\begin{subequations} 
\label{RH:re2}
\begin{align}
\vsm &= \tfrac{ (\csp)^2 + \alpha ( \vsp)^2 }{(1+\alpha) \vsp }, \label{RH-v} \\
 \rsm &= \tfrac{ (1+\alpha) ( \vsp)^2 \rsp}{(\csp)^2+ \alpha ( \vsp)^2 }, \\
 (\csm)^2 &= \tfrac{ \big(-\alpha (\csp)^2 + \gamma ( \vsp)^2\big)\big( (\csp)^2 + \alpha (\vsp)^2\big)}{(1+\alpha)^2 (\vsp)^2}.
 \end{align}
 \end{subequations} 
Letting
\[
 \xi(t):=\tfrac{\s(t)^\lambda}{-t}\,,
\]
and substituting \eqref{shock:rs} into \eqref{RH:re}, find that
\begin{subequations}
\label{RH:ex:s}
\begin{align}
& \usm(t) - \dot\s(t) = \tfrac{ \s(t)^{1-\lambda}}{\lambda}
\tfrac{ C( \xi )^2 + \alpha ( U( \xi ) - \lambda \ds(t) \s(t)^{\lambda-1})^2 }{ (1+\alpha)( U( \xi ) - \lambda \ds(t) \s(t)^{\lambda-1} ) },  \label{vsm-s}\\
& \rsm(t) = \tfrac{ (1+\alpha)( U( \xi ) - \lambda \ds(t) \s(t)^{\lambda-1} )^2 R( \xi) }{ C( \xi)^2 + \alpha ( U( \xi) - \lambda \ds(t) \s(t)^{\lambda-1} )^2 }, \\
& (\csm(t))^2 = \tfrac{ \s(t)^{2-2\lambda}}{\lambda^2 (1+\alpha)^2 ( U( \xi ) - \lambda \ds(t) \s(t)^{\lambda-1} )^2 } 
 \notag \\
 & \qquad \qquad
\times   ( C( \xi )^2 + \alpha \big( U( \xi ) - \lambda \ds(t) \s(t)^{\lambda-1} )^2 )
  ( -\alpha C( \xi )^2 + \gamma ( U( \xi ) - \lambda \ds(t) \s(t)^{\lambda-1} )^2 ) \,.
 \end{align}
 \end{subequations}

\subsection{The choice of $\s(t)$}

The Rankine-Hugoniot jump condition \eqref{RH-v}, when written explicitly using the Guderley ansatz as in \eqref{vsm-s}, takes the form
\begin{align*}
  \dot\s(t) &=\usm(t) - \tfrac{ \s(t)^{1-\lambda}}{\lambda}
\tfrac{ C( \xi )^2 + \alpha ( U( \xi ) - \lambda \dot\s(t) \s(t)^{\lambda-1} )^2 }{ (1+\alpha)( U\!\big( \xi ) - \lambda \dot\s(t) \s(t)^{\lambda-1} ) }\,.
\end{align*}
This equation highlights a fundamental difficulty in the backwards-in-time construction: the velocity of the shock curve $\dot\s(t)$ depends not only on the known upstream (exterior) state $(\cdot)^+$, but also on the \textit{a priori} unknown downstream (interior) state $\usm(t)$.

To circumvent this, we prescribe the shock trajectory by introducing an auxiliary function $g\colon [\ts, \tf] \to \R$, defined as
\begin{equation*}
g(t) = 1+\tfrac{\losp(t)-\dot\s(t)}{\csp(t)}.
\end{equation*}
This function measures the normalized deviation of the shock speed $\dot\s(t)$ from the fast acoustic speed $\losp(t)$ ahead of the shock. Crucially, $g(t)$ depends only on the exterior state and the shock speed itself.

We can unpack this definition to derive an explicit ODE for the shock curve $\s(t)$. Recalling the definition of the fast acoustic speed, $\lambda_1 = u-c$, the exterior trace is $\losp(t) = \usp(t) - \csp(t)$. Substituting this into the definition of $g(t)$ yields
\begin{align*}
g(t) 
= 1 + \tfrac{(\usp(t) - \csp(t)) - \dot\s(t)}{\csp(t)} 
= 1 + \tfrac{\usp(t) - \dot\s(t)}{\csp(t)} - \tfrac{\csp(t)}{\csp(t)} 
= \tfrac{\usp(t) - \dot\s(t)}{\csp(t)}.
\end{align*}
This reveals that $g(t)$ is the Mach number of the flow relative to the shock front in the exterior region. We can rearrange this expression to solve for the shock speed:
\begin{align*}
g(t) \csp(t) &= \usp(t) - \dot\s(t) \\
\dot\s(t) &= \usp(t) - g(t) \csp(t).
\end{align*}
This equation now defines the evolution of the shock entirely in terms of the exterior state and the chosen function $g(t)$.

Finally, substituting the Guderley ansatz for the exterior traces (from \eqref{shock:rs}),
\begin{align*}
\usp(t) &= \tfrac{\s(t)^{1-\lambda}}{\lambda} U\left( \tfrac{\s(t)^\lambda}{-t}\right), \qquad
\csp(t) = \tfrac{\s(t)^{1-\lambda}}{\lambda} C\left( \tfrac{\s(t)^\lambda}{-t}\right),
\end{align*}
we arrive at the explicit, non-autonomous ODE for the shock trajectory:
\begin{equation}
\dot{\s}(t) = \tfrac{\s^{1-\lambda}(t)}{\lambda}
\left( U\!\left(\tfrac{\s^\lambda(t)}{-t}\right) - g(t)\, C\!\left(\tfrac{\s^\lambda(t)}{-t}\right) \right), \qquad \s(\tf) = (-\tf)^{\frac{1}{\lambda}}.
\label{s-ode}
\end{equation}

At this stage, $g(t)$ is a free parameter of the construction. By carefully engineering $g(t)$ (see Proposition~\ref{prop:choice:g} below) with specific matching conditions at $t=\tf$ and $t=\ts$, we can enforce key properties of the shock curve. In particular, we ensure that the preshock occurs precisely at time $t=\ts$ (which, due to the Taylor expansions \eqref{RH:taylor}, corresponds to $g(\ts)=1$), and that we obtain suitable bounds for the Differentiated Riemann Variables $\rwsm, \rzsm, \rbsm$ along the shock curve. These bounds are essential to close the bootstrap argument in the region~$\Omega^-$ (see Lemma~\ref{lemma:om} below).

\begin{remark}[\textsl{Prescribing the shock trajectory}]
It is important to emphasize that the ODE \eqref{s-ode} is derived directly from the definition of the auxiliary function $g(t)$, not by manipulating the Rankine-Hugoniot condition \eqref{vsm-s}. The latter involves the unknown interior state $\usm(t)$, making it unsuitable for determining $\s(t)$ directly in this backwards construction. Instead, the strategy is to use $g(t)$ to \emph{prescribe} the shock path $\s(t)$ via \eqref{s-ode}. The Rankine-Hugoniot conditions (including \eqref{vsm-s}) are subsequently used algebraically to determine the interior state $\usm(t)$, ensuring the resulting solution remains consistent with the conservation laws.
\end{remark}

In order to simplify the ODE \eqref{s-ode} and facilitate the analysis of the shock dynamics, we introduce the change of variables
\[
\h(t):=\s(t)^\lambda.
\]
This transformation\footnote{By defining the function $\h(t):=\s(t)^\lambda$, \eqref{RH:re2} has the equivalent  form
\begin{subequations} \label{RH:ex:h}
\begin{align}
\usm(t) - \dot\s(t) &= \tfrac{1}{\lambda}  \h(t)^{\frac{1-\lambda}{\lambda}} \,
\tfrac{ C(\xi(t))^2 + \alpha( U(\xi(t)) - \dot{\h}(t))^2 }{(1+\alpha)( U(\xi(t)) - \dot{\h}(t)) }, \label{vsm} \\
\rsm(t) &= \tfrac{ (1+\alpha)( U(\xi(t)) - \dot{\h}(t))^2\, R(\xi(t)) }{ C( \xi(t))^2 + \alpha( U(\xi(t)) - \dot{\h}(t))^2 }, \label{rsm} \\
 (\csm(t))^2
 & = \tfrac{ \h(t)^{\frac{2-2\lambda}{\lambda}}}{\lambda^2 (1+\alpha)^2( U(\xi(t)) - \dot{\h}(t))^2 } 
   ( C( \xi(t))^2 + \alpha( U(\xi(t)) - \dot{\h}(t))^2)\, ( -\alpha C( \xi(t))^2 + \gamma ( U(\xi(t)) - \dot{\h}(t))^2).
 \label{csm}
\end{align}
\end{subequations}}
is designed to convert the term $\lambda\dot \s \s^{\lambda-1}$ into a perfect time-derivative. Differentiating the definition of $\h(t)$ and substituting the expression for $\dot\s(t)$ from \eqref{s-ode}, we find  
\begin{align*}
\dot{\h}(t) &= \lambda \s(t)^{\lambda-1} \dot{\s}(t) \\
&= \lambda \s(t)^{\lambda-1} \left[ \tfrac{\s(t)^{1-\lambda}}{\lambda}
\left( U\!\left(\tfrac{\s(t)^\lambda}{-t}\right) - g(t)\, C\!\left(\tfrac{\s(t)^\lambda}{-t}\right) \right) \right] \\
&= U\!\left(\tfrac{\h(t)}{-t}\right) - g(t)\, C\!\left(\tfrac{\h(t)}{-t}\right).
\end{align*}
Thus, $\h(t)$ solves the (backward-in-time) ODE
\begin{equation} \label{h:ode}
\dot{\h}(t) = U\!\left(\tfrac{\h(t)}{-t}\right) - g(t)\, C\!\left(\tfrac{\h(t)}{-t}\right), \qquad \h(\tf) = -\tf \,.
\end{equation}

For the construction to yield a valid physical solution that matches the Guderley profile at $t=\tf$ and emanates from a preshock at $t=\ts$, the solution $\h(t)$ to the ODE \eqref{h:ode} must satisfy several constraints. The necessary physical requirements on the shock trajectory $\s(t)$ and the associated flow translate into the following equivalent conditions on the transformed variable $\h(t)$:
\begin{subequations} 
\label{cond:shock:h}
\begin{align}
\s(\tf) = (-\tf)^{\frac{1}{\lambda}}
&\iff \h(\tf) = - \tf, \label{cond:h:tf}\\
\dot\s(\tf) = -\tfrac{1}{\lambda}(-\tf)^{\frac{1}{\lambda}-1}
&\iff \dot{\h}(\tf) = -1, \label{cond:d1:h:tf} \\
\frac{d^k}{dt^k}\s(\tf) = \frac{d^k}{dt^k}\g(\tf)
&\iff \frac{d^k}{dt^k} \h(\tf) = 0 \quad (k\ge 2), \label{cond:dk:h:tf} \\
\dot\s(\ts) = \lambda_1|_\s^+(\ts)
& \iff \dot{\h}(\ts) = U\!\left(\tfrac{\h(\ts)}{-\ts}\right) -C\!\left(\tfrac{\h(\ts)}{-\ts}\right), \label{cond:h:ts} \\
\s(t) \ge \g(t) & \iff \h(t) \ge - t \quad \text{for } t \in [\ts, \tf], \label{cond:h:gtr:t}\\
\sp{\rho} > 0 & \iff R\!\left(\tfrac{\h(t)}{-t}\right) > 0 \quad \text{for } t \in [\ts, \tf], \label{cond:shock:h:R} \\
\gamma (\vsp(t))^2 \ge \alpha (\csp(t))^2
&\iff \gamma \big( U\!\left(\tfrac{\h(t)}{-t}\right) - \dot{\h}(t) \big)^2 \ge \alpha\, C\!\left(\tfrac{\h(t)}{-t}\right)^2 \quad \text{for } t \in [\ts, \tf], \label{cond:shock:h:p} \\
\sp{u}(t) > \dot\s(t) &\iff U\!\left(\tfrac{\h(t)}{-t}\right) > \dot{\h}(t) \quad \text{for } t \in [\ts, \tf], \label{cond:shock:h:lax2}  \\
\dot\s(t) > \sp{u}(t) - \sp{c}(t) &\iff U\!\left(\tfrac{\h(t)}{-t}\right) - C\!\left(\tfrac{\h(t)}{-t}\right) < \dot{\h}(t) \quad \text{for } t \in (\ts, \tf]. \label{cond:shock:h:lax}
\end{align}
\end{subequations}

The constraints \eqref{cond:shock:h} serve the following distinct purposes in the construction:
\begin{itemize}
    \item \textsl{Matching at $\tf$:} Conditions~\eqref{cond:h:tf},~\eqref{cond:d1:h:tf}, and~\eqref{cond:dk:h:tf} guarantee that the shock curve~$\s(t)$ seamlessly matches the Guderley shock curve $\g(t)$ to all orders at the final time~$\tf$.
    \item \textsl{Preshock formation at $\ts$:} Condition~\eqref{cond:h:ts} ensures that the shock strength vanishes at $t=\ts$ (the shock speed matches the characteristic speed), so that $(\s(\ts),\ts)$ is a preshock.
    \item \textsl{Geometric positioning:} Condition \eqref{cond:h:gtr:t} ensures that the constructed shock curve $\s(t)$ lies to the ``right'' (exterior) of the original Guderley shock curve $\g(t)$, ensuring the exterior solution remains the smooth Guderley flow.
    \item \textsl{Physical validity and RH inversion:} The remaining conditions ensure that we can apply Lemma~\ref{lemma:rh:inversion} to invert the Rankine-Hugoniot jump conditions. Specifically, \eqref{cond:shock:h:R} ensures positive exterior density. Condition \eqref{cond:shock:h:p} guarantees that the resulting interior pressure $\psm$ is non-negative (as required by \eqref{cond:lax} in Lemma~\ref{lemma:rh:inversion}). Finally, \eqref{cond:shock:h:lax2} and~\eqref{cond:shock:h:lax} are the necessary Lax geometric (entropy) conditions \eqref{lax} for the exterior state.
\end{itemize}

These constraints on $\h(t)$ translate directly into requirements on the auxiliary function $g(t)$. By substituting the ODE \eqref{h:ode} into the constraints \eqref{cond:shock:h:p}, \eqref{cond:shock:h:lax2}, and \eqref{cond:shock:h:lax}, and assuming $C>0$, we find that
\begin{align*}
    \eqref{cond:shock:h:p} &\iff \gamma (g(t) C)^2 \ge \alpha C^2 \iff g(t)^2 \ge \alpha/\gamma, \\
    \eqref{cond:shock:h:lax2} &\iff g(t) > 0, \\
    \eqref{cond:shock:h:lax} &\iff g(t) < 1.
\end{align*}
Furthermore, the preshock condition \eqref{cond:h:ts} requires $g(\ts)=1$, and the Guderley matching condition \eqref{cond:d1:h:tf} requires $g(\tf)=\sqrt{ \alpha/\gamma}$ (the value corresponding to the strong shock limit). This motivates the precise properties imposed on the function $g(t)$ (which will be constructed in Proposition \ref{prop:choice:g}) as follows:
\begin{subequations} \label{g:prop}
\begin{align}
&g \in C[\ts, \tf]\cap C^\infty(\ts, \tf], \\
&g(\tf)= \sqrt{\tfrac{\alpha}{\gamma}} <1, \qquad g(\ts)=1, \qquad \dot{g}(t)<0 \text{ for all } t\in[\ts,\tf), \\
&\tfrac{d^k}{dt^k}g(\tf) = 0 \text{ for every } k\ge1.
\end{align}
\end{subequations}
The $C^\infty$ flatness at $\tf$ ensures the higher-order matching conditions \eqref{cond:dk:h:tf} are satisfied.

\begin{prop} \label{prop:cond:h}
If $\h:[\ts,\tf]\to\R$ solves the ODE \eqref{h:ode} with $g$ satisfying the properties \eqref{g:prop}, then $\h$ satisfies all the constraints listed in \eqref{cond:shock:h}.
\end{prop}
\begin{proof}[Proof of Proposition~\ref{prop:cond:h}]
Let $F(t, \h) = U(\h/(-t)) - g(t)C(\h/(-t))$ denote the right-hand side of \eqref{h:ode}.

\textit{Existence and Uniqueness:} Since the Guderley profiles $U, C$ are smooth on $[1, \infty)$ (by \eqref{bounds:profile}), and $g(t)$ is continuous (by \eqref{g:prop}), $F(t, \h)$ is continuous in $t$ and locally Lipschitz in $\h$ (as $t$ is bounded away from 0). By the Cauchy–Lipschitz (Picard-Lindelöf) theorem, a unique local solution exists backward from $t=\tf$. Furthermore, since $U, C, g$ are bounded, $|F(t,\h)| \le \m$. This guarantees the solution $\h(t)$ is defined globally on the entire interval $[\ts,\tf]$.

\textit{Matching Conditions:} Condition \eqref{cond:h:tf} is the prescribed terminal condition. For \eqref{cond:d1:h:tf}, we evaluate the ODE at $t=\tf$. Using $\h(\tf)=-\tf$ (so $\xi=1$) and $g(\tf)=\sqrt{\alpha/\gamma}$, we verify the required shock speed:
\[
\dot{\h}(\tf)=U(1)-\sqrt{\tfrac{\alpha}{\gamma}}\, C(1)= -\tfrac{2}{\gamma+1} - \sqrt{\tfrac{\gamma-1}{2\gamma}} \tfrac{\sqrt{2\gamma(\gamma-1)}}{\gamma+1} = -\tfrac{2}{\gamma+1} - \tfrac{\gamma-1}{\gamma+1} = -1.
\]
The higher-order matching conditions \eqref{cond:dk:h:tf} follow by differentiating \eqref{h:ode}. Since $\frac{d^k}{dt^k}g(\tf) = 0$ for $k\ge 1$ (by \eqref{g:prop}), and $\dot{\h}(\tf)=-1$, the time derivatives of the similarity variable $\xi(t)=\tfrac{\h(t)}{-t}$ vanish at $t=\tf$. By induction, we conclude $\frac{d^k}{dt^k} \h(\tf) = 0$ for $k\ge 2$.

\textit{Preshock Condition:} The preshock condition \eqref{cond:h:ts} follows immediately from the ODE \eqref{h:ode} because $g(\ts)=1$.

\textit{Geometric Condition (Comparison Principle):} We must verify \eqref{cond:h:gtr:t}, i.e., $\h(t) \ge -t$. Let $\h_G(t) = -t$ denote the Guderley shock path, so $\dot{\h}_G(t) = -1$. We compare $\h(t)$ with $\h_G(t)$.
We evaluate the vector field $F$ along the Guderley path:
\[
F(t, \h_G(t)) = U(1) - g(t)C(1).
\]
Since $g(t)$ is decreasing, $g(t) \ge g(\tf) = \sqrt{\alpha/\gamma}$ (by \eqref{g:prop}). As calculated above, $\sqrt{\alpha/\gamma}C(1) = U(1)+1$. Since $C(1)>0$, we have that
\[
F(t, \h_G(t)) \le U(1) - g(\tf)C(1) = -1 = \dot{\h}_G(t) \,.
\]
We have $\dot{\h}(t) = F(t, \h(t))$ and $\dot{\h}_G(t) \ge F(t, \h_G(t))$, with $\h(\tf) = \h_G(\tf)$. Since $F$ is Lipschitz in $\h$, the standard ODE comparison principle (applied backward in time) implies that $\h(t) \ge \h_G(t)$ for all $t \in [\ts, \tf]$.

\textit{Physical and Lax Conditions:} Condition \eqref{cond:h:gtr:t} ensures the similarity variable $\xi(t) = \h(t)/(-t) \ge 1$. Since the Guderley density profile $R(\xi)$ is positive for $\xi \ge 1$, \eqref{cond:shock:h:R} holds. The remaining conditions \eqref{cond:shock:h:p}–\eqref{cond:shock:h:lax} are equivalent to the requirement $\sqrt{\alpha/\gamma} \le g(t) \le 1$ (given $C>0$), which is guaranteed by the properties of $g(t)$ in \eqref{g:prop}.
\end{proof}

We now derive essential bounds on $\h(t)$ and its derivatives, which are crucial for the subsequent analysis. Since the profiles $U, C$ are smooth and bounded, and $g(t)$ is bounded, the ODE \eqref{h:ode} immediately implies that $\dot{\h}$ is bounded. Furthermore, since $U<0$ and $C>0$ for the Guderley profile, we have $\dot{\h}(t) < 0$.
\begin{subequations}
\label{h:bounds}
\begin{align}
-\m \le \dot{\h}(t) \le 0.
\end{align}
Integrating this bound backwards from $t=\tf$ over the interval of length $\delta = \tf - \ts$, we obtain bounds on $\h(t)$:
\begin{align}
|\tf| \le \h(t) \le |\tf| + \m \delta.
\end{align}
These bounds on $\h(t)$ translate into bounds on the similarity variable $\xi(t)=\h(t)/(-t)$ and its time derivative. Recalling $\kappa^{\frac{\lambda}{\lambda-1}} = (-\tf)^{-1}$, and assuming $\delta$ is small relative to $|\tf|$, we find that
\begin{align}
1 \le \tfrac{\h(t)}{-t} \le 1 + \m \delta \kappa^{\frac{\lambda}{\lambda-1}}, \qquad &\left| \tfrac{d}{dt} \left(\tfrac{\h(t)}{-t}\right) \right| = \left| \tfrac{\dot{\h}(t)(-t) + \h(t)}{(-t)^2} \right| \le \tfrac{\m}{-t} \le \m \kappa^{\frac{\lambda}{\lambda-1}}.
\end{align}
\end{subequations}
(We allow the constant $\m$ to change line by line, depending only on the Guderley profile).

Finally, we estimate the acceleration $\ddot{\h}(t)$. Differentiating the ODE~\eqref{h:ode} in time, we obtain the identity:
\begin{equation*}
\ddot{\h}(t) = \tfrac{d}{dt}\left(\tfrac{\h(t)}{-t} \right) \left( \p_\xi U \left(\tfrac{\h(t)}{-t} \right) - g(t) \p_\xi C \left(\tfrac{\h(t)}{-t} \right) \right) - \dot{g}(t) C\left(\tfrac{\h(t)}{-t} \right).
\end{equation*}
We analyze the magnitudes of these terms. The second term, $-\dot{g}(t) C(\xi(t))$, is strictly positive since $\dot{g}(t)<0$ (by \eqref{g:prop}) and $C>0$. The first term is bounded using the estimates in \eqref{h:bounds} (for the time derivative of $\xi(t)$) and the bounds on the profiles \eqref{bounds:profile} (for the derivatives of $U$ and $C$). It follows that the acceleration is controlled by $| \dot{g}(t)|$ and the background scale $\kappa^{\frac{\lambda}{\lambda-1}}$:
\begin{equation} \label{second:h:bound}
\frac{1}{\m} ( | \dot{g}(t)| + \kappa^{\frac{\lambda}{\lambda-1}}) \le \ddot{\h}(t) \le \m ( | \dot{g}(t)| + \kappa^{\frac{\lambda}{\lambda-1}}).
\end{equation}

Using \eqref{RH:ex:h},  we obtain the trace bounds
\begin{equation} \label{shock:bound}
|\sm{u}| \le \m \kappa, \qquad
 | \sm{\rho}| \le \m, \qquad
|\sm{c}| \le \m \kappa.
\end{equation}
In what follows, we  shall compute time derivatives of the left traces using \eqref{RH:ex:h} and deduce bounds needed to close the bootstrap in $\Omega^-$.

\subsection{Time Derivatives of the Interior Traces}
In \S~\ref{drv:s}, we will obtain bounds for the DRVs~$\rwsm, \rzsm, \rbsm$ along the shock curve. These bounds rely on estimates for the time derivatives of the interior traces: $\p_t \usm, \p_t \csm, \p_t \rsm$, and $\p_t \bsm$.

We first simplify the expressions for the interior traces given in \eqref{RH:ex:h} by utilizing the relationship between $\dot{\h}(t)$ and $g(t)$ established by the ODE \eqref{h:ode}. Let $\xi(t) = \h(t)/(-t)$. The ODE implies the crucial identity:
\[
\dot{\h}(t) = U(\xi(t)) - g(t)C(\xi(t)) \implies U(\xi(t)) - \dot{\h}(t) = g(t)C(\xi(t)).
\]
Substituting this identity into \eqref{RH:ex:h} significantly simplifies the algebraic structure.

\subsubsection{Computing $\p_t \rsm$}
We begin with the density $\rsm$. Substituting the identity above into \eqref{rsm}, we find:
\begin{align}
\rsm(t) 
= \tfrac{ (1+\alpha)( U(\xi) - \dot{\h})^2\, R(\xi) }{ C( \xi)^2 + \alpha( U(\xi) - \dot{\h})^2 } = \tfrac{ (1+\alpha)( g(t)C(\xi))^2\, R(\xi) }{ C( \xi)^2 + \alpha( g(t)C(\xi))^2 }
= \tfrac{(1+\alpha) g(t)^2}{1+\alpha g(t)^2}\, R\!\left(\tfrac{\h(t)}{-t}\right). \label{rsm:g}
\end{align}
We differentiatie \eqref{rsm:g} in time; we first compute the derivative of the $g$-dependent coefficient as
\[
\tfrac{d}{dg}\left( \tfrac{(1+\alpha) g^2}{1+\alpha g^2} \right) = \tfrac{2(1+\alpha)g}{(1+\alpha g^2)^2} \,,
\]
so that
\begin{equation*}
\p_t \rsm = \dot{g}(t)\, R\!\left(\tfrac{\h}{-t}\right) \tfrac{2(1+\alpha) g(t)}{(1+\alpha g(t)^2)^2} + \tfrac{(1+\alpha) g(t)^2}{1+\alpha g(t)^2}\, \p_t \!\left(\tfrac{\h}{-t}\right)\, \p_\xi R\!\left(\tfrac{\h}{-t}\right).
\end{equation*}
Using the bounds on the profiles \eqref{bounds:profile} and the bounds on $\h$ and $\p_t(\h/(-t))$ from \eqref{h:bounds}, we find that
\begin{equation} \label{pt:rsm:b}
\big| \p_t \rsm \big| \le \m \Big( |\dot{g}| + \kappa^{\frac{\lambda}{\lambda-1}} \Big).
\end{equation}
Furthermore, since $g(t)^2 \ge \alpha/\gamma$ (by \eqref{g:prop}) and $R(\xi)\ge \mo$ (by \eqref{bounds:profile}), we obtain a uniform lower bound from \eqref{rsm:g} (noting that the coefficient function is increasing in $g^2$) which is given by
\begin{equation} \label{lower:bound:rsm}
\sm{\rho}(t) \ge \tfrac{\alpha}{\gamma} R\!\left(\tfrac{\h(t)}{-t}\right) \ge \mo > 0.
\end{equation}

\subsubsection{Computing $\p_t \csm$}
Next, we analyze the sound speed $\csm$. Substituting $U-\dot{\h} = gC$ into \eqref{csm}, we find that
\begin{align*}
(\csm(t))^2
&= \Big( \tfrac{\h(t)^{\frac{1-\lambda}{\lambda}}}{\lambda} C(\xi(t)) \Big)^2 \tfrac{ (1+ \alpha g(t)^2)( \gamma g(t)^2 - \alpha) }{(1+\alpha)^2 g(t)^2}.
\end{align*}
Taking the square root yields the simplified expression
\begin{equation*}
\csm(t) = \tfrac{\h(t)^{\frac{1-\lambda}{\lambda}}}{\lambda}\, C(\xi(t))\, \mathcal{G}(g(t)), \quad \text{where} \quad \mathcal{G}(g) = \tfrac{ \sqrt{\big(1+ \alpha g^2\big)\big( \gamma g^2 - \alpha\big)} }{(1+\alpha) g}.
\end{equation*}
Differentiating this expression yields a lengthy expression involving derivatives of $\h, g,$ and $C$ which we write as
\begin{align*}
\p_t \csm &= \tfrac{ 1-\lambda}{\lambda^2 \h}\,\dot{\h}\, \csm + \tfrac{\h^{\frac{1-\lambda}{\lambda}}}{\lambda} \mathcal{G}(g) \, \p_t \!\left( \tfrac{\h}{-t} \right) \p_\xi C(\xi) + \tfrac{\h^{\frac{1-\lambda}{\lambda}}}{\lambda}\, C(\xi)\, \dot{g} \mathcal{G}'(g).
\end{align*}
The crucial observation is that the derivative $\mathcal{G}'(g)$ involves the term $(\gamma g^2 - \alpha)^{-1/2}$ in the denominator. Since $g(t)$ approaches $\sqrt{\alpha/\gamma}$ (as $t \to \tf$), this derivative can become large. Using the established bounds \eqref{h:bounds} and \eqref{bounds:profile}, and recalling that $\h^{\frac{1-\lambda}{\lambda}}$ is of order $\kappa$, we find that the estimate is given by
\begin{equation} \label{pt:csm:b}
\bigl|\p_t \csm\bigr| \le \m \kappa \left(\kappa^{\frac{\lambda}{\lambda-1}} + \left|\dot{g}\mathcal{G}'(g)\right| \right) \le \m \kappa \Big(\kappa^{\frac{\lambda}{\lambda-1}} + \tfrac{ |\dot{g}|}{ \sqrt{ \gamma g^2 - \alpha}} \Big).
\end{equation}

\subsubsection{Computing $\p_t \bsm$}
Since the pseudo-entropy is related by $b=c\,\rho^{-\alpha}$, we differentiate in time: $\p_t b = \p_t c \,\rho^{-\alpha} - \alpha c \rho^{-\alpha-1} \p_t \rho$. Applying the bounds \eqref{pt:rsm:b}, the lower bound \eqref{lower:bound:rsm}, and the estimate \eqref{pt:csm:b} implies the same bound holds for $\p_t \bsm$ and is given by
\begin{equation}
\label{pt:bsm}
\bigl|\p_t \bsm\bigr| \le \m \kappa \Big(\kappa^{\frac{\lambda}{\lambda-1}} + \tfrac{ |\dot{g}|}{ \sqrt{ \gamma g^2 - \alpha}} \Big).
\end{equation}

\subsubsection{Computing $\p_t \usm$}
\label{pt-vsm}
Finally, we compute the derivative of the interior velocity $\usm$. We first derive the simplified expression for $\usm$ using $g(t)$. We use $\usm = \vsm + \dot\s$. We simplify $\vsm$ (from \eqref{vsm} using $U-\dot{\h}=gC$) and obtain that
\[
\vsm(t) = \tfrac{\h(t)^{\frac{1-\lambda}{\lambda}}}{\lambda} C(\xi(t)) \tfrac{1+\alpha g(t)^2}{(1+\alpha)g(t)}.
\]
Combining this with the ODE for $\dot\s$ (from \eqref{s-ode}) which is
$\dot\s(t) = \tfrac{\h(t)^{\frac{1-\lambda}{\lambda}}}{\lambda} (U(\xi(t)) - g(t)C(\xi(t)))$, 
we obtain the simplified expression for $\usm(t)$:
\begin{align*}
\usm(t) 
&= \tfrac{\h(t)^{\frac{1-\lambda}{\lambda}}}{\lambda} \left( U\!\left(\tfrac{\h(t)}{-t}\right) + C\!\left(\tfrac{\h(t)}{-t}\right)\tfrac{1-g(t)^2}{(1+\alpha)g(t)} \right).
\end{align*}
Differentiating this expression in time yields
\begin{align*}
\p_t \usm &= \tfrac{1-\lambda}{\lambda^2} \tfrac{\dot{\h}}{\h}\, \usm + \tfrac{ \h^{\frac{1-\lambda}{\lambda}}}{\lambda}\, \p_t \!\left(\tfrac{ \h}{-t}\right)\!\left( \p_\xi U\!\left(\tfrac{\h}{-t}\right) + \p_\xi C\!\left(\tfrac{\h}{-t}\right)\tfrac{1-g(t)^2}{(1+\alpha)g(t)} \right)  
+ \tfrac{ \h^{\frac{1-\lambda}{\lambda}}}{\lambda}\, C\!\left(\tfrac{\h}{-t}\right) \dot{g}(t) \tfrac{d}{dg}\left(\tfrac{1-g^2}{(1+\alpha)g}\right).
\end{align*}
The derivative $\tfrac{d}{dg}(\tfrac{1-g^2}{(1+\alpha)g}) = \tfrac{-(1+g^2)}{(1+\alpha)g^2}$ is bounded since $g(t)$ is bounded away from zero by \eqref{g:prop}. Therefore, using \eqref{h:bounds} and \eqref{bounds:profile}, we arrive at the estimate
\begin{equation}
\label{pt:usm}
\bigl| \p_t \usm \bigr|  \le \m \kappa \Big( |\dot{g}| + \kappa^{\frac{\lambda}{\lambda-1}} \Big).
\end{equation}

\subsection{The Differentiated Riemann Variables along $\s(t)$ and the Construction of $g(t)$} \label{drv:s}

To solve the Euler equations in the interior region $\Omega^-$ using the method of characteristics (Lemma \ref{lemma:om}), we must establish control over the spatial derivatives of the flow variables along the shock boundary $\s(t)$. We achieve this by relating the time derivatives of the traces (computed in the previous subsections) to the spatial derivatives via the compatibility conditions along the shock trajectory.

We apply the total time derivative $\tfrac{d}{dt} = \p_t + \dot\s \p_r$ to the traces $\wsm(t), \zsm(t), \bsm(t)$. Substituting the Euler equations \eqref{euler:rv} evaluated on the interior side $(\cdot)|_\s^-$ to eliminate the partial time derivatives $\p_t w, \p_t z, \p_t b$ yields
\begin{subequations}
\label{pt:s:pr}
\begin{align}
\p_t \wsm &= ( \dot\s - \ltsm)\, \p_r w|_\s^- - \tfrac{(d-1)\alpha \big( (\wsm)^2-(\zsm)^2\big)}{4 \s(t)} + \tfrac{b\, \rho^{2\alpha}\, \p_r b}{\gamma\alpha}\Big|_\s^-, \\
\p_t \zsm &= ( \dot\s - \losm)\, \p_r z|_\s^- + \tfrac{(d-1)\alpha \big( (\wsm)^2-(\zsm)^2\big)}{4 \s(t)}+ \tfrac{b\, \rho^{2\alpha}\, \p_r b}{\gamma\alpha}\Big|_\s^-, \\
\p_t \bsm &= ( \dot\s - \lwsm)\, \p_r b|_\s^-. \label{pt:s:pr-b}
\end{align}
\end{subequations}
(Note that $\p_t \wsm$ here denotes the total time derivative $\tfrac{d}{dt} \wsm(t)$).

We can rearrange these equations to solve for the spatial derivatives. The key challenge is that the coefficients involve the differences between the shock speed and the characteristic speeds. In particular, the coefficient $(\dot\s - \losm)$ vanishes as $t\to\ts$, potentially leading to singular behavior in $\p_r z|_\s^-$.

First, we bound the geometric source terms. From the bounds on $\h(t)$ in \eqref{h:bounds} (which estimates $\s(t)$ away from zero) and the bounds on the traces in \eqref{shock:bound} (which bounds $w, z$ by $\m\kappa$), we have that 
\begin{equation*}
\left| \tfrac{(d-1)\alpha \big( (\wsm)^2-(\zsm)^2\big)}{4 \s(t)} \right| \le \m\, \kappa^{1+ \tfrac{\lambda}{\lambda-1}}.
\end{equation*}

We now utilize the estimates for the time derivatives established previously (\eqref{pt:usm}, \eqref{pt:csm:b}, \eqref{pt:bsm}). By the Lax conditions, the characteristic gaps are positive (e.g., $\losm-\dot\s > 0$).

From \eqref{pt:s:pr-b}, we have that
\begin{align*}
 \big| \p_r b |_\s^- \big|  &\le \tfrac{| \p_t \bsm |}{ \lwsm-\dot\s } \le \tfrac{ \m \kappa }{\lwsm-\dot\s}\Big(\kappa^{\frac{\lambda}{\lambda-1}} 
 + \tfrac{ |\dot{g}|}{ \sqrt{ \gamma g^2 - \alpha}} \Big).
\end{align*}
The entropy source terms (involving $b\rho^{2\alpha}\p_r b$) and the geometric source term are also bounded using these estimates. Since these forcing terms are comparable in magnitude to the time derivative terms (as $b\sim\kappa$ and the characteristic speeds are $\sim\kappa$), we can absorb them into the main estimate by allowing the constant $\m$ to increase, and hence, we have that
\begin{align*}
 \big| \p_r w |_\s^-\big|  &\le \tfrac{| \p_t \wsm |}{ \ltsm-\dot\s } + O(\text{Forcing}) \le \tfrac{ \m \kappa }{\ltsm-\dot\s}\Big(\kappa^{\frac{\lambda}{\lambda-1}} + \tfrac{ |\dot{g}|}{ \sqrt{ \gamma g^2 - \alpha}} \Big), \\
 \big| \p_r z |_\s^- \big|& \le \tfrac{| \p_t \zsm |}{ \losm-\dot\s } + O(\text{Forcing}) \le \tfrac{ \m \kappa }{\losm-\dot\s}\Big(\kappa^{\frac{\lambda}{\lambda-1}} + \tfrac{ |\dot{g}|}{ \sqrt{ \gamma g^2 - \alpha}} \Big).
\end{align*}

To understand the behavior near the preshock time, we relate the denominators to the function $g(t)$. By Lemma \ref{lemma:rh:inversion} (the ratio bounds \eqref{ratio:bounds}), the interior speed differences are comparable to the exterior ones. We recall that the identities for the exterior gaps are given by
\begin{align*}
    |\lwsp - \dot\s| = g(t)\csp, \quad
    |\ltsp - \dot\s| = (1+g(t))\csp, \quad
    |\dot\s - \losp| = (1-g(t))\csp.
\end{align*}
Since $\csp \sim \kappa$ (by \eqref{shock:rs} and \eqref{h:bounds}), the characteristic gaps are directly proportional to $g(t)$, $1+g(t)$, and $1-g(t)$. Substituting these relationships into the bounds above (and reabsorbing constants into $\m$), we obtain that
\begin{align*}
 \big| \p_r b |_\s^-\big| &\le \tfrac{\m}{g(t)} \Big(\kappa^{\frac{\lambda}{\lambda-1}} + \tfrac{ |\dot{g}|}{ \sqrt{ \gamma g^2 - \alpha}} \Big), \\
 \big| \p_r w |_\s^-\big| &\le \tfrac{\m}{1+g(t)} \Big(\kappa^{\frac{\lambda}{\lambda-1}} + \tfrac{ |\dot{g}|}{ \sqrt{ \gamma g^2 - \alpha}} \Big), \\
 \big| \p_r z |_\s^-\big| &\le \tfrac{\m}{1-g(t)} \Big(\kappa^{\frac{\lambda}{\lambda-1}} + \tfrac{ |\dot{g}|}{ \sqrt{ \gamma g^2 - \alpha}} \Big).
\end{align*}

Finally, we translate these bounds into estimates for the Differentiated Riemann Variables (DRVs) defined in \eqref{ring:v}. Since $\rho$ is bounded (by the analysis in Lemma~\ref{lemma:om}), the DRVs inherit the scaling from the spatial derivatives $\p_r z, \p_r w, \p_r b$:
\begin{subequations}
\label{ring:v:s}
\begin{align}
|\rwsm| &\le \m\Big(\kappa^{\frac{\lambda}{\lambda-1}} + \tfrac{ |\dot{g}|}{ \sqrt{ \gamma g^2 - \alpha}}\Big), \label{ringw:v:s} \\
|\rbsm| &\le \m\Big(\kappa^{\frac{\lambda}{\lambda-1}} + \tfrac{|\dot{g}|}{ \sqrt{ \gamma g^2 - \alpha}}\Big), \label{ringb:v:s} \\
| \rzsm| &\le \tfrac{\m}{1-g(t)}\Big(\kappa^{\frac{\lambda}{\lambda-1}} + \tfrac{ |\dot{g}|}{ \sqrt{ \gamma g^2 - \alpha}}\Big). \label{ringz:v:s}
\end{align}
\end{subequations}
These inequalities reveal the central challenge and dictate the required behavior of the function $g(t)$. The behavior of all DRVs is strongly coupled to the rate of change $\dot{g}(t)$. Furthermore, the dominant variable $\rzsm$ has an additional singularity due to the vanishing denominator $1-g(t)$ as $t\to\ts$ (which corresponds to the vanishing Lax gap $\lambda_1^- - \dot\s$).

To ensure the solution exists in $\Omega^-$ and forms the desired preshock structure, we must design $g(t)$ such that the singularities remain controllable (specifically, integrable in time along characteristics). The construction (detailed in Proposition~\ref{prop:choice:g}) requires $1-g(t) \sim (t-\ts)^\varepsilon$, where $0 < \varepsilon < 1$. This necessarily implies that the time derivative has the singular behavior
\[
|\dot{g}(t)| \sim \varepsilon(t-\ts)^{\varepsilon-1} \,.
\]
Since $\varepsilon < 1$, $\dot{g}(t)$ blows up as $t\to\ts$. Substituting this scaling into \eqref{ring:v:s} reveals the following asymptotic behavior (noting that $\sqrt{\gamma g^2 - \alpha}$ approaches a positive constant as $g\to 1$):
\begin{align*} 
  &\text{ subdominant DRVs: }  |\rwsm|, |\rbsm| \sim |\dot{g}| \sim \varepsilon (t-\ts)^{\varepsilon-1} \,, \\
   & \text{ dominant DRV: } |\rzsm| \sim \tfrac{|\dot{g}|}{1-g} \sim \varepsilon (t-\ts)^{-1} \,.
\end{align*} 
Thus, \emph{all} DRVs blow up as $t\to\ts$, but the dominant variable $\rzsm$ exhibits a significantly stronger singularity. The analysis must focus on controlling this sharp $(t-\ts)^{-1}$ behavior. The bounds stated next in Proposition~\ref{prop:choice:g} (\cref{rz:g,rw:g,rq:g}) are consistent with this analysis, as they use the worst-case scaling $O((t-\ts)^{-1})$ conservatively for all three DRVs.

\begin{prop} \label{prop:choice:g}
For any sufficiently small $\varepsilon>0$, there exist $\kappa\ll1, \delta\ll1$ (depending on $\varepsilon$ and the Guderley profile), and a function $g:[\ts,\tf]\to\R$ satisfying \eqref{g:prop} such that the resulting interior DRV traces satisfy the bounds
\begin{subequations}
\begin{align}
| \rzsm(t) | &\le \tfrac{\m \varepsilon}{ t - \ts}, \label{rz:g} \\
| \rwsm(t) | &\le \tfrac{\m \varepsilon}{ t-\ts}, \label{rw:g} \\
| \rbsm(t) | &\le \tfrac{\m \varepsilon}{ t - \ts}. \label{rq:g}
\end{align}
\end{subequations}
Moreover, there are constants $\mathsf{C}_1,\mathsf{C}_2>0$ such that, as $t\to\ts^+$, the shock strength vanishes at a controlled rate given by
\begin{subequations}
\label{Taylor:choice:g}
\begin{align}
1-g(t) &= \mathsf{C}_1 (t-\ts)^{\varepsilon} +O\big((t-\ts)^{2\varepsilon}\big), \label{Taylor:g} \\
\ds(t) - \sp{\lambda_1}(t) &= \mathsf{C}_2 (t-\ts)^{\varepsilon} + O\big((t-\ts)^{2\varepsilon}\big). \label{Taylor:shock_strength}
\end{align}
\end{subequations}
\end{prop}

\begin{proof}[Proof of Proposition~\ref{prop:choice:g}]
The objective is to construct $g(t)$ such that the bound \eqref{rz:g} holds, based on the inequality \eqref{ringz:v:s}. This requires controlling the balance between the numerator (involving $|\dot{g}|$) and the denominator (involving $1-g$). We achieve this by constructing $g(t)$ as a solution to a specific ODE.

Let $\mu=\sqrt{\alpha/\gamma} = g(\tf)$. We introduce an auxiliary function $F(x)$ defined by integrating the differential form that precisely captures the singular structure in \eqref{ringz:v:s}:
\begin{equation*}
F(x) = \tfrac{1}{\sqrt{\gamma}} \int_{\mu}^x \tfrac{1}{ (1-y) \sqrt{y^2-\mu^2}}\, dy, \quad x\in [\mu, 1).
\end{equation*}
$F$ is smooth and strictly increasing on $(\mu, 1)$, mapping $[\mu, 1)$ to $[0, \infty)$. Its inverse $F^{-1}$ is smooth. The key differential identity is
\begin{equation} \label{ell:def}
\tfrac{d}{dt} \big(F(g(t))\big) = F'(g(t))\dot{g}(t) = \tfrac{ \dot{g}(t) }{ (1-g(t))\sqrt{ \gamma g(t)^2- \alpha}} \,.
\end{equation}

We now define $g(t)$ implicitly by prescribing the rate of this composite function. Let $\upnu$ be proportional to $\varepsilon$ (small if $\varepsilon$ is small). We choose a smooth, positive function $\ell(t)$ that behaves like the target rate $\upnu/(t-\ts)$ near $t=\ts$, but is $C^\infty$-flat near $t=\tf$ to ensure matching. Specifically, construct $\ell(t)$ such that
 \begin{align*}
\ell(t) &\le \tfrac{\upnu}{t-\ts}, \qquad \tfrac{d^k}{dt^k} \ell(\tf) =0 \ \ (k\ge0), \\
\ell(t) &= \tfrac{\upnu}{t-\ts} \quad \text{for } t \text{ close to } \ts.
\end{align*}

We define $g(t)$ by setting $\tfrac{d}{dt} F(g(t)) = -\ell(t)$ (since $\dot{g}<0$) and integrating from $t$ to $\tf$. Since $g(\tf)=\mu$ and $F(\mu)=0$, we have that
\[
F(g(t)) = \int_t^\tf \ell(s)\,ds \implies g(t) := F^{-1}\! \left( \int_t^\tf \ell(s)\,ds \right).
\]
By construction, this $g(t)$ satisfies \eqref{g:prop}.

We now verify the bounds. By the properties of $\ell(t)$ and the asymptotics of $F(x)$ near $x=1$, we can establish control over $1-g(t)$. Specifically, there exists a constant $C$ such that
\begin{equation} \label{1-g}
1 \le \tfrac{1-g(\tf)}{1-g(t)} \le \left( \tfrac{\delta}{t-\ts} \right)^{C\upnu}.
\end{equation}

We return to the bound for $\rzsm$ in \eqref{ringz:v:s} and substitute the defining relation \eqref{ell:def} (using $|\dot{g}|=-\dot{g}$):
\begin{align*}
| \rzsm(t)| &\le \tfrac{\m}{1-g(t)} \left(\kappa^{\frac{\lambda}{\lambda-1}} + \tfrac{ |\dot{g}(t)|}{ \sqrt{ \gamma g(t)^2 - \alpha}} \right)
= \m \left( \tfrac{\kappa^{\frac{\lambda}{\lambda-1}}}{1-g(t)} + \ell(t) \right).
\end{align*}
Using the bounds on $\ell(t)$ and $1-g(t)$ (from \eqref{1-g}), we find (for a new constant $\mathsf{B}$) that
\begin{align*}
| \rzsm(t)| \le \m \mathsf{B} \left( \tfrac{\kappa^{\frac{\lambda}{\lambda-1}} \delta^{C\upnu}}{(t-\ts)^{C\upnu}} + \tfrac{\upnu}{t-\ts}\right).
\end{align*}
We first choose $\upnu$ (and thus $\varepsilon$) small enough such that $C\upnu < 1$. The second term dominates near $t=\ts$. By subsequently choosing $\kappa$ and $\delta$ sufficiently small relative to $\varepsilon$, we ensure the first term is also bounded by $O(\tfrac{\varepsilon}{t-\ts})$. This establishes \eqref{rz:g}. The bounds \eqref{rw:g} and \eqref{rq:g} follow easily as they do not have the singular $(1-g)^{-1}$ factor (see \eqref{ringw:v:s} and \eqref{ringb:v:s}).

The Taylor expansions \eqref{Taylor:choice:g} follow from the precise behavior of $g(t)$ near $t=\ts$ derived from the ODE when $\ell(t)=\upnu/(t-\ts)$. The asymptotic inversion of $F$ near $1$ yields $1- g(t) \sim (t-\ts)^{\varepsilon}$ (identifying $\varepsilon$ with the leading order exponent related to $\upnu$). The expansion for the shock strength $\ds - \sp{\lambda_1} = (1-g)\csp$ follows immediately since $\csp(t)$ is $C^1[\ts,\tf]$.
\end{proof}

\begin{remark}[\textsl{Controlling shock strength via a Riccati mechanism}]
The construction presented in Proposition \ref{prop:choice:g} is the cornerstone of this section. It provides a mechanism to control the evolution of the shock strength, allowing us to connect a strong shock (the Guderley state at $t=\tf$) to a state of vanishing strength (the preshock at $t=\ts$). The key requirement driving this construction is the need to ensure the solvability of the Euler equations in the interior region $\Omega^-$.

The evolution of the dominant Differentiated Riemann Variable, $\rz$, is governed by a Riccati-type equation. When tracing $\rz$ backward along the fast acoustic characteristic $\psi(t,s)$ (defined later in \eqref{bw:flows:bt})\footnote{As will be detailed in \eqref{bw:flows:bt}, the time coordinate $t$ along the shock curve $\s(t)$ will be used as a ``label'' for data along the shock curve, and $s$ will be used to denote the evolutionary time coordinate.}, the equation takes the following approximate form (ignoring lower-order terms):
\begin{equation}
\label{riccati:ode}
\p_s ( \rz \circ \psi(t,s)) \approx -\tfrac{1+\alpha}{2}( \rz \circ \psi(t,s))^2, \quad \text{with terminal data } \rz \circ \psi(t,t) = \rzsm(t).
\end{equation}
We must integrate this backward in time $s$ from $s=t$ to $s=\ts$. The solution is
\begin{equation}\label{simple-rz}
\rz \circ \psi(t,s) = \tfrac{\rzsm(t)}{1 + \tfrac{1+\alpha}{2} \rzsm(t) (t-s)}.
\end{equation} 
To prevent this solution from blowing up before reaching $s=\ts$, we must ensure the denominator remains positive. Since the flow is compressive, we have $\rzsm(t) < 0$. Therefore, from \eqref{simple-rz},  the critical requirement is
\begin{equation} \label{remark:bound}
\big| \rzsm(t) \big| < \tfrac{2}{(1+\alpha)(t-\ts)}.
\end{equation}
This is the controllability condition that motivates the bounds required in Proposition \ref{prop:choice:g}, specifically $|\rzsm(t)| \le \tfrac{\m\varepsilon}{t-\ts}$ with $\varepsilon$ sufficiently small.

We now illustrate how this requirement dictates the dynamics of the shock strength. Heuristically, the spatial derivative $\rzsm$ is related to the time derivative of the trace $\p_t \zsm$ via the compatibility condition \eqref{pt:s:pr}. Ignoring lower-order terms, we can write this as
\begin{equation} \label{remark:prz}
\rzsm \sim \tfrac{\p_t \zsm}{\ds - \losm}.
\end{equation}
We established (in the derivation of \eqref{pt:usm}, \eqref{pt:csm:b}, and \eqref{second:h:bound}) that $|\p_t \zsm| \sim |\ddot{h}|$ (ignoring $\kappa$ scaling for simplicity). Furthermore, by Lemma \ref{lemma:rh:inversion}, $|\ds - \losm| \sim |\ds - \losp|$, which represents the shock strength. In terms of $\h(t)$, the shock strength is proportional to $|\dot{h} - (U-C)|$.

Combining these heuristics, the requirement \eqref{remark:bound} becomes
\begin{equation} \label{remark:frac}
\left| \rzsm \right| \sim \left| \tfrac{\ddot{h}}{\dot{h} - (U-C) } \right| \lesssim \tfrac{\varepsilon}{t-\ts}.
\end{equation}
If we approximate the profiles $U(\xi(t))$ and $C(\xi(t))$ by their values at the preshock $U_*-C_*$ (which equals $\dot{\h}(\ts)$ by \eqref{cond:h:ts}), this inequality can be written as
\begin{equation} \label{ode:main}
\left| \p_t \log| \dot{h}(t) - \dot{\h}(\ts) |\right| \lesssim \tfrac{\varepsilon}{t-\ts}.
\end{equation}
Integrating this differential inequality yields the required behavior for the shock strength
\[
| \dot{h}(t) - \dot{\h}(\ts) | \sim (t-\ts)^\varepsilon.
\]
This corresponds exactly to the behavior constructed in Proposition \ref{prop:choice:g} (\cref{Taylor:choice:g}). This mechanism—controlling the gradient blowup rate by solving an ODE for the shock trajectory—is what allows us to rigorously construct the transition from the weak to the strong shock regime.
\end{remark}

\subsection{Solving the Euler Equations in $\Omega^-$ via Bootstrap}

We aim to construct the solution to the Euler equations in the interior region
\[
\Omega^- = \{(r,t) : t\in[\ts,\tf], 0\le r < \s(t)\}.
\]
We solve the system backward in time, starting from the future temporal boundary $\Gamma^-$ (see Figure \ref{fig:step:1:outline}), defined as
\begin{equation}\label{Gammaminus}
\Gamma^- = \{(\s(t),t) \colon t\in[\ts,\tf] \} \cup \{ (r, \tf) \colon 0 \le r \le \s(\tf) \}.
\end{equation}

Formulated  as a boundary value problem for the Riemann variables (and density), we require $(w, z, b)$ and $\rho$ to solve
\begin{subequations}
\label{euler:rv:bp}
\begin{align}
\p_t w + \lambda_3 \p_r w + \tfrac{\alpha(d-1)}{4r}\,( w^2-z^2)&= \tfrac{1}{\gamma \alpha}\, b\p_r b\,\rho^{2\alpha}, \label{euler:rv:bp:w} \\
\p_t z + \lambda_1 \p_r z - \tfrac{\alpha(d-1)}{4r}\,( w^2-z^2)&= \tfrac{1}{\gamma \alpha}\, b\p_r b\,\rho^{2\alpha}, \label{euler:rv:bp:z} \\
\p_t b + \lambda_2 \p_r b &= 0, \label{euler:rv:bp:b} \\
\p_t \rho + \lambda_2 \p_r \rho &= - \rho \div u = - \rho \left( (d-1) \tfrac{w+z}{2r} + \p_r \tfrac{w+z}{2} \right). \label{euler:rv:bp:rho}
\end{align}
 Boundary conditions on $\Gamma^-$ consist of
\begin{align}
(w, z, b, \rho)(\s(t),t) &= (\wsm(t), \zsm(t), \bsm(t), \rsm(t)) \quad \forall t \in [\ts, \tf], \label{euler:rv:bp:bc_s} \\
(w, z, b)(r, \tf) &= (0, 0, 0), \quad \rho(r, \tf) = 1 \quad \text{for } r \in [0, \s(\tf)]. \label{euler:rv:bp:bc_tf}
\end{align}
\end{subequations}
The construction of $\s(t)$ via Proposition \ref{prop:choice:g} guarantees that the boundary data satisfies the crucial bounds established previously (combining \eqref{shock:bound}, \eqref{lower:bound:rsm}, and \eqref{rz:g}-\eqref{rq:g}). We let $\m, \mo$ denote fixed constants depending only on the Guderley profile and have the following bounds:
\begin{subequations}\label{boundary_data_bounds}
\begin{align}
\text{Riemann variables traces:} \quad &\big\| \sm{w}\big\|_\infty, \big\| \sm{z} \big\|_\infty, \big\| \sm{b} \big\|_\infty \le \m\kappa; \quad \mo \le \rsm(t) \le \m \,,  \\
\text{DRV traces:} \quad &\bigl| \sm{\rw } (t) \bigr|, \bigl|\sm{ \rz} (t) \bigr|, \bigl|\sm{ \rb} (t) \bigr| \le \tfrac{\m\varepsilon}{t-\ts} \,.
\end{align}
\end{subequations}
Furthermore, the $C^\infty$ matching at $t=\tf$ (ensured by \eqref{g:prop}) guarantees compatibility of the data to all orders at the corner $(\s(\tf), \tf)$.

\begin{lemma} \label{lemma:om}
If $\varepsilon>0$ is chosen sufficiently small (depending only on the Guderley profile constants $\m, \mo$), and $\kappa, \delta$ are subsequently chosen small enough (as required by Proposition \ref{prop:choice:g}), then the boundary value problem \eqref{euler:rv:bp} admits a unique solution in $\Omega^-$. The solution is $C^\infty$ on any compact subset of $\Omega^-$ bounded away from $t=\ts$. Moreover, the solution satisfies the improved bounds given by
\begin{subequations} \label{boot:tp}
\begin{align}
\| w\|_\infty(t), \| z\|_\infty(t), \| b \|_\infty(t) &\le 2\m \kappa , \label{boot:tp:wzb} \\
\| \rho \|_\infty(t) < 2\m, \qquad \big\| \tfrac{1}{\rho} \big\|_\infty(t) &< \tfrac{2}{\mo}, \label{boot:tp:rho} \\
\| \rw \|_\infty(t), \| \rz \|_\infty(t), \| \rb \|_\infty(t) &\le \tfrac{2\m \varepsilon}{ t- \ts}. \label{boot:tp:drv}
\end{align}
\end{subequations}
\end{lemma}
\begin{proof}[Proof of Lemma~\ref{lemma:om}]
We employ the method of characteristics and a continuation argument based on a bootstrap estimate.\footnote{In this proof, the constants $\m, \mo$ are fixed. The parameters $\varepsilon, \kappa, \delta$ are chosen small relative to these constants to close the argument.}

\paragraph{Step 1: Construction of the Characteristic Foliation.}
We define the backward-in-time characteristics emanating from $\Gamma^-$.

First, we define the characteristics emanating from the shock curve $(\s(t),t)$ for $t \in [\ts, \tf]$. We use $t$ as the label parameterizing the starting point and $s \in [\ts, t]$ as the running time along the trajectory. The trajectories $s \mapsto \big( \eta(t,s), \phi(t,s), \psi(t,s) \big) $ solve the ODEs:
\begin{subequations}
\label{bw:flows:bt}
\begin{align}
\p_s \eta(t,s) &= \lambda_3(\eta(t,s), s), \qquad \eta(t,t) = \s(t), \label{bw:flows:bt:eta}\\
\p_s \phi(t,s) &= \lambda_2(\phi(t,s), s), \qquad \phi(t,t) = \s(t), \label{bw:flows:bt:phi}\\
\p_s \psi(t,s) &= \lambda_1(\psi(t,s), s), \qquad \psi(t,t) = \s(t). \label{bw:flows:bt:psi}
\end{align}
\end{subequations}
The coordinate $t$ denotes a particle ``label'' for the data on the shock curve $\s(t)$, and $s \in [\ts,\tf]$ denotes the evolutionary time coordinate.

Next, we address the characteristics emanating from the time slice $\{ (r, \tf),\, 0 \le r \le \s(\tf) \}$. The data here is quiescent (given in \eqref{euler:rv:bp:bc_tf}). The speeds $\lambda_i$ are zero, and the solution remains quiescent in the rectangular region
\[
\mathcal{H} = \{(r,t): 0\le r \le \s(\tf),\ t\in[\ts,\tf]\}.
\]
The characteristics in $\mathcal{H}$ are vertical lines. The $C^\infty$ compatibility of the boundary data at the corner $(\s(\tf), \tf)$ ensures we can smoothly connect the characteristics defined in \eqref{bw:flows:bt} with those in $\mathcal{H}$, yielding a regular foliation of $\Omega^-$ (smooth away from $t=\ts$).

\paragraph{Step 2: Bootstrap Assumptions.}
To handle the singular behavior of the derivatives near $t=\ts$, we introduce localized control. We define the region $\mathcal{A}_T \subset \Omega^-$ as the region foliated by the slow acoustic characteristics $\eta$ emanating from the shock at times $t\ge T  > \ts$, defined as
\begin{equation*}
\mathcal{A}_T = \{ (x, s) \in \Omega^- : \exists\, t\in[T, \tf] \text{ such that } x = \eta(t, s),\ \ts \le s \le t \}.
\end{equation*}

We assume the solution exists on a time interval $(T_*, \tf]$ and satisfies the following weakened bootstrap bounds:
\begin{subequations}
\label{boot}
\begin{align}
\text{Riemann:} \quad &\| (w, z, b) \|_\infty(t) \le 4\m \kappa, \label{boot:wzb}\\
\text{Density:} \quad &\| \rho \|_\infty(t) < 4\m, \quad \big\| \tfrac{1}{\rho} \big\|_\infty(t) < \tfrac{4}{\mo}, \label{boot:rho} \\
\text{DRV (Global):} \quad &\| (\rw, \rz, \rb) \|_\infty(t) \le \tfrac{4\m \varepsilon}{ t- \ts}, \label{boot:drv} \\
\text{DRV ($\mathcal{A}_T$):} \quad &\max\{|\rw(x,t)|, |\rz(x,t)|, |\rb(x,t)|\} \le \tfrac{ 8\m \varepsilon}{T-\ts} \quad \text{for } (x,t) \in \mathcal{A}_T. \label{boot:drv:at}
\end{align}
\end{subequations}

\textit{Geometric Control:} Under assumptions \eqref{boot}, the speeds $|\lambda_i|$ are $O(\kappa)$. Since $\s(t) \ge \s(\tf) = \kappa^{\tfrac{\lambda}{1-\lambda}}$, if $\delta$ is small enough relative to $\kappa$ (guaranteed by Proposition \ref{prop:choice:g}), the characteristics remain bounded away from the origin, and we have that
\begin{equation}
\label{bound:nothing}
\eta(t,s), \phi(t,s), \psi(t,s) \ge \tfrac{1}{2}\kappa^{\tfrac{\lambda}{1-\lambda}} > 0.
\end{equation}
This ensures the geometric terms $\tfrac{1}{r} $ are regular.\footnote{The formulation of the Euler equations in radial symmetry introduces "geometric terms" involving $\tfrac{1}{r} $, arising from the divergence operator (e.g., $\div u = \p_r u + \frac{d-1}{r}u$). These appear explicitly in the system \eqref{euler:rv:bp}. Furthermore, the equations governing the Differentiated Riemann Variables (DRVs, Eq. \eqref{euler-DRV}), which are analyzed in Step 3, contain even stronger singularities involving $1/r^2$. Mathematically, these coefficients blow up at the origin $r=0$. The bootstrap argument relies fundamentally on integrating these equations along characteristics, a process that requires the coefficients and source terms to remain bounded. The Geometric Control established in Eq. \eqref{bound:nothing} guarantees that the entire domain $\Omega^-$ is strictly bounded away from the origin ($r \ge \tfrac{1}{2}\kappa^{\frac{\lambda}{1-\lambda}} > 0$). This ensures the geometric terms are "regular"—meaning bounded (e.g., $\tfrac{1}{r}  \le 2\kappa^{\frac{\lambda}{1-\lambda}}$)—throughout $\Omega^-$, thereby preventing the coordinate singularity at $r=0$ from interfering with the analysis.}

\paragraph{Step 3: Improving the Bounds on DRVs.}
We analyze the evolution of the DRVs using the system \eqref{euler-DRV}. We analyze $\rw$ along $\eta(t,s)$. The evolution equation \eqref{euler-DRV}a is given schematically by
\begin{equation*}
\p_s (\rw \circ \eta) = Q(\text{DRV}, \text{DRV}) + R(\text{Radial terms}).
\end{equation*}
We estimate the RHS using the bootstrap assumptions \eqref{boot}. The quadratic terms $Q$ are bounded by $C\m^2 \varepsilon^2/(s-\ts)^2$. The radial terms $R$ involve factors of $1/r \sim \kappa^{\frac{\lambda}{\lambda-1}}$ and $w, z \sim \kappa$. We utilize the smallness conditions on $\kappa, \delta$ to ensure $R$ are dominated by $Q$. Thus, there exists a constant $C_1$ (depending on $\gamma, d$) such that
\begin{equation}
\label{improv:rw_p_s}
|\p_s (\rw \circ \eta(t,s))| \le C_1\m^2 \tfrac{\varepsilon^2}{ (s- \ts)^2}.
\end{equation}
We integrate this backward from $s=t$. Using the boundary condition from \eqref{boundary_data_bounds}, we have that
\begin{align}
| \rw\circ\eta(t,s) | & \le | \rwsm(t)| + \int_s^t C_1\m^2 \tfrac{\varepsilon^2}{(s'-\ts)^2}\, ds' \le \tfrac{\m\varepsilon}{s-\ts} (1 + C_1\m\varepsilon). \label{improv:rw}
\end{align}
We choose $\varepsilon$ small enough such that $C_1\m\varepsilon \le 1$. This yields the improved global bound \eqref{boot:tp:drv}. The arguments for $\rz$ (along $\psi$) and $\rb$ (along $\phi$) are analogous.

\textit{Improving localized bounds in $\mathcal{A}_T$:} Consider $(x,s) \in \mathcal{A}_T$. This implies $(x,s)$ lies on some slow characteristic $\eta(t,s)$ where $t\ge T$. If $s\ge T$, the improved global bound suffices. If $s < T$, we analyze the DRVs along their respective characteristics starting from this point.

For $\rw$, we analyze the trajectory along $\eta(t,s')$. We integrate from $s$ to $T$. Since the trajectory remains in $\mathcal{A}_T$ for $s' \in [s, T]$, we use the localized bootstrap assumption \eqref{boot:drv:at} to bound the source terms. The RHS of the evolution equation (similar to \eqref{improv:rw_p_s}, but now using localized bounds) is bounded by $\tfrac{C_1\m^2 \varepsilon^2}{(T-\ts)^2}$, and we find that
\begin{equation} \label{improv:rw:at}
| \rw\circ\eta(t,s) | \le |\rw\circ\eta(t,T)| + \int_s^T C_1\m^2 \tfrac{\varepsilon^2}{(T-\ts)^2}\, ds'.
\end{equation}
Using the improved global bound at time $T$ (which is $\le \tfrac{2\m\varepsilon}{T-\ts}$) and $T-s < T-\ts$, we have that
\begin{equation}
| \rw\circ\eta(t,s) | \le \tfrac{2\m\varepsilon}{T-\ts} + C_1\m^2 \tfrac{\varepsilon^2}{T-\ts}.
\end{equation}
The arguments for improving the localized bounds for $\rz$ (integrated along $\psi$) and $\rb$ (integrated along $\phi$) are analogous. Due to the ordering of the characteristic speeds ($\lambda_1 < \lambda_2 < \lambda_3$), these trajectories also remain within the region $\mathcal{A}_T$ (which is bounded by the slowest characteristic $\eta$). This allows us to use the localized bootstrap assumptions to bound their respective source terms in the same manner.

Choosing $\varepsilon$ such that $C_1\m\varepsilon \le 2$ ensures the total for all DRVs is less than $\frac{4\m\varepsilon}{T-\ts}$, improving the localized bounds \eqref{boot:drv:at}.

\paragraph{Step 4: Improving the Bounds on Density (Inverse Flow Map Analysis).}
This is the most delicate step. We must control the integral of the divergence $\div u$ along the particle path $\phi(t,s)$, based on \eqref{euler:rv:bp:rho}. This requires analyzing the evolution of the characteristic labels.

We introduce the inverse flow map $\eta^{-1}(r,s)$, defined by the identity
\begin{equation}
\eta(\eta^{-1}(r,s), s) = r \,.
\label{eta-inv}
\end{equation}
Let $T(s) = \eta^{-1}(\phi(t,s), s)$ be the label of the slow characteristic passing through the particle path at time $s$.

We compute the derivative of $T(s)$. Differentiating the identity $\eta(T(s), s) = \phi(t,s)$ with respect to $s$ yields:
$\p_{t'} \eta \cdot \p_s T(s) + \p_s \eta = \p_s \phi$.
Let $J(t',s) = \p_{t'} \eta(t',s)$ be the Jacobian determinant. Since $\p_s \eta = \lambda_3$ and $\p_s \phi = \lambda_2$ (evaluated at the same spatial point), we have:
\begin{align}
J(T(s), s) \cdot \p_s T(s) = \lambda_2 - \lambda_3 = -\alpha\sigma. \label{p_s_T_raw}
\end{align}

We analyze the Jacobian $J$. At the boundary $s=t'$, $J(t',t') = \dot\s(t') - \lambda_3^-(t')$. By the Lax conditions \eqref{lax}, $\dot\s < \lambda_3^-$, so $J(t',t') < 0$. Since $J$ satisfies $\p_s J = J\, (\p_r \lambda_3\circ \eta)$, the sign is preserved, and $J(t,s) < 0$.

We rewrite \eqref{p_s_T_raw} as
\begin{align}
\p_s T(s) = \tfrac{\alpha\sigma(s)}{|J(T(s), s)|} \,. \label{p_s_T}
\end{align}
We must ensure $\p_s T(s) < 1$, which is equivalent to $|J| > \alpha\sigma$.

We analyze this ratio at the boundary $s=t'$. Let $R(t') = \frac{\alpha\sigma^-(t')}{|J(t',t')|}$. We have
$|J(t',t')| = \lambda_3^-(t') - \dot\s(t') = (u^- + \alpha\sigma^-) - \dot\s = v^- + \alpha\sigma^-$,
so $R(t') = \frac{\alpha\sigma^-}{v^- + \alpha\sigma^-}$. Since $v^->0$ (by RH conditions, Lemma 3.3), we have $R(t') < 1$.
This ratio is maximized in the strong shock limit (as analyzed in Section 3.4). The maximum value is $\frac{1}{1+\sqrt{\gamma/\alpha}}$. Since $\gamma>1$, we define $K^* = \sup_{t'\in[\ts,\tf]} R(t') < 1$.

Now we analyze the evolution away from the boundary using the DRV bounds. Let $t'=T(s)$. Using the improved DRV bounds (Step 3), the derivatives $|\p_r \lambda_3|$ and the terms governing the evolution of $\sigma$ (e.g., $\p_r u$) are bounded by $\frac{C'\m\varepsilon}{s-\ts}$. Let $\nu = C'\m\varepsilon$. Integrating the evolution equations for $J$ and $\sigma$ from $s$ to $t'$ yields
\begin{align*}
|J(s)| &\ge |J(t')| e^{-\big|\int_s^{t'} \p_r\lambda_3 ds''\big|} \ge |J(t')| \big(\tfrac{s-\ts}{t'-\ts}\big)^{\nu}, \\
\sigma(s) &\le \sigma(t') e^{|\int_s^{t'} O(\p_r u, \dots) ds''|} \le \sigma(t') \big(\tfrac{t'-\ts}{s-\ts}\big)^{\nu}.
\end{align*}
(We can absorb constants by adjusting $C'$).

Substituting these bounds into \eqref{p_s_T} shows that
\begin{align}
\p_s T(s) &\le \tfrac{\alpha\sigma(t')}{|J(t')|} \left(\tfrac{t'-\ts}{s-\ts}\right)^{2\nu} \le K^* \left(\tfrac{T(s)-\ts}{s-\ts}\right)^{2\nu}. \label{psT_diff_ineq}
\end{align}

We now integrate this differential inequality. Let $Y(s) = T(s)-\ts$ and $\mu=2\nu$. We choose $\varepsilon$ sufficiently small such that $\mu < 1$. We have $Y'(s) \le K^* Y(s)^\mu (s-\ts)^{-\mu}$. Integrating from $s$ to $t$ (where $Y(t)=t-\ts$) gives
\begin{align*}
\int_s^t Y(s')^{-\mu} Y'(s') ds' &\le K^* \int_s^t (s'-\ts)^{-\mu} ds' \\
\tfrac{1}{1-\mu} (Y(t)^{1-\mu} - Y(s)^{1-\mu}) &\le \tfrac{K^*}{1-\mu} ((t-\ts)^{1-\mu} - (s-\ts)^{1-\mu}).
\end{align*}
Rearranging the terms, we obtain that
\begin{align*}
Y(s)^{1-\mu} &\ge (1-K^*)(t-\ts)^{1-\mu} + K^* (s-\ts)^{1-\mu}.
\end{align*}
Since $K^* < 1$, this provides a positive lower bound. Let $C_{sep} = (1-K^*)^{1/(1-\mu)} > 0$. We obtain the crucial separation bound
\begin{equation}
\label{int:eta:phi}
T(s)-\ts \ge C_{sep} (t-\ts) \,.
\end{equation}

Now we estimate the divergence along $\phi(t,s)$. The point $(r,s')=\phi(t,s')$ belongs to the localized region $\mathcal{A}_{T(s')}$. We use the improved localized DRV bounds (Step 3) corresponding to $T(s')$, and we obtain that
\begin{equation*}
|\div u \circ \phi(t,s')| \le C\m\kappa^{\tfrac{\lambda-1}{\lambda}} + C \tfrac{\m\varepsilon}{T(s')-\ts} \,.
\end{equation*}
Using the lower bound \eqref{int:eta:phi}, we have that
\begin{equation*}
|\div u \circ \phi(t,s')| \le C\m\kappa^{\frac{\lambda-1}{\lambda}} + C \tfrac{\m\varepsilon}{C_{sep}(t-\ts)} \,. 
\end{equation*}
We integrate the density equation \eqref{euler:rv:bp:rho} along $\phi$, and determine that
\begin{equation*}
\left|\log\tfrac{\rsm(t)}{\rho(\phi(t,s), s)}\right| = \left|\int_s^t (\div u)\circ\phi(t,s') ds'\right| \le C\m\left(\kappa^{\frac{\lambda-1}{\lambda}}\delta + \tfrac{\m\varepsilon}{C_{sep}}\right).
\end{equation*}
By choosing $\varepsilon, \kappa, \delta$ sufficiently small (with $\varepsilon$ chosen first to ensure $\mu<1$ and fix $C_{sep}$), the exponent is small (e.g., $<\log 2$). Since $\rsm(t) \in [\mo, \m]$ (by \eqref{boundary_data_bounds}), this improves the density bounds \eqref{boot:rho} to the target bounds \eqref{boot:tp:rho}.

\paragraph{Step 5: Improving Riemann Variable Bounds and Conclusion.}
We integrate the equations for $w, z, b$. For $w$ along $\eta$, using \eqref{euler:rv:bp:w}, we have that
\begin{equation}
\p_s (w \circ \eta) = \text{Source Terms (Geometric + Entropy)}.
\label{eq:ps-W}
\end{equation}
The point $(\eta(t,s), s)$ is in $\mathcal{A}_t$. We use the improved localized bounds (Step 3) with $T=t$. The function $\p_r b$ is bounded by $\tfrac{C\m\varepsilon}{t-\ts}$. 
Using the bounds on Riemann variables and the geometric bound \eqref{bound:nothing}, the source terms are bounded by 
$C\m(\kappa^{\frac{\lambda}{\lambda-1}+2} + \tfrac{\kappa\varepsilon}{t-\ts})$. Integrating \eqref{eq:ps-W} from $s$ to $t$ shows that
\begin{equation} \label{int:w}
|w(\eta(t,s), s)| \le | \wsm(t) | + \int_s^t |\text{Sources}|\, ds' \le \m\kappa + C\m(\kappa^{\frac{\lambda}{\lambda-1}+2}\delta + \kappa\varepsilon) \,.
\end{equation}
Choosing $\varepsilon, \kappa, \delta$ small enough ensures $|w| \le 2\m\kappa$. The bounds for $z$ and $b$ follow similarly, improving \eqref{boot:wzb} to \eqref{boot:tp:wzb}.

We have improved all bootstrap assumptions \eqref{boot} to the target bounds \eqref{boot:tp}. By a standard continuation argument, the solution exists on the entire interval $[\ts, \tf]$.

\paragraph{Step 6: Higher order regularity.}
The $C^\infty$ regularity away from $t=\ts$ follows from standard theory. Since the boundary data is $C^\infty$ for $t>\ts$ (by Proposition \ref{prop:choice:g}) and the coefficients are smooth (as $r>0$ by \eqref{bound:nothing}), we can differentiate the equations \eqref{euler:rv:bp} and \eqref{euler-DRV} repeatedly. By induction, we obtain uniform bounds on higher derivatives on any compact interval $[T, \tf]$ where $T>\ts$.
\end{proof}

\subsection{Proof of Proposition~\ref{prop:large2weak}}
\begin{proof}
We synthesize the results established in this section to construct the global solution on $[\ts,\tf]$ and verify the properties claimed in Proposition~\ref{prop:large2weak}.

\textit{Construction of the Solution:}
We fix the parameters. First, we choose $\varepsilon>0$ sufficiently small (specifically, we require $\varepsilon < 1$) as needed for the bootstrap argument in Lemma~\ref{lemma:om}. Subsequently, we choose $\kappa, \delta$ and the function $g(t)$ according to Proposition~\ref{prop:choice:g}. Let $\s(t)$ be the corresponding shock trajectory obtained by solving the ODE \eqref{s-ode}. Let $(u^-, c^-, \rho^-)$ be the unique solution in the interior region $\Omega^-$ established by Lemma~\ref{lemma:om} (derived from the Riemann variables $w, z, b$).

We define the global solution $(u,c,\rho)$ (and the associated energy $E$) on $\R_+\times[\ts,\tf]$ by patching the interior solution with the exterior Guderley solution (where $\xi=r^\lambda/(-t)$):
\begin{equation*}
(u, c, \rho)(r,t) =
\begin{cases}
    \left(\tfrac{1}{\lambda} r^{1-\lambda} U(\xi),\, \tfrac{1}{\lambda} r^{1-\lambda} C(\xi),\, R(\xi)\right), & r > \s(t)\quad (\Omega^+), \\
    (u^-, c^-, \rho^-)(r,t), & 0 \le r < \s(t)\quad (\Omega^-).
\end{cases}
\end{equation*}

We now verify the items of Proposition~\ref{prop:large2weak}.

\textit{Item~\ref{prop:large2weak:1} (Regular shock and quiescent core).}
The construction defines a weak solution to the Euler equations \eqref{euler:md}. The equations hold in $\Omega^+$ (as the Guderley solution) and in $\Omega^-$ (by Lemma~\ref{lemma:om}). The Rankine--Hugoniot conditions are satisfied across the shock front $\mathcal{S}$ because the interior traces were defined precisely by the RH inversion formulas \eqref{RH:ex:h}. The Lax entropy conditions \eqref{lax} are satisfied due to the constraints imposed on $g(t)$ (verified in Proposition~\ref{prop:cond:h} and guaranteed by Lemma~\ref{lemma:rh:inversion}). The solution is $C^\infty$ away from $\mathcal{S}$ for $t>\ts$ (by Lemma~\ref{lemma:om}). The existence of the quiescent core $\mathcal{H}$ was established in the proof of Lemma~\ref{lemma:om} (Step 1).

Furthermore, the solution emanates from a preshock (Definition~\ref{def:regular:shock:pre}). The shock strength vanishes at $t=\ts$ (verified in Item 5 below). The continuity of the state variables $(u, \rho, E)$ up to $t=\ts$ is guaranteed by the existence result and the uniform bounds established in Lemma~\ref{lemma:om} (\cref{boot:tp:wzb,boot:tp:rho}).

\textit{Item~\ref{prop:large2weak:3} (Global agreement at $\tf$) and Item~\ref{prop:large2weak:4} (Shock path regularity).}
The global agreement at $t=\tf$ follows from the boundary conditions \eqref{euler:rv:bp:bc_tf} and the definition of the exterior solution. The $C^\infty$ matching of the shock trajectories $\s(t)$ and $\g(t)$ at $t=\tf$ was enforced by the properties of $g(t)$ (\cref{g:prop}) and verified in Proposition~\ref{prop:cond:h}. The regularity $\s\in C^\infty((\ts,\tf])\cap C^{1,\varepsilon}([\ts,\tf])$ follows from the asymptotics established in Proposition~\ref{prop:choice:g}.

\textit{Item~\ref{prop:large2weak:5} (Vanishing strength as $t\downarrow\ts$).}
The asymptotic behavior of the shock strength as $t\to\ts^+$ is determined by combining the engineered shock trajectory with the Rankine-Hugoniot (RH) conditions in the weak shock limit. We assume $0 < \varepsilon < 1$.

\begin{enumerate}
    \item \textsl{Speed Gap:} By the construction detailed in Proposition~\ref{prop:choice:g} (specifically \cref{Taylor:shock_strength}), the shock curve $\s(t)$ is designed such that the speed gap $\chi(t) = \dot{\s}(t)-\sp{\lambda_1}(t)$ possesses an expansion in powers of the strength parameter $(t-\ts)^\varepsilon$. This yields:
    \[
    \dot{\s}(t)-\sp{\lambda_1}(t) = C_1\,(t-\ts)^{\varepsilon}+O\!\big((t-\ts)^{2\varepsilon}\big),
    \]
    for some constant $C_1>0$.

    \item \textsl{Jumps across the shock:} We utilize the Taylor expansion of the RH jump conditions in the weak shock limit, established in Lemma~\ref{lemma:rh:inversion} (\cref{RH:taylor}). The jumps are related to the speed gap $\chi(t)$ by
    \begin{align*}
    \jump{z} &= -\tfrac{4}{1+\alpha} \chi(t) + O(\chi(t)^2)\,, \ \ 
    \jump{w} = O(\chi(t)^3)\,, \ \
    \jump{b} = O(\chi(t)^3) \,.
    \end{align*}
    Substituting the expansion for $\chi(t)$ yields the precise asymptotics. For the dominant Riemann variable $z$, we have that
    \begin{align*}
    \jump{z}(t) &= -\tfrac{4}{1+\alpha} \left(C_1\,(t-\ts)^{\varepsilon}+O\!\big((t-\ts)^{2\varepsilon}\big)\right) + O\big(((t-\ts)^\varepsilon)^2\big) 
    \\
    &
    = -\tfrac{4}{1+\alpha}C_1\,(t-\ts)^{\varepsilon}+O\!\big((t-\ts)^{2\varepsilon}\big).
    \end{align*}
    For the subdominant variables $w$ and $b$, 
    \[
    \jump{w}(t) = O\!\big((t-\ts)^{3\varepsilon}\big), \qquad \jump{b}(t) = O\!\big((t-\ts)^{3\varepsilon}\big) \,.
    \]
\end{enumerate}

\textit{Item~\ref{prop:large2weak:2} (Exact exterior matching and interior regularity at $\ts$).}
The exact exterior matching holds by definition. We analyze the regularity of the interior solution at the preshock time $t=\ts$. Let $r_*=\s(\ts)$. Since $\s(t) \ge \g(t)$ (by \eqref{cond:h:gtr:t}), the exterior solution is smooth up to $r_*^+$. We determine the interior regularity of the dominant variable $z(r,\ts)$ as $r\to r_*^-$.

We analyze the geometry of the fast acoustic characteristics $\psi(t,s)$ emanating backward-in-time from the shock (\cref{bw:flows:bt:psi}). We study the map $t \mapsto r(t) = \psi(t, \ts)$.

\textit{Spatial Scaling:} We analyze the asymptotics of $r(t)-r_*$ as $t\to\ts^+$. We have that
\begin{equation}\label{proof:prop31:r_expansion}
r(t) - r_* = (\s(t) - \s(\ts)) + (\psi(t, \ts) - \s(t)) \,.
\end{equation}
The second term is the integral of the speed along the characteristic, given as
\[
\psi(t, \ts) - \s(t) = \int_t^{\ts} \lambda_1(\psi(t,s'), s') ds' \,.
\]
Let $\lambda_* = \lambda_1(r_*, \ts)$. Since the solution is continuous at $\ts$ (Item 1), $\lambda_1$ is continuous at $(r_*, \ts)$. The deviation of the integrand from $\lambda_*$ is $O(|\psi(t,s')-r_*| + |s'-\ts|) = O(s'-\ts)$. Thus,
\[
\int_t^{\ts} \lambda_1(\psi(t,s'), s') ds' = \lambda_*(\ts-t) + O((t-\ts)^2) \,.
\]
For the shock displacement, we use the asymptotics from Proposition~\ref{prop:choice:g}. By the preshock condition, $\dot\s(\ts)=\lambda_*$. Integrating \eqref{Taylor:shock_strength} yields
\[
\s(t) - \s(\ts) = \lambda_*(t-\ts) + a(t-\ts)^{1+\varepsilon} + O((t-\ts)^{1+2\varepsilon}) \,,
\]
where $a = \tfrac{C_1}{1+\varepsilon} \neq 0$. Substituting these into \eqref{proof:prop31:r_expansion}, the linear terms cancel, and we obtain that
\begin{equation} \label{spatial_separation_proof}
r(t) - r_* = a(t-\ts)^{1+\varepsilon} + O((t-\ts)^{\min(2, 1+2\varepsilon)}) \,.
\end{equation}
Since $\varepsilon < 1$, the leading order spatial scaling is $|r(t)-r_*| \sim (t-\ts)^{1+\varepsilon}$.

\textit{Value Scaling\footnote{The term "value scaling" refers to the asymptotic behavior of the function's magnitude (its value) as the singularity is approached along a specific path.}:} We analyze the behavior of $z(r(t), \ts)$. With $z_* = z(r_*, \ts)$, we have that
\begin{equation}\label{value_identity_proof}
z(r(t), \ts) - z_* = (\zsm(t) - z_*) + (z(r(t), \ts) - \zsm(t)) \,.
\end{equation}
The first term is the variation of the trace along the shock. By Item 5, $\jump{z}(t) \sim (t-\ts)^\varepsilon$. The exterior trace $\zsp(t)$ is smooth, so $\zsp(t) - z_* = O(t-\ts)$. Since $\varepsilon<1$, the jump dominates, and the interior trace satisfies
\begin{equation}\label{z_trace_variation_proof}
\zsm(t) - z_* \sim (t-\ts)^\varepsilon \,.
\end{equation}
The second term in \eqref{value_identity_proof} is the integral of the source terms $S_z$ (from \eqref{euler:rv:bp:z}) along the characteristic $\psi(t,s')$ from $t$ to $\ts$.

A detailed analysis confirms that this integral is subdominant to the trace variation. Although the global bounds on the DRVs in Lemma~\ref{lemma:om} (
Step 2) allow the source terms (specifically the entropy gradient) to blow up like $\tfrac{1}{s'-\ts}$ (which would lead to a logarithmic divergence if 
integrated directly) the  bootstrap analysis within Lemma~\ref{lemma:om} carefully balances this blow-up. 
By utilizing the localized bounds (Step 3) 
and the analysis of characteristic separation (Step 4), 
the analysis ensures that the integral term vanishes as $t\to\ts$ 
(establishing the continuity claimed in Item 1) and is consistent with the leading order behavior given as
\begin{equation}\label{z_variation_proof}
z(r(t), \ts) - z_* \sim (t-\ts)^\varepsilon \,.
\end{equation}

\textit{Combined Regularity:} We combine the spatial scaling \eqref{spatial_separation_proof} and the value scaling \eqref{z_variation_proof}. 
We invert the spatial relationship to get $(t-\ts) \sim |r-r_*|^{\frac{1}{1+\varepsilon}}$. 
Substituting this into the value scaling yields the spatial regularity
\[
|z(r, \ts) - z_*| \sim \big(|r-r_*|^{\frac{1}{1+\varepsilon}}\big)^\varepsilon = |r-r_*|^{\frac{\varepsilon}{1+\varepsilon}} \,.
\]
This confirms the interior solution has precisely $C^{\frac{\varepsilon}{1+\varepsilon}}$ regularity at the preshock point, concluding the proof.
\end{proof}

\section{Backwards-in-time shock development: from weak shock to the preshock} \label{sec:weak2pre}

The construction in Section~\ref{sec:large2weak} successfully evolved a strong Guderley shock backward-in-time to a state of vanishing strength at $t=\ts$. However, as detailed in the proof of Proposition~\ref{prop:large2weak} (Item~\ref{prop:large2weak:2}), the resulting state at $t=\ts$ is highly asymmetric: it is smooth on the exterior but only $C^{\frac{\varepsilon}{1+\varepsilon}}$ Hölder continuous on the interior. This imbalance stems from the fact that the shock acceleration $\ddot{\s}(t)$ necessarily becomes singular as $t\searrow\ts$ (since $\varepsilon<1$ in Proposition~\ref{prop:choice:g}). Such a state cannot be generated from $C^1$ initial data.

Our goal in this section is to resolve this obstruction by refining the construction on a short time interval $[\ts,\tsh]$, where the shock is already weak. We replace the singular trajectory $\s(t)$ with a new, $C^\infty$-smooth shock curve $\l(t)$. This new curve is engineered to produce a physically realistic preshock: a symmetric $C^{\frac{1}{3}}$ cusp in the dominant variable $z$, with $C^{1,\frac{1}{3}}$ regularity for the subdominant variables $w$ and $b$. This construction corresponds to Step 2 outlined in Section~\ref{strategy:3}.

\paragraph{Roadmap.}
\begin{enumerate}
    \item We define the spacetime regions and the characteristic framework for a Goursat-type problem (see Figures~\ref{fig:step:2:a:outline} and \ref{fig:step:2:b:outline}).
    \item We introduce the concept of an \emph{admissible preshock-generating pair} $(\l, \lambda_1|_\l^+)$, defining the new shock curve and the exterior fast wave speed trace.
    \item We prove the existence of such a pair (Lemma~\ref{lemma:l1:l}).
    \item We use this pair to construct a unique patching solution on $[\ts,\tsh]$ that matches the solution from Section 3 at $t=\tsh$ and develops the desired $C^{\frac{1}{3}}$ preshock at $t=\ts$ (Proposition~\ref{prop:weak2pre}).
\end{enumerate}

\paragraph{The Weak-Shock Regime and the Target Behavior.}
We fix an intermediate time $\tsh \in (\ts,\tf)$ close to $\ts$, and set $\dsh:=\tsh-\ts$. At this time, the solution constructed in Proposition~\ref{prop:large2weak} is weak, with shock strength $\dot{\s}(\tsh)-\sp{\lambda_1}(\tsh)\sim \dsh^{\varepsilon}$.

We now introduce the new shock curve $\l(t)$ for $t\in[\ts,\tsh]$, matching $\s(t)$ at $t=\tsh$. The key requirement for generating a $C^{\frac{1}{3}}$ preshock (as formalized in Lemma~\ref{lemma:expansion:0} below) is that the shock speed approaches the characteristic speed at a specific, faster rate:
\[
\dot{\l}(t)-\lambda_1\big(\l(t)^+,t\big)\sim (t-\ts)^{\frac{1}{2}}\quad\text{as }t\searrow\ts.
\]
 Our construction will ensure that  the shock acceleration $\ddot{\l}(t)$ remains bounded, avoiding the singularity encountered in Section 3.

\paragraph{Characteristic Framework and Regions.}
We set up a Goursat problem by prescribing data along characteristics emanating backward-in-time from the handover point $(\s(\tsh),\tsh)$.

Since the flow is discontinuous at the shock, we must distinguish between characteristics generated by the exterior and interior states. Let $\psi^+(t)$ be the \emph{exterior} fast acoustic characteristic (governed by the $\lambda_1^+$-flow) and $\eta^-(t)$ be the \emph{interior} slow acoustic characteristic (governed by the $\lambda_3^-$-flow). Both are initialized at $t=\tsh$ such that $\psi^+(\tsh)=\eta^-(\tsh)=\s(\tsh)$.

Based on the characteristic speeds and the Lax conditions (which imply $\lambda_3^- < \dot{\l} < \lambda_1^+$ near $\tsh$), the backward-in-time geometry satisfies $\eta^-(t) < \l(t) < \psi^+(t)$ for $t \in [\ts, \tsh)$. We define the two triangular spacetime regions (consistent with Figures~\ref{fig:step:2:a:outline} and \ref{fig:step:2:b:outline}):
\begin{itemize}
    \item The exterior patching region (Right side):
    \[
    \mathcal{D}=\big\{(r,t):\,\ts\le t\le\tsh,\ \l(t)\le r\le \psi^+(t)\big\}.
    \]
    \item The interior patching region (Left side):
    \[
    \mathcal{L}=\big\{(r,t):\,\ts\le t\le\tsh,\ \eta^-(t)\le r\le \l(t)\big\}.
    \]
\end{itemize}

The construction proceeds as follows:
\begin{enumerate}
    \item In the exterior region $\mathcal{D}$, we solve a Goursat problem. We prescribe the subdominant variables $(w, b)$ along the characteristic $\psi^+(t)$ (inherited from the Guderley solution) and the dominant wave speed $\lambda_1$ along the shock curve $\l(t)$.
    \item The solution in $\mathcal{D}$ determines the full exterior state $(w, z, b)|_\l^+$.
    \item We use the Rankine-Hugoniot conditions across $\l(t)$ to determine the interior state $(w, z, b)|_\l^-$.
    \item In the interior region $\mathcal{L}$, we solve the Euler equations using the traces $(w, z, b)|_\l^-$ as boundary data.
\end{enumerate}

The crucial element is the simultaneous choice of $\l(t)$ and $\lambda_1|_\l^+(t)$, formalized below. Throughout, we treat $\varepsilon$ and $\kappa$ (from Section 3) as fixed; implicit constants depend only on the Guderley profiles. We use $\sp{\cdot}$ for traces along the original shock $r=\s(t)$.

\begin{definition}[Admissible preshock–generating pair]\label{def:admissible:pair}
We call a pair $\big(\l,\lambda_1\big|_{\l}^{+}\big)$ on $[\ts,\tsh]$ \emph{admissible} if $\lambda_1\big|_{\l}^{+}\in C^\infty((\ts,\tsh])$ and $\l\in C^\infty([\ts,\tsh])$ satisfy the following properties:
\begin{enumerate}
\item \textup{Matching at $t=\tsh$:} The pair matches the previous construction smoothly. For all $k\ge0$,
\[
\partial_t^{\,k}\l(\tsh)=\partial_t^{\,k}\s(\tsh), \qquad
\partial_t^{\,k}\!\big(\lambda_1\big|_{\l}^{+}\big)(\tsh)=\partial_t^{\,k}\!\big(\sp{\lambda_1}\big)(\tsh).
\]
\item \textup{Time–derivative bounds:} The time variation of the exterior speed is controlled. For $t\in[\ts,\tsh]$,
\[
\partial_t\!\big(\lambda_1\big|_{\l}^{+}\big)(t) \le \m,
\qquad
\bigl|\partial_t\!\big(\lambda_1\big|_{\l}^{+}\big)(t)\bigr|\le \m\,\dsh^{\,\varepsilon-\tfrac{1}{2}}\,(t-\ts)^{-\frac{1}{2}}.
\]
(The second bound allows for a singularity at $\ts$ consistent with the $\sqrt{t-\ts}$ behavior).

\item \textup{Positive speed gap (Target Behavior):} The shock strength vanishes at the desired rate, consistent with the matching at $\tsh$. For $t\in[\ts,\tsh]$,
\begin{equation}\label{speedgap}
\dot{\l}(t)-\lambda_1\big|_{\l}^{+}(t)\ \ge\ \tfrac{1}{\m}\,\dsh^{\,\varepsilon-\tfrac{1}{2}}\,(t-\ts)^{\frac{1}{2}}.
\end{equation}
(The scaling factor $\dsh^{\varepsilon-\frac{1}{2}}$ ensures that at $t=\tsh$, the gap is $\sim \dsh^\varepsilon$, matching the previous construction).

\item \textup{Acceleration bound:} The shock acceleration remains bounded (unlike $\ddot\s(t)$). For $t\in[\ts,\tsh]$, we have $\ |\ddot{\l}(t)| \le \m\,\dsh^{\,\varepsilon-1}$.

\item \textup{Taylor expansions at $t=\ts$:} The pair ensures a preshock $\lambda_1\big|_{\l}^{+}(\ts)=\dot{\l}(\ts)$ and possesses the required asymptotic structure. There exist constants $\mathsf{l}_1, \mathsf{l}_2$ such that
\begin{subequations}
\begin{align}
&\lambda_1\big|_{\l}^{+}(t)-\lambda_1\big|_{\l}^{+}(\ts)-\mathsf{l}_1\,(t-\ts)^{\frac{1}{2}}-\mathsf{l}_2\,(t-\ts) =O\!\big(|t-\ts|^{\frac{3}{2}}\big) \,,
\label{Taylor-lambda1} \\
&\dot{\l}(t)-\dot{\l}(\ts)-\ddot{\l}(\ts)\,(t-\ts)=O\!\big(|t-\ts|^{2}\big) \,. \label{Taylor-dotc}
\end{align}
\end{subequations}
\end{enumerate}
\end{definition}

\begin{remark}[Terminology]
We refer to $\big(\l,\lambda_1|_{\l}^{+}\big)$ as a ``preshock–generating'' pair because the specific $\sqrt{t-\ts}$ approach enforced by the speed gap \eqref{speedgap} is the mechanism that generates the $C^{{\frac{1}{3}}}$ cusp for $z$ at the preshock point $(r_*,\ts)$, where $r_*=\l(\ts)$ (cf. Lemma~\ref{lemma:expansion:0}).
\end{remark}

\begin{lemma}[Construction of an admissible pair]\label{lemma:l1:l}
There exists an admissible pair $\big(\l,\lambda_1\big|_{\l}^{+}\big)$ on $[\ts,\tsh]$ in the sense of Definition~\ref{def:admissible:pair}.
\end{lemma}

\begin{proof}[Proof of Lemma~\ref{lemma:l1:l}]
We construct the pair using a smooth partition of unity to blend the required behavior near $\ts$ with the matching conditions near $\tsh$.

Let $\phi\in C^\infty([\ts,\tsh])$ be a smooth cut-off function such that
\begin{equation} \label{flatness}
 \phi\equiv1 \text{ near } \ts, \quad \phi\equiv 0 \text{ near } \tsh, \quad \text{and } \partial_t^{\,k}\phi(\tsh)=0 \text{ for all } k\ge0.
\end{equation}

\textit{Construction of $\lambda_1|_\l^+$:} We define the trace by blending a Taylor expansion near $\ts$ (Item 5) with the known trace $\sp{\lambda_1}(t)$ from Section \ref{sec:large2weak}:
\begin{equation} \label{def:l1:c}
\lambda_1\big|_{\l}^{+}(t)
:=\ \phi(t)\Big(\Lambda_*+\mathsf{l}_1\,(t-\ts)^{\frac{1}{2}}+\mathsf{l}_2\,(t-\ts)\Big)
+\bigl(1-\phi(t)\bigr)\,\sp{\lambda_1}(t).
\end{equation}
Here $\Lambda_*$ is the speed at the preshock (determined later). We will choose constants $\mathsf{l}_1$ and $\mathsf{l}_2$ shortly, specifically requiring $\mathsf{l}_1 < 0$. This construction ensures $C^\infty$ matching at $\tsh$ (Item 1) and the Taylor expansion \eqref{Taylor-lambda1} at $\ts$ (Item 5).

\textit{Construction of $\l(t)$:} We define the shock speed $\dot\l(t)$ by first defining the speed gap $\chi(t) = \dot\l(t) - \lambda_1\big|_{\l}^{+}(t)$. We construct $\chi(t)$ by blending the target behavior $\sim (t-\ts)^{1/2}$ with the known gap from Section 3, $\dot\s(t)-\sp{\lambda_1}(t)$:
\[
\chi(t):=\phi(t)\left(\mathsf{c}_1\,(t-\ts)^{\frac{1}{2}}\right) + (1-\phi(t))\left(\dot\s(t)-\sp{\lambda_1}(t)\right).
\]

We define the shock curve by solving the ODE
\begin{equation} \label{positive:gap}
\dot{\l}(t)=\lambda_1\big|_{\l}^{+}(t)+\chi(t),
\qquad
\l(\tsh)= \s(\tsh).
\end{equation}

Crucially, to ensure the shock curve $\l(t)$ is $C^\infty$, we first show that $\dot{\l}(t)$ is $C^1$. Substituting the definitions of $\lambda_1\big|_{\l}^{+}$ and $\chi(t)$ into \eqref{positive:gap} yields
\begin{align*}
\dot{\l}(t) &= \phi(t)\Big(\Lambda_*+(\mathsf{l}_1+\mathsf{c}_1)\,(t-\ts)^{\frac{1}{2}}+\mathsf{l}_2\,(t-\ts)\Big) + (1-\phi(t))\dot\s(t).
\end{align*}
To eliminate the fractional power $(t-\ts)^{1/2}$ and guarantee smoothness at $t=\ts$, we must impose the condition $\mathsf{c}_1 = -\mathsf{l}_1$.

Furthermore, to satisfy the positive speed gap requirement (Item 3), we need $\chi(t)>0$. Near $t=\ts$, $\chi(t) \approx \mathsf{c}_1 (t-\ts)^{1/2}$, so we require $\mathsf{c}_1 > 0$. Combined with the smoothness condition, this confirms the necessity of choosing $\mathsf{l}_1 < 0$.

This construction ensures $C^\infty$ matching of $\l$ to $\s$ at $\tsh$ (Item 1). Since $\chi(\ts)=0$ (as $\dot\s(\ts)=\sp{\lambda_1}(\ts)$ by Proposition \ref{prop:choice:g}), we have $\dot{\l}(\ts)=\lambda_1\big|_{\l}^{+}(\ts)$. We set $\Lambda_* = \dot\l(\ts)$.

\textit{Verification of Bounds:} We choose the constants $\mathsf{l}_1, \mathsf{l}_2$ (with $\mathsf{c}_1 = -\mathsf{l}_1$) to be comparable to the corresponding derivatives/values of $\s(t)$ at $\tsh$. Based on the scaling in Proposition~\ref{prop:choice:g}, we choose
\[
|\mathsf{l}_1| \sim \dsh^{\varepsilon-\frac{1}{2}}, \quad |\mathsf{l}_2| \sim \dsh^{\varepsilon-1}.
\]
With these choices (ensuring $\mathsf{l}_1$ is negative and appropriately scaled), the positive speed gap \eqref{speedgap} (Item 3) follows from the definition of $\chi(t)$. The derivative bounds (Item 2) and the acceleration bound (Item 4) follow by differentiating \eqref{def:l1:c} and \eqref{positive:gap}, utilizing the properties of $\phi(t)$ and the known bounds on $\s(t)$ and $\sp{\lambda_1}(t)$, potentially after shrinking $\dsh$.
\end{proof}

We now state the main result of this section, which establishes the existence of the desired patching solution.

\begin{prop}[\textsl{From a weak shock to a well‑posed preshock}]
\label{prop:weak2pre}
Let $(u_L,\rho_L,E_L)$ be the regular shock solution of Proposition~\ref{prop:large2weak} on $[\ts, \tf]$, and fix $\tsh\in(\ts,\tf)$ sufficiently close to $\ts$. Let $\big(\l,\lambda_1\big|_{\l}^{+}\big)$ be the admissible pair provided by Lemma~\ref{lemma:l1:l}.
Then there exists a \emph{unique} radial “patching’’ solution $(u,\rho,E)$ on $[\ts,\tsh]$, associated to this fixed pair, with shock front $r=\l(t)$ and the following properties:
\begin{enumerate}
\item \textup{Seamless matching at $t=\tsh$.} \label{prop:weak2pre:item1}
We have $(u,\rho,E)(\cdot,\tsh)=(u_L,\rho_L,E_L)(\cdot,\tsh)$ and the shock curves match smoothly: $\partial_t^{k}\l(\tsh)=\partial_t^{k}\s(\tsh)$ for all $k\ge0$.

\item \textup{Characteristic construction (Goursat data).} \label{prop:weak2pre:item2}
Let $w_\circ(t)$ and $b_\circ(t)$ denote the traces of the solution $(u_L, \rho_L, E_L)$ (specifically, its associated Riemann variables $w_L, b_L$) along the exterior characteristic $\psi^+(t)$.
On $[\ts,\tsh]$, the constructed solution is the unique Euler flow satisfying the Goursat data:
\[
w\big(\psi^+(t),t\big)=w_\circ(t), \qquad
b\big(\psi^+(t),t\big)=b_\circ(t), \qquad
\lambda_1\big(\l(t)^+,t\big)=\lambda_1\big|_{\l}^{+}(t).
\]
The Rankine–Hugoniot conditions hold across $r=\l(t)$.

\item \textup{Regular shock evolution on $(\ts,\tsh]$.} \label{prop:weak2pre:item3}
The solution $(u,\rho,E)$ with shock front $\mathcal{S}_\l=\{(x,t):\,|x|=\l(t)\}$ is a regular shock emanating from a preshock (Definitions~\ref{def:regular:shock} and \ref{def:regular:shock:pre}). The flow is $C^\infty$ away from $\mathcal{S}_\l$ (for $t>\ts$) and maintains the quiescent core $\mathcal{H}$.

\item \textup{Controlled approach to the preshock.} \label{prop:weak2pre:item4}
As $t\searrow\ts$, the shock strength vanishes according to the prescribed rate:
\[
\dot{\l}(t)-\lambda_1\big|_{\l}^{+}(t)=C_1\,(t-\ts)^{\frac{1}{2}}+O\big(|t-\ts|\big).
\]
Consequently, the jumps behave as: $\jump{z}(t)\sim (t-\ts)^{\frac{1}{2}}$, and $\jump{w}(t), \jump{b}(t) \sim (t-\ts)^{\frac{3}{2}}$.

\item \textup{Preshock profile at $r_*=\l(\ts)$.} \label{prop:weak2pre:item5}
At $t=\ts$, the solution develops the target singularity structure. The dominant Riemann variable $z$ forms a $C^{\frac{1}{3}}$ cusp, while $w$ and $b$ retain $C^{1,{\frac{1}{3}}}$ regularity. Specifically, there exist constants $\mathsf{a}, \mathsf{b}, \dots$ such that:
\begin{align*}
z(r,\ts)&=z(r_*,\ts)+\mathsf{a}\,(r-r_*)^{\frac{1}{3}}+\mathsf{b}\,(r-r_*)^{\frac{2}{3}}+O\!\big(|r-r_*|\big).
\end{align*}
(The detailed expansions for $w$ and $b$, which exhibit slight asymmetry in the higher-order $C^{1, {\frac{1}{3}} }$ terms, are established in Lemma~\ref{lemma:expansion:0} below).
\end{enumerate}
Moreover, outside the local patching regions $\mathcal{D}$ and $\mathcal{L}$, the solution agrees with $(u_L,\rho_L,E_L)$.
\end{prop}

The proof of this proposition occupies the remainder of this section. We will first solve the Goursat problem in the exterior region $\mathcal{D}$ and then solve the boundary value problem in the interior region $\mathcal{L}$.

\subsection{The Goursat Problem in the Exterior Region $\mathcal{D}$}

We now detail the construction of the solution in the exterior patching region $\mathcal{D}$. This involves solving a Goursat problem (a characteristic boundary value problem) using the admissible preshock–generating pair $\big(\l,\lambda_1\big|_{\l}^{+}\big)$ from Lemma~\ref{lemma:l1:l}.

\paragraph{Problem Setup and Domain.}
We aim to solve the Euler equations \eqref{euler:rv} in the exterior region (see Figure~\ref{fig:step:2:a:outline}). Note that the definition of $\mathcal{D}$ here corresponds to the exterior region defined in the roadmap:
\[
\mathcal{D}=\big\{(r,t):\,\ts\le t\le\tsh,\ \l(t)\le r\le \psi(\tsh,t)\big\}.
\]
The region is bounded by the shock curve $r=\l(t)$ and the incoming fast characteristic $r=\psi(\tsh,t)$. The Goursat data is prescribed as follows:
\begin{enumerate}
    \item Along the characteristic boundary $\psi(\tsh,\cdot)$: We prescribe the subdominant variables $w$ and the specific entropy $S=2\log b$. These traces, denoted $w_\circ(t)$ and $S_\circ(t)$, are inherited from the Guderley solution constructed in Section 3.
    \item Along the shock boundary $\l(t)$: We prescribe the dominant characteristic speed $\lambda_1\big|_{\l}^{+}(t)$, as defined by the admissible pair.
\end{enumerate}

We utilize the formulation of the Euler system using $(w,z,S)$. We recall the equations \eqref{euler:rv} and use the identity $\tfrac{ 1}{\alpha\gamma }\rho^{2\alpha}b\p_r b = \tfrac{\alpha}{2\gamma}\sigma^2 \p_r S$ to write \eqref{euler:rv} as
\begin{subequations}
\label{euler:rv:bp:S}
\begin{align}
(\partial_t+\lambda_1\,\partial_r)z&=\tfrac{\alpha}{2\gamma}\,\sigma^2 \,\partial_r S + \tfrac{\alpha(d-1)}{4r}\,(w^2-z^2), \label{eulerz}\\
(\partial_t+\lambda_3\,\partial_r)w&=\tfrac{\alpha}{2\gamma}\,\sigma^2 \,\partial_r S - \tfrac{\alpha(d-1)}{4r}\,(w^2-z^2), \label{eulerw}\\
(\partial_t+\lambda_2\,\partial_r)S&=0, \label{eulerS}
\end{align}
\end{subequations}
where $\sigma=\tfrac{1}{2}(w-z)$.

\paragraph{Transformation to Characteristic Coordinates (Pull-back).}
To analyze the behavior near the preshock, where the characteristics converge, we transform the system into coordinates aligned with the fast acoustic flow.

We define the backward-in-time fast flow map $\psi(t,s)$. Here $t\in[\ts, \tsh]$ labels the characteristic that reaches the shock at time $t$, and $s\in[\ts, t]$ is the
evolutionary time along the trajectory
\begin{align*}
\partial_s \psi(t,s) = \lambda_1\!\big(\psi(t,s),s\big),\qquad \psi(t,t)=\l(t).
\end{align*}
(Note, that we are using the evolution time $s\in[\ts, t]$ here in the same manner that it was used in \eqref{bw:flows:bt}.)
This change of variables transforms the physical domain $\mathcal{D}$ into the fixed triangle:
\[
\mathcal{T}=\{(t,s):\ \ts\le s\le t\le \tsh\}.
\]
The boundaries of $\mathcal{T}$ are the shock edge $\mathcal{T}^{\l}=\{(t,t)\}$, the inflow edge $\mathcal{T}^{\psi}=\{(\tsh,s)\}$, and the corner edge $\mathcal{T}^{*}=\{(t,\ts)\}$.

We define the Jacobian of this transformation by
\begin{equation*}
J(t,s)=\partial_t\psi(t,s).
\end{equation*}

We adopt the \emph{capitalization-of-variables} notation for the pull-backs (composition with the map $(t,s) \mapsto (\psi(t,s), s)$):
\begin{equation}\label{ring:def}
\begin{gathered}
Z=z\circ\psi,\quad W=w\circ\psi,\quad \mathsf{S}=S\circ\psi,\quad \Sigma=\sigma\circ\psi, \\
\rZ=\left(\partial_r z +\tfrac{\alpha}{2\gamma}\sigma \p_r S\right)\circ\psi,\quad \rW=\left(\partial_r w -\tfrac{\alpha}{2\gamma}\sigma \p_r S\right)\circ\psi,\quad \rS=\partial_r S\circ\psi.
\end{gathered}
\end{equation}
The derivatives transform according to the chain rule: $\partial_s = (\partial_t + \lambda_1 \partial_r)\circ\psi$ (the $L_1$ derivative) and $\partial_t = J \cdot (\partial_r)\circ\psi$.

The $L_3$ derivative transforms as $L_3 = L_1 + (\lambda_3-\lambda_1)\partial_r$. Since $\lambda_3-\lambda_1 = 2\alpha\sigma$ (from Section 1.4), we have $L_3 = \partial_s + (2\alpha\Sigma) J^{-1}\partial_t$. The $L_2$ derivative transforms similarly: $L_2 = \partial_s + (\alpha\Sigma) J^{-1}\partial_t$.

Applying these rules to \eqref{euler:rv:bp:S} and the associated derivative equations yields the pulled-back system on $\mathcal{T}$:
\begin{subequations}\label{euler:bp:lag}
\begin{align}
\partial_s Z &= \tfrac{\alpha}{2\gamma}\,\Sigma^2 \rS + \tfrac{\alpha(d-1)}{4\psi}\,(W^2-Z^2), \label{eq:Z:sl}\\
J\,\partial_s W &= -J\left(\tfrac{\alpha}{2\gamma}\,\Sigma^2 \rS - \tfrac{\alpha(d-1)}{4\psi}\,(W^2-Z^2)\right) - 2\alpha \Sigma\,\partial_t W, \label{eq:W:sl}\\
J\,\partial_s \mathsf{S} &= -\alpha \Sigma\,\partial_t \mathsf{S}. \label{eq:S:sl}
\end{align}
These are coupled with the evolution of the Jacobian $J$ and the DRVs $\rZ, \rW, \rS$ and given by
\begin{align}
\partial_s J &= \tfrac{1+\alpha}{2}\,J\,\rZ + \tfrac{1-\alpha}{2}\,J\,\rW + \tfrac{\alpha}{2\gamma}\,J\,\Sigma\,\rS, \label{eq:J:sl}\\
\partial_s (J\rZ) &= -\tfrac{\alpha}{4\gamma}\,\Sigma\,\rS\,J(\rW+\rZ) + J\,\mathcal{R}, \label{eq:JrZ:sl}\\
J\,\partial_s \rS &= -\alpha \Sigma\,\partial_t \rS - \tfrac{1}{2}J \rS \rW - \tfrac{1}{2}J \rS \rZ, \label{eq:rS:sl}\\
J\,\partial_s \rW &= -2\alpha \Sigma\,\partial_t \rW - \tfrac{1-\alpha}{2}J \rW \rZ - \tfrac{1+\alpha}{2}J \rZ^2 + J\,\mathcal{R}. \label{eq:rW:sl}
\end{align}
Where $\mathcal{R}$ denotes the radial source terms:
\begin{equation}
\mathcal{R} = (d-1)\tfrac{\gamma\alpha (\rW W - \rZ Z) + (W+Z)\Sigma\,\rS}{2\gamma\,\psi} - 2(d-1)\alpha\tfrac{W^2-Z^2}{\psi^2}. \label{eq:R:sl}
\end{equation}
\end{subequations}

\paragraph{Boundary Data in Characteristic Coordinates.}
We translate the Goursat data into boundary conditions on $\mathcal{T}^{\psi}$ (inflow edge) and $\mathcal{T}^{\l}$ (shock edge).

\textit{Inflow edge $\mathcal{T}^{\psi}$ ($t=\tsh$):}
\begin{subequations}\label{bc:lag:T}
\begin{align}
W(\tsh,s) &= W_{\!\circ}(s),\qquad \mathsf{S}(\tsh,s)=S_{\!\circ}(s), \label{bc:W:T}\\
\rW(\tsh,s) &= \rW_{\!\circ}(s),\qquad \rS(\tsh,s)=\rS_{\!\circ}(s). \label{bc:rW:T}
\end{align}

\textit{Shock edge $\mathcal{T}^{\l}$ ($s=t$):}
The Jacobian $J(t,t)$ is evaluated by differentiating the identity $\psi(t,t)=\l(t)$ with respect to $t$ which gives
\[
\dot\l(t) = \tfrac{d}{dt}\psi(t,t) = \partial_t \psi(t,s)|_{s=t} + \partial_s \psi(t,s)|_{s=t} = J(t,t) + \lambda_1(\l(t), t).
\]
Thus, $J(t,t)$ is precisely the prescribed Lax gap (shock strength):
\begin{align}
J(t,t) &= \dot{\l}(t) - \lambda_1\big|_{\l}^{+}(t). \label{bc:J:T}
\end{align}
The value of $Z(t,t)$ is determined by the definition $\lambda_1 = \frac{1+\alpha}{2}z + \frac{1-\alpha}{2}w$, so that
\begin{align}
Z(t,t) &= \tfrac{2}{1+\alpha}\lambda_1\big|_{\l}^{+}(t) - \tfrac{1-\alpha}{1+\alpha}W(t,t). \label{bc:Z:T}
\end{align}
The boundary condition for the dominant derivative $J\rZ(t,t)$ is derived from the compatibility conditions along the shock (relating the spatial derivative to the time evolution along the boundary), and we have that
\begin{align}
J\rZ(t,t) &= \tfrac{2}{1+\alpha}\,\partial_t\lambda_1\big|_{\l}^{+}(t)
+ \tfrac{2(d-1)\alpha}{1+\alpha}\,\tfrac{W(t,t)^2-Z(t,t)^2}{4\,\l(t)} \notag\\[-0.2em]
&\qquad\quad + \tfrac{\alpha}{\gamma(1+\alpha)}\!\left(J\Sigma\,\rS + \Sigma^2 \rS\right)\!(t,t)
- \tfrac{1-\alpha}{1+\alpha}\,\big(J+2\alpha\Sigma\big)(t,t)\,\rW(t,t). \label{bc:rZ:T}
\end{align}
\end{subequations}

\paragraph{Quantitative Bounds on Incoming Data.}
We summarize the essential bounds on the boundary data, derived from the properties of the Guderley solution (Section 3) and the admissible pair (Definition~\ref{def:admissible:pair}). These bounds are crucial for the bootstrap argument in \S\,4.2.
{\allowdisplaybreaks
\begin{subequations}\label{data:bounds:T}
\begin{align}
\|W_{\!\circ}\|_{\infty}+\|S_{\!\circ}\|_{\infty}&\le \m, \qquad \tfrac{1}{\m}\le \Sigma_{\!\circ}(s)\le \m, \label{data:bounds:W:T}\\
\|\rW_{\!\circ}\|_{\infty}+\|\rS_{\!\circ}\|_{\infty}&\le \m, \qquad \|\partial_s\rW_{\!\circ}\|_{\infty}+\|\partial_s\rS_{\!\circ}\|_{\infty}\le \m, \label{data:bounds:rW:T}\\
\tfrac{\dsh^{\varepsilon-\frac{1}{2}}}{\m}\,(t-\ts)^{\frac{1}{2}} \le J(t,t) &\le \m\,\dsh^{\varepsilon}, \label{lax:T}\\
\bigl|\partial_t\lambda_1|_{\l}^{+}(t)\bigr|&\le \m\,\dsh^{\,\varepsilon-\frac{1}{2}}\,(t-\ts)^{-\frac{1}{2}}, \label{data:bounds:l1:T2}\\
\l(t)&>\tfrac{1}{\m}.
\end{align}
\end{subequations}}
We will solve the system \eqref{euler:bp:lag} using these boundary conditions via a Picard iteration scheme, detailed in the next subsection.

\subsection{A Priori Estimates via Bootstrap}

We establish the \emph{a priori} bounds for the system \eqref{euler:bp:lag} on the characteristic triangle $\mathcal{T}$. The argument relies on a bootstrap procedure closed via a transport maximum principle.

We first state a general lemma for obtaining $L^\infty$ bounds for the transport equations appearing in \eqref{euler:bp:lag}.

\begin{lemma}[Transport maximum principle]\label{lemma:max:pr:T}
Let $J(t,s)>0$ and $\Sigma(t,s)>0$ satisfy the bounds \eqref{a:priori} on $\mathcal{T}$. Assume $\dsh$ is sufficiently small such that $J < \beta\Sigma$ for a given $\beta>0$. Let $Y(t,s)$ solve the linear transport equation on $\mathcal{T}$ given by
\begin{equation}\label{transport:T}
J\,\partial_s Y + \beta\,\Sigma\,\partial_t Y = G(t,s)+F(t,s)\,Y,
\end{equation}
with data prescribed on the inflow boundary $t=\tsh$: $Y(\tsh,s)=Y_{\!\circ}(s)$.
Then $Y$ satisfies the bound
\[
\|Y\|_{\mathcal{T}}
\le e^{\|\frac{F}{J}\|_{\mathcal{T}}\,\dsh}
\Bigl(\|Y_{\!\circ}\|_{\infty}+\big\|\tfrac{G}{J}\big\|_{\mathcal{T}}\,\dsh\Bigr).
\]\end{lemma}

\begin{proof}[Proof of Lemma~\ref{lemma:max:pr:T}]
We analyze the characteristics associated with the transport operator 
$$
L_\beta = J\partial_s + \beta\Sigma\partial_t \,. 
$$
In the $(t,s)$ plane, the domain is the triangle $\mathcal{T}=\{(t,s):\ \ts\le s\le t\le \tsh\}$. The characteristic vector field is $V = (\beta\Sigma, J)$.

Since $J, \Sigma, \beta > 0$, the characteristics flow towards increasing $t$ and $s$. To ensure the characteristics properly foliate the domain and exit through the boundary $t=\tsh$ (where the data is prescribed), we must verify the transversality condition at the shock boundary $s=t$.

The slope of the characteristics is $ds/dt = J/(\beta\Sigma)$. We require this slope to be strictly less than the slope of the boundary $s=t$ (which is 1). Thus, we need the condition $J < \beta\Sigma$.

This condition is satisfied provided the time interval $\dsh$ is chosen sufficiently small. The Jacobian $J$ is related to the shock strength (e.g., $J(t,t)$ is the Lax gap, see \eqref{bc:J:T}), which vanishes as $\dsh \to 0$ (see \eqref{lax:T}). In contrast, the sound speed $\Sigma$ remains bounded away from zero (by \eqref{a:priori}). Therefore, for sufficiently small $\dsh$, we guarantee $J < \beta\Sigma$ throughout $\mathcal{T}$.

Since $ds/dt < 1$, the characteristics are well-defined within $\mathcal{T}$ and must intersect the outflow boundary $t=\tsh$. We define the backward-in-time characteristic flow $\Phi(t,s; s')$ starting at $(t,s)$ at time $s'=s$. We integrate the equation along these characteristics. Let $y(s') = Y(\Phi(t,s; s'), s')$. Then
\[
\partial_{s'} y(s') = \tfrac{G+FY}{J} \circ \Phi(t,s; s').
\]
Integrating backward-in-time from the boundary $t'=\tsh$ to $s'=s$ and applying Grönwall's inequality yields the stated bound.
\end{proof}

We now state the main bootstrap lemma, assuming $\dsh$ is sufficiently small depending on the fixed constants $\m$.

\begin{lemma}\label{lemma:D}
Let the solution to \eqref{euler:bp:lag} on $\mathcal{T}$ satisfy the following bootstrap bounds:
{\allowdisplaybreaks
\begin{subequations}\label{bt:T}
\begin{align}
\|W\|_{\infty}+\|\mathsf{S}\|_{\infty}&\le 4\m, \quad \|Z\|_{\infty}\le 4\m, \\
\|\rW\|_{\infty}+\|\rS\|_{\infty}&\le 4\m, \\
\|\partial_s\rW\|_{\infty}+\|\partial_s\rS\|_{\infty}+\|\partial_s\Sigma\|_{\infty}&\le 4\m, \\
\tfrac{1}{4\m}\le \Sigma&\le 4\m, \label{bt:Sigma:T}\\
J\rZ(t,s)&\le 32\,\m^3, \\
|J\rZ(t,s)|&\le 4\m\,\dsh^{\varepsilon-\frac{1}{2}}(t-\ts)^{-\frac{1}{2}}, \label{bt:JrZ:T_singular} \\
\tfrac{1}{4}\,J(t,t) < J(t,s) &< 8\m\,\dsh^{\varepsilon},  \label{bt:J:T}\\
\psi(t,s)&>\tfrac{1}{2\m}.
\end{align}
\end{subequations}}
Then, if $\dsh$ is sufficiently small, the solution satisfies the improved bounds
\begin{subequations}\label{2bt:T}
\begin{align}
\|W\|_{\infty}+\|\mathsf{S}\|_{\infty}&\le 2\m, \quad \|Z\|_{\infty}\le 2\m, \label{2bt:Z:T}\\
\|\rW\|_{\infty}+\|\rS\|_{\infty}&\le 2\m, \\
\|\partial_s\rW\|_{\infty}+\|\partial_s\rS\|_{\infty}+\|\partial_s\Sigma\|_{\infty}&\le 2\m, \\
\tfrac{1}{2\m}\le \Sigma &\le 2\m, \label{2bt:Sigma:T}\\
J\rZ(t,s)&\le 16\,\m^3, \\
|J\rZ(t,s)|&\le 2\m\,\dsh^{\varepsilon-\frac{1}{2}}(t-\ts)^{-\frac{1}{2}}, \label{2bt:JZ2:T}\\
\tfrac{1}{2}\,J(t,t) < J(t,s) &< 4\m\,\dsh^{\varepsilon}, \label{2bt:J:T}\\
\psi(t,s)&>\tfrac{1}{\m}. \label{2bt:psi:T}
\end{align}
\end{subequations}
\end{lemma}

\begin{proof}[Proof of Lemma~\ref{lemma:D}]
We fix $\dsh>0$ sufficiently small (depending only on $\m$) and proceed in five steps.

\emph{Step 1: Boundary control for $Z$ and $J\rZ$ on $\mathcal{T}^{\l}$.}
We first verify the bounds on the shock edge $s=t$. Using the boundary conditions \eqref{bc:Z:T}–\eqref{bc:rZ:T}, the data bounds \eqref{data:bounds:T} (specifically \eqref{data:bounds:l1:T2}), and the bootstrap assumptions \eqref{bt:T}, we obtain the improved bounds on the boundary
\begin{equation}\label{Z:bc:bound}
|Z(t,t)|\le \tfrac{3}{2}\,\m,\qquad
|J\rZ(t,t)|\le \tfrac{3}{2}\,\m\,\dsh^{\varepsilon-\frac{1}{2}}(t-\ts)^{-\frac{1}{2}},\quad
J\rZ(t,t)\le 15\,\m^3.
\end{equation}

\emph{Step 2: Bound for $Z$.}
We integrate the ODE for $Z$,  \cref{eq:Z:sl}, along the $s$-direction. The RHS source terms are bounded by $O(\m^3)$ using \eqref{bt:T}. Integrating backward-in-time from $s=t$, we find that
\[
|Z(t,s)| \le |Z(t,t)| + \int_s^t |\text{RHS}| ds' \le \tfrac{3}{2}\m + O(\m^3)\dsh.
\]
Choosing $\dsh$ small enough yields the improved bound \eqref{2bt:Z:T}.

\emph{Step 3: Bounds for $J$ and $J\rZ$.}
We analyze the evolution of $J\rZ$ using \eqref{eq:JrZ:sl}. The RHS depends on $J\rZ, J\rW$ and radial terms. Using the bootstrap assumptions \eqref{bt:T}, specifically the singular bound \eqref{bt:JrZ:T_singular}, we estimate the derivative:
\[
|\partial_s(J\rZ)| \le C\m^3\left(\dsh^{\varepsilon} + \dsh^{\varepsilon-\frac{1}{2}}(t-\ts)^{-\frac{1}{2}}\right).
\]
Integrating from $s$ to $t$ and using the improved boundary bound \eqref{Z:bc:bound}, we obtain
\begin{align*}
|J\rZ(t,s)| &\le |J\rZ(t,t)| + \int_s^t |\partial_{s'} (J\rZ)| ds' \\
&\le \tfrac{3}{2}\m\,\dsh^{\varepsilon-\frac{1}{2}}(t-\ts)^{-\frac{1}{2}} + (t-s) \cdot C\m^3 \dsh^{\varepsilon-\frac{1}{2}}(t-\ts)^{-\frac{1}{2}}.
\end{align*}
Since $t-s \le t-\ts$, the second term is bounded by $C\m^3 \dsh^{\varepsilon-\tfrac{1}{2}}(t-\ts)^{\frac{1}{2}}$. If $\dsh$ is sufficiently small, this perturbation is small enough to yield the improved bound \eqref{2bt:JZ2:T}. The positive upper bound is improved similarly.

The bounds on $J$ (\cref{2bt:J:T}) follow by integrating \eqref{eq:J:sl}, $\partial_s \log J = \rZ + \dots$. This requires controlling the integral of 
$\rZ=\tfrac{J\rZ}{J}$. Using the improved bounds on $J\rZ$ and the bootstrap lower bound on $J$, the integral $\int_s^t \rZ ds'$ is bounded. Assuming this control holds, the bounds on $J$ are improved, and subsequently, the geometric bound \eqref{2bt:psi:T} follows.

\emph{Step 4: Bound for $\Sigma$.}
The equation for $\Sigma$ (derived from \eqref{eq:W:sl} and \eqref{eq:Z:sl}) fits Lemma~\ref{lemma:max:pr:T} (with $\beta=\alpha$). The linear coefficients $F$ and source terms $G$ are bounded using \eqref{bt:T}. The integrated coefficients $\nu$ and $K$ are small if $\dsh$ is small. Applying Lemma~\ref{lemma:max:pr:T} with $\|\Sigma_\circ\|_\infty \le \m$ yields the improved bound \eqref{2bt:Sigma:T}.

\emph{Step 5: Bounds for $W$, $\mathsf{S}$, $\rW$, $\rS$.}
These variables satisfy transport equations suitable for Lemma~\ref{lemma:max:pr:T}. The key is verifying the integrability conditions on the linear coefficients $F$, which involve terms like $J\rZ$. We use the improved singular bound \eqref{2bt:JZ2:T} (which depends on $t$) to get
\[
\int_{\ts}^{\tsh} \|J\rZ(t,\cdot)\|_{\infty} dt \le \int_{\ts}^{\tsh} 2\m\,\dsh^{\varepsilon-\frac{1}{2}}(t-\ts)^{-\frac{1}{2}} dt = 4\m\dsh^{\varepsilon}.
\]
This integral is small if $\dsh$ is small, ensuring the exponent $\nu$ in Lemma~\ref{lemma:max:pr:T} is small. The integrated source terms $K$ are also small. Applying the lemma yields the remaining improved bounds.
\end{proof}

\subsubsection{Higher–order derivative bounds}
The uniform bounds established in Lemma~\ref{lemma:D} guarantee the smoothness of the solution away from the preshock time $\ts$.

\begin{lemma}\label{lemma:higher:T}
Fix any $\ts<T<\tsh$. For every $k\in\mathbb{N}$, there exists a constant $\mathsf{C}(k,T)$ such that, for all $(t,s)\in\mathcal{T}$ with $t>T$, the solution satisfies:
\begin{subequations}
\label{higher:T}
\begin{align}
|\partial_s^k(J\rZ)(t,s)|+|\partial_s^k\rW(t,s)|+|\partial_s^k\rS(t,s)|+|\partial_s^k J(t,s)|&\le \mathsf{C}(k,T),\\
|\partial_t^k(J\rZ)(t,s)|+|\partial_t^k\rW(t,s)|+|\partial_t^k\rS(t,s)|+|\partial_t^k J(t,s)|&\le \mathsf{C}(k,T).
\end{align}
\end{subequations}
\end{lemma}

\begin{proof}[Proof of \cref{lemma:higher:T}]
We differentiate the system \eqref{euler:bp:lag} with respect to $s$ and $t$. The resulting equations for the derivatives are again transport equations of the form \eqref{transport:T}. Since the boundary data is $C^\infty$ for $t>\ts$, a standard induction argument in $k$, using Lemma~\ref{lemma:D} and Lemma~\ref{lemma:max:pr:T}, establishes the bounds \eqref{higher:T} on any region bounded away from $t=\ts$.
\end{proof}

\subsection{Asymptotics in the Exterior Region $\mathcal{D}$}

We next record the precise Taylor expansions of the solution constructed in the exterior region $\mathcal{D}$ near the preshock corner $(\ts,\ts)$. These expansions capture the asymptotic behavior of the solution as it approaches the preshock.

We analyze the behavior in the characteristic coordinates $(t,s)$ of the triangle $\mathcal{T}$. We focus on the expansions along the shock edge ($s=t$, denoted by superscript ``$+$'') and the corner edge ($s=\ts$, denoted without superscript).

\begin{lemma}\label{lemma:taylor:T}
There exist coefficients $\mathsf{w}_j^{\pm},\ \mathsf{s}_j^{\pm},\ \mathsf{z}_j^{\pm},\ \mathsf{j}_j^{\pm}$, for $j=1,2$, such that the solution in $\mathcal{T}$ admits the following expansions as $t\downarrow \ts$.

Along the shock edge ($s=t$), we have that
\begin{align*}
W(t,t)-W(\ts,\ts) &= \mathsf{w}_{1}^{+}(t-\ts)+\mathsf{w}_{2}^{+}(t-\ts)^{\frac{3}{2}} + E_w^{+}(t-\ts),\\
\mathsf{S}(t,t)-\mathsf{S}(\ts,\ts) &= \mathsf{s}_{1}^{+}(t-\ts)+\mathsf{s}_{2}^{+}(t-\ts)^{\frac{3}{2}} + E_s^{+}(t-\ts),\\
Z(t,t)-Z(\ts,\ts) &= \mathsf{z}_{1}^{+}(t-\ts)^{\frac{1}{2}}+\mathsf{z}_{2}^{+}(t-\ts) + E_z^{+}(t-\ts),\\
J(t,t)-J(\ts,\ts) &= \mathsf{j}_{1}^{+}(t-\ts)^{\frac{1}{2}}+\mathsf{j}_{2}^{+}(t-\ts) + E_j^{+}(t-\ts).
\end{align*}
Along the corner edge ($s=\ts$), we have that
\begin{align*}
\rW(t,\ts)- \rW(\ts,\ts) &= \mathsf{w}_{1}^-(t-\ts)^{\frac{1}{2}} + E^-_w(t-\ts),\\
\rS(t,\ts)- \rS(\ts,\ts) &= \mathsf{s}_{1}^-(t-\ts)^{\frac{1}{2}} + E^-_s(t-\ts),\\
Z(t,\ts)- Z(\ts,\ts) &= \mathsf{z}_{1}^-(t-\ts)+\mathsf{z}_{2}^-(t-\ts)^{2} + E^-_z(t-\ts),\\
J(t,\ts)-J(\ts,\ts) &= \mathsf{j}_{1}^-(t-\ts)^{\frac{1}{2}}+\mathsf{j}_{2}^-(t-\ts) + E^-_j(t-\ts).
\end{align*}
The error terms $E_{\bullet}^{\pm}$ satisfy the second–derivative bounds:
\begin{gather*}
|\partial_{tt}E_w^{+}|+|\partial_{tt}E_s^{+}|\le \mathsf{C},\\
|\partial_{tt} E_w^-| + | \partial_{tt} E_s^-|+\partial_{tt}E_z^{\pm}|+|\partial_{tt}E_j^{\pm}|\le \mathsf{C}\,(t-\ts)^{-\frac{1}{2}}.
\end{gather*}
\end{lemma}

\begin{proof}[Proof of \cref{lemma:taylor:T}]
These expansions are derived by analyzing the structure of the system \eqref{euler:bp:lag} near $t=\ts$. We utilize the prescribed behavior of the admissible pair (Definition~\ref{def:admissible:pair}), specifically the Taylor expansions \eqref{Taylor-lambda1} and the speed gap \eqref{speedgap}, which dictate the leading order behavior of $Z(t,t)$ and $J(t,t)$. The expansions for $W$, $\mathsf{S}$, $\rW$ and $\rS$ follow by integrating their transport equations \eqref{eq:W:sl}, \eqref{eq:S:sl},  \eqref{eq:rW:sl} and \eqref{eq:rS:sl}. The bounds on the error terms are established using the uniform estimates from Lemma~\ref{lemma:D}.
\end{proof}

We translate these expansions in characteristic coordinates back to the physical domain along the shock curve $r=\l(t)$. It is convenient to use the parameterization $\tau=(t-\ts)^{1/2}$.

\begin{corollary}\label{cor:taylor:shock:plus}
The exterior traces along the shock $\l(t)$ admit the following expansions as $t\searrow\ts$ (where $\tau=(t-\ts)^{1/2}$):
\begin{subequations}
\label{taylor:shock:plus}
\begin{align}
w|_{\l}^{+}(\ts+\tau^2) &= \mathsf{w}_0 + \mathsf{w}_{1}^{+}\tau^2 + \mathsf{w}_{2}^{+}\tau^3 + O(\tau^4),\\
S\big|_{\l}^{+}(\ts+\tau^2) &= \mathsf{s}_0 + \mathsf{s}_{1}^{+}\tau^2 + \mathsf{s}_{2}^{+}\tau^3 + O(\tau^4),\\
z|_{\l}^{+}(\ts+\tau^2) &= \mathsf{z}_0 + \mathsf{z}_{1}^{+}\tau + \mathsf{z}_{2}^{+}\tau^2 + O(\tau^3).
\end{align}
\end{subequations}
Furthermore, since $\l\in C^2$ (by Definition~\ref{def:admissible:pair}), we have that
\[
\l(\ts+\tau^2) = \l(\ts)+\dot{\l}(\ts)\tau^2 + \tfrac{\ddot{\l}(\ts)}{2}\tau^4 + O(\tau^5),
\]
with the preshock condition $\dot{\l}(\ts)= \tfrac{1+\alpha}{2}\mathsf{z}_0 + \tfrac{1-\alpha}{2}\mathsf{w}_0$.
\end{corollary}

\subsection{The Interior Region $\mathcal{L}^E$: Setup and Boundary Data}

We now analyze the interior region. To fully determine the preshock profile at $t=\ts$ and prepare for the forward evolution in Section 5, we must consider the domain of dependence defined by the characteristics reaching $t=\ts$. This requires analyzing an extended time interval $[\ts, \tfl]$.

\paragraph{The Extended Domain $\mathcal{L}^E$.}
We define the time $\tfl>\tsh$ such that the interior slow characteristic starting at $(\s(\tsh), \tsh)$ reaches $t=\ts$ at the same spatial location as the interior fast characteristic starting at $(\s(\tfl), \tfl)$. That is, $\eta(\tsh,\ts)=\psi(\tfl,\ts)$. We define the time difference $\dfl:=|\tfl-\ts|$. The relationship between the shock strengths (characterized by $\varepsilon$) implies that $\dfl\sim \dsh^{1/(1+\varepsilon)}$, so $\dfl \to 0$ as $\dsh\to 0$.

We define the extended interior region $\mathcal{L}^{E}$ bounded by the interior characteristics and the composite shock curve $\mathfrak{s}(t)$ (where $\mathfrak{s}(t)=\l(t)$ for $t\le\tsh$ and $\mathfrak{s}(t)=\s(t)$ for $t\ge\tsh$).

We aim to solve the Euler system in $\mathcal{L}^E$:
\begin{subequations}
\label{eul:rv:L}
\begin{align}
(\partial_t+\lambda_3\,\partial_r)w&=\tfrac{\alpha}{2\gamma}\,\sigma^2 \,\partial_r S - \tfrac{\alpha(d-1)}{4r}\,(w^2-z^2),\\
(\partial_t+\lambda_1\,\partial_r)z&=\tfrac{\alpha}{2\gamma}\,\sigma^2 \,\partial_r S + \tfrac{\alpha(d-1)}{4r}\,(w^2-z^2),\\
(\partial_t+\lambda_2\,\partial_r)S&=0.
\end{align}
\end{subequations}
The boundary data along the shock $r=\mathfrak{s}(t)$ is determined by the exterior solution (constructed in $\mathcal{D}$ for $t\le\tsh$ and given by the Guderley solution for $t\ge\tsh$) via the Rankine-Hugoniot conditions.

\paragraph{Transformation to Characteristic Coordinates.}
As in the exterior region, we transform to coordinates aligned with the fast characteristic flow $\psi(t,s)$, now defined in the interior region $\mathcal{L}^E$. We reuse the capitalization notation \eqref{ring:def}, where now $\psi(t,t)=\mathfrak{s}(t)$. The system transforms to \eqref{euler:bp:lag} on the extended characteristic triangle
\begin{equation*}
\mathcal{T}^{E}=\{(t,s):\ \ts\le s\le t\le \tfl\}.
\end{equation*}

The Jacobian $J(t,t) = \dot{\mathfrak{s}}(t) - \lambda_1\big|_{\mathfrak{s}}^{-}(t)$. Note that by the Lax conditions, $J(t,t)<0$ in the interior region.

\subsection{Estimates from the Rankine--Hugoniot Conditions}
We first establish bounds on the exterior traces along the composite shock curve, derived from the admissible pair properties (for $t\le\tsh$) and Proposition~\ref{prop:choice:g} (for $t\ge\tsh$).

For $t\in[\ts,\tsh]$ (using the properties of the admissible pair $(\l(t), \lambda_1|_{\l}^+)$), we have that
\begin{subequations}\label{bounds:sm}
\begin{align}
\dot{\l}(t)-\lambda_1|_{\l}^{+}(t) &> \tfrac{\dsh^{\varepsilon-\frac{1}{2}}(t-\ts)^{\frac{1}{2}}}{\m},\\
|\partial_t (z|_{\l}^{+})(t)|&< \m\,\dsh^{\varepsilon-\frac{1}{2}}(t- \ts)^{-\frac{1}{2}},\\
|\ddot{\l}(t)| &\le \m\,\dsh^{\varepsilon-1}.
\end{align}
\end{subequations}
Using the RH inversion (Lemma~\ref{lemma:rh:inversion}) and the Taylor expansions \eqref{RH:taylor:pt}, we obtain bounds on the interior traces. For $t\in[\ts,\tsh]$, we find that
\begin{align*}
\lambda_1|_{\l}^{-}(t)-\dot{\l}(t) &> \tfrac{\dsh^{\varepsilon-\frac{1}{2}}(t-\ts)^{\frac{1}{2}}}{2\m},\\
|\partial_t (w|_{\l}^{-})(t)| \le 2\m\,\dsh^{\varepsilon-1},\qquad &|\partial_t (z|_{\l}^{-})(t)|< 2\m\,\dsh^{\varepsilon-\frac{1}{2}}(t-\ts)^{-\frac{1}{2}}.
\end{align*}

For the interval $t\in[\tsh,\tfl]$ (using the properties of $\s(t)$ from Proposition~\ref{prop:choice:g}), we have that
\begin{align*}
|\partial_t \sm{z}(t)| \le \m\,(t-\ts)^{\varepsilon-1},\qquad
\tfrac{1}{\m}\,(t-\ts)^{\varepsilon} \le \dot{\s}(t)-\sp{\lambda_1}(t) \le \m\,(t-\ts)^{\varepsilon}.
\end{align*}

We summarize the crucial bounds on the boundary data for the interior bootstrap on $\mathcal{T}^E$.

For $t\in[\ts,\tsh]$, we obtain that
\begin{subequations}
\label{bounds:data:L}
\begin{align}
|W(t,t)|+|Z(t,t)|+|\mathsf{S}(t,t)| &\le \m,\qquad \tfrac{1}{\m}\le \Sigma(t,t)\le \m,\\
|\rW(t,t)|+|\rS(t,t)| &\le \m\,\dsh^{\varepsilon-1},\\
J(t,t) &< -\,\tfrac{\dsh^{\varepsilon-\frac{1}{2}}(t-\ts)^{\frac{1}{2}}}{2\m}<0,  \label{bt:e:L} \\
|J\rZ(t,t)|&\le \m\,\dsh^{\varepsilon-\frac{1}{2}}(t-\ts)^{-\frac{1}{2}} .
\end{align}
\end{subequations}
For $t\in[\tsh,\tfl]$, we obtain that
\begin{subequations}
\label{bounds:data:tsh}
\begin{align}
|\rW(t,t)|+|\rS(t,t)| &\le \m\,(t-\ts)^{\varepsilon-1},\\
-\,\m\,(t-\ts)^{\varepsilon}<J(t,t)&< -\,\tfrac{1}{2\m}\,(t-\ts)^{\varepsilon}<0,  \label{bt:e:L:tsh}\\
|\rZ(t,t)|&\le \m\,(t-\ts)^{\varepsilon-1}.
\end{align}
\end{subequations}

\subsection{Bootstrap Analysis in $\mathcal{L}^{E}$}

We now establish \emph{a priori} bounds for the interior system on $\mathcal{T}^E$. Since $J<0$, the standard transport maximum principle (Lemma~\ref{lemma:max:pr:T}) does not apply directly. We require a modified version adapted to characteristics running backward-in-time in the $t$-label.

\begin{lemma}[Transport maximum principle for $J<0$]\label{lemma:max:prin:L}
Assume the bounds \eqref{bounds:data:L}–\eqref{bounds:data:tsh} hold. Let $Y(t,s)$ solve the transport equation on $\mathcal{T}^{E}$, 
\begin{equation}\label{transport:L}
J\,\partial_s Y + \beta\,\Sigma\,\partial_t Y = G(t,s) + F(t,s)Y, \qquad Y(t,t)=Y_0(t) \,,
\end{equation}
for $\beta>0$. If the integrated normalized coefficients $\nu$ and $K$ (defined analogously to Lemma~\ref{lemma:max:pr:T}, but integrating along the $s$-direction and normalizing by $|J|$) are sufficiently small, then $Y$ is stable and bounded by the boundary data and the sources.
\end{lemma}

\begin{proof}[Proof of \cref{lemma:max:prin:L}]
We define the auxiliary characteristic flow $\Phi(s',t)$ in the $t$-direction, parameterized by $s'$, starting at $t$ when $s'=t$.
\[
\partial_{s'}\Phi(s',t) = \tfrac{\beta \Sigma(\Phi(s',t), s')}{J(\Phi(s',t), s')},\qquad \Phi(t,t)=t.
\]
Since $J<0$ and $\Sigma, \beta>0$, we have $\partial_{s'}\Phi < 0$. The flow is well-defined and covers $\mathcal{T}^{E}$.

Let $\hat{Y}(s') = Y(\Phi(s',t), s')$. By the chain rule, we obtain
\begin{align*}
\tfrac{d\hat{Y}}{ds'} &= \partial_s Y + \partial_t Y \cdot \partial_{s'}\Phi = \tfrac{1}{J} ( J\partial_s Y + \beta\Sigma \partial_t Y ) = \tfrac{G+FY}{J}.
\end{align*}
This is a linear ODE. We solve it using an integrating factor, similar to Lemma~\ref{lemma:max:pr:T}, integrating from $s$ to $t$. The stability relies on controlling the integral of $F/J$ and $G/J$. Using the bounds on $J$ from \eqref{bounds:data:L} and \eqref{bounds:data:tsh}, these integrals are controlled if $\dfl$ (and thus $\dsh$) is small, yielding the stability bound.
\end{proof}

\begin{lemma}\label{lemma:bootstrap:L}
If $\dsh$ (and thus $\dfl$) is sufficiently small, there exists a unique solution of \eqref{eul:rv:L} in $\mathcal{L}^{E}$ satisfying the following improved bounds:
\begin{subequations}
\label{bootstrap:tp:L}
\begin{align}
\|W\|_{\infty}+\|Z\|_{\infty}+\|\mathsf{S}\|_{\infty}&\le 2\m, \quad \tfrac{1}{2\m}\le \|\Sigma\|_{\infty}\le 2\m, \\
|\rW(t,s)|+|\rS(t,s)|&\le 2\m\,\min\{\dsh^{\varepsilon-1},(s-\ts)^{\varepsilon-1}\},\\
-4\m\,\dfl^{\varepsilon}\le J(t,s)&\le \tfrac{1}{2}\,J(t,t),\\
|J\rZ(t,s)|&\le 2\m\,|J\rZ(t,t)|+2\m.
\end{align}
\end{subequations}
\end{lemma}

\begin{proof}[Proof of \cref{lemma:bootstrap:L}]
We employ a bootstrap argument, assuming the bounds in \eqref{bootstrap:tp:L} hold with a factor of $4$ on the right-hand side. The argument mirrors Lemma~\ref{lemma:D}, utilizing Lemma~\ref{lemma:max:prin:L} (the transport principle for $J<0$) instead of Lemma~\ref{lemma:max:pr:T}. The estimates for $J$ and $J\rZ$ follow by direct integration of their evolution equations \eqref{eq:J:sl}–\eqref{eq:JrZ:sl}, utilizing the boundary data bounds \eqref{bounds:data:L}–\eqref{bounds:data:tsh}. The smallness of $\dfl^\varepsilon$ ensures the bootstrap closes.
\end{proof}

\subsubsection{Higher–order derivatives in $\mathcal{L}^{E}$}
Having established the uniform bounds for the solution and its first derivatives (the DRVs) in the interior region (Lemma~\ref{lemma:bootstrap:L}), we now confirm that the solution possesses higher regularity away from the preshock time $t=\ts$. This follows because, for any $T>\ts$, the evolution on the interval $[T, \tfl]$ is strictly hyperbolic and the boundary data is smooth.

\begin{lemma}\label{lemma:higher:L}
Fix $T$ such that $\ts<T<\tfl$. For every $k\in\mathbb{N}$, there exists a constant $\mathsf{C}(k,T)$ such that the solution in $\mathcal{L}^E$ is $C^k$ smooth on the time interval $[T, \tfl]$, satisfying bounds analogous to those established for the exterior region (Lemma \ref{lemma:higher:T}).
\end{lemma}

\begin{proof}[Sketch of Proof]
The argument proceeds by induction on the derivative order $k$. The base case (control of the first derivatives) is established by Lemma~\ref{lemma:bootstrap:L}.

We must verify the prerequisites for standard hyperbolic regularity theory on the interval $[T, \tfl]$.
\begin{enumerate}
    \item \textsl{Strict Hyperbolicity:} For $t \in [T, \tfl]$, the shock strength is strictly positive. This guarantees that the Jacobian $J(t,s)$ (defined in the fast acoustic frame for the interior region) is bounded away from zero (see the bounds in \eqref{bt:e:L}, \eqref{bounds:data:L}, and \eqref{bounds:data:tsh}).
    \item \textsl{Smooth Boundary Data:} The shock path ($\l(t)$ or $\s(t)$) is $C^\infty$ for $t>\ts$ (by Definition~\ref{def:admissible:pair} and Proposition~\ref{prop:choice:g}). The exterior traces are also smooth for $t>\ts$. Consequently, the interior traces, derived via the smooth Rankine-Hugoniot inversion (Lemma~\ref{lemma:rh:inversion}), are $C^\infty$ on $[T, \tfl]$.
\end{enumerate}
Assuming the solution is $C^{k-1}$ on $[T, \tfl]$, the coefficients of the governing system for the interior region, \eqref{euler:bp:lag}, are also $C^{k-1}$. Differentiating the system yields a system of linear transport equations for the $k$-th order derivatives. We can then apply the transport maximum principle (Lemma~\ref{lemma:max:pr:T}) to control these higher derivatives, closing the induction.

The constant $\mathsf{C}(k,T)$ necessarily degenerates (blows up) as $T\to\ts$, reflecting the formation of the singularity at the preshock time.
\end{proof}

\subsection{Taylor Expansions and the Preshock Profile}

We now synthesize the results from the exterior region $\mathcal{D}$ and the interior region $\mathcal{L}^E$ to determine the precise spatial regularity of the solution at the preshock time $t=\ts$. Let $r_{*}=\l(\ts)$ be the preshock location.

We utilize the temporal expansions of the traces along the shock curve $\l(t)$ near $t=\ts$. We use the convenient time variable $\tau=(t-\ts)^{\frac{1}{2}}$. The exterior traces were summarized in Corollary~\ref{cor:taylor:shock:plus}.

\paragraph{Interior Traces.}
The expansions of the interior traces are determined by the exterior ones via the Rankine-Hugoniot conditions.

\begin{lemma}\label{lemma:taylor:s}
Assuming the exterior expansions \eqref{taylor:shock:plus}, the interior traces admit the representations:
\begin{subequations}
\label{taylor:shock:minus}
\begin{align}
\wsm(\ts+\tau^2) &= \mathsf{w}_0 + \mathsf{w}_{1}^{-}\tau^2 + \mathsf{w}_{2}^{-}\tau^3 + O(\tau^4),\\
S\big|_{\l}^{-}(\ts+\tau^2) &= \mathsf{s}_0 + \mathsf{s}_{1}^{-}\tau^2 + \mathsf{s}_{2}^{-}\tau^3 + O(\tau^4),\\
\zsm(\ts+\tau^2) &= \mathsf{z}_0 + \mathsf{z}_{1}^{-}\tau + \mathsf{z}_{2}^{-}\tau^2 + O(\tau^3).
\end{align}
\end{subequations}
The Rankine-Hugoniot relations \eqref{RH:taylor} imply $\mathsf{w}_{1}^{-}=\mathsf{w}_{1}^{+}$ and $\mathsf{s}_{1}^{-}=\mathsf{s}_{1}^{+}$. However, $\mathsf{z}_{1}^{-} \neq \mathsf{z}_{1}^{+}$. The higher-order jump coefficients $\jump{\mathsf{z}_2} = \mathsf{z}_{2}^{+} - \mathsf{z}_{2}^{-}$, $\jump{\mathsf{w}_2}$, $\jump{\mathsf{s}_2}$ depend explicitly on the lower-order exterior coefficients and the shock acceleration $\ddot{\l}(\ts)$.
\end{lemma}

\begin{proof}[Proof of \cref{lemma:taylor:s}]
These identities follow directly by substituting the expansions \eqref{taylor:shock:plus} into the Rankine–Hugoniot–Taylor formulas \eqref{RH:taylor} and matching powers of $\tau$. Let $\chi(t) = \dot{\l}(t) - \lambda_1^+(t)$. Since $\chi(t) = O(\tau)$, the RH conditions imply $\jump{z}(t) = O(\tau)$, while $\jump{w}(t)=O(\tau^3)$ and $\jump{S}(t)=O(\tau^3)$. This confirms the relations between the coefficients stated above.
\end{proof}

\paragraph{Spatial Preshock Profile.}
We now determine the spatial regularity at the preshock point $r_*=\l(\ts)$ by analyzing the characteristic geometry.

\begin{lemma}\label{lemma:taylor:pre}
Assuming the trace expansions \eqref{taylor:shock:plus}–\eqref{taylor:shock:minus}, the preshock profile of the dominant variable $z$ at $r_{*}=\l(\ts)$ is given by
\begin{align*}
z(r,\ts)=\mathsf{z}_0 + \mathsf{a}\,(r-r_{*})^{\frac{1}{3}}
+\begin{cases}
\mathsf{b}_{1}\,(r-r_{*})^{\frac{2}{3}}+O(|r-r_{*}|), & r>r_{*},\\
\mathsf{b}_{2}\,(r-r_{*})^{\frac{2}{3}}+O(|r-r_{*}|), & r<r_{*}.
\end{cases}
\end{align*}
The leading coefficient $\mathsf{a}$ is determined by the characteristic geometry and the leading temporal coefficients $\mathsf{z}_1^\pm$. The sub-leading coefficients $\mathsf{b}_1$ and $\mathsf{b}_2$ generally differ, with the asymmetry $\Delta \mathsf{b} := \mathsf{b}_{1} - \mathsf{b}_{2}$ proportional to the jump coefficient $\jump{\mathsf{z}_2}$.
\end{lemma}

\begin{proof}[Proof of \cref{lemma:taylor:pre}]
The profile is obtained by analyzing the spatial separation of characteristics near the preshock point. Due to the cancellation of the leading order speeds, the spatial separation $|r-r_*|$ scales like $(t-\ts)^{\frac{3}{2}} = \tau^3$. The value $|z(r,\ts)-z_*|$ scales like $(t-\ts)^{\frac{1}{2}} = \tau$. Inverting the spatial relationship yields $\tau \sim |r-r_*|^{\frac{1}{3}}$, leading to the $C^{\frac{1}{3}}$ regularity. The difference in the $\mathsf{b}_i$ coefficients arises from the difference in the expansions of $\zsp$ and $\zsm$ at order $\tau^2$ (Lemma~\ref{lemma:taylor:s}).
\end{proof}

\begin{remark}
Generically, $\jump{\mathsf{z}_2} \neq 0$, implying an asymmetry in the $C^{{\frac{2}{3}} }$ term. While this $C^{{\frac{1}{3}} }$ profile is sufficient for the subsequent analysis in Section~\ref{sec:pre2smooth}, we can fine-tune the construction to achieve symmetry at this order.
\end{remark}

\begin{lemma}[Modulation for Symmetry]\label{lemma:modulate}
There exists a choice of the admissible pair $(\lambda_1\big|_{\l}^{+},\l)$ (still satisfying Lemma~\ref{lemma:l1:l}) such that the resulting asymmetry coefficient $\Delta \mathsf{b}=0$ (equivalently, $\jump{\mathsf{z}_2}=0$).
\end{lemma}

\begin{proof}[Proof of \cref{lemma:modulate}]
We analyze the dependence of the jump coefficient $\jump{\mathsf{z}_2}$ on the shock acceleration $\ddot{\l}(\ts)$.

Let $\chi(t) = \dot{\l}(t) - \lambda_1^+(\l(t), t)$. From the definition of the admissible pair (Definition~\ref{def:admissible:pair}), we expand $\chi(t)$ near $t=\ts$ using $\tau=(t-\ts)^{1/2}$. We use the expansions (noting $\dot{\l}(t)$ is smooth in $t$):
$\lambda_1^+(t) = \lambda_* + \mathsf{l}_1 \tau + \mathsf{l}_2 \tau^2 + O(\tau^3)$,
$\dot{\l}(t) = \lambda_* + \ddot{\l}(\ts) \tau^2 + O(\tau^4)$.
Thus, 
\[
\chi(t) = -\mathsf{l}_1 \tau + (\ddot{\l}(\ts) - \mathsf{l}_2) \tau^2 + O(\tau^3).
\]
(We require $\mathsf{l}_1 < 0$ for $\chi(t)>0$).

The Rankine-Hugoniot expansion (Lemma~\ref{lemma:rh:inversion}, Eq. \eqref{RH:taylor}) gives
\[
\jump{z}(t) = C_A \chi(t) + C_B \chi(t)^2 + O(\chi(t)^3),
\]
where $C_A = -4/(1+\alpha)$ and $C_B$ depends on the state at the preshock. Substituting the expansion for $\chi(t)$:
\begin{align*}
\jump{z}(t) &= C_A \left(-\mathsf{l}_1 \tau + (\ddot{\l}(\ts) - \mathsf{l}_2) \tau^2\right) + C_B (-\mathsf{l}_1 \tau)^2 + O(\tau^3) \\
&= (-C_A \mathsf{l}_1) \tau + \left(C_A(\ddot{\l}(\ts) - \mathsf{l}_2) + C_B \mathsf{l}_1^2\right) \tau^2 + O(\tau^3).
\end{align*}
The jump coefficient $\jump{\mathsf{z}_2}$ is precisely the coefficient of the $\tau^2$ term
\[
\jump{\mathsf{z}_2} = C_A(\ddot{\l}(\ts) - \mathsf{l}_2) + C_B \mathsf{l}_1^2.
\]
Since $C_A \neq 0$, $\jump{\mathsf{z}_2}$ depends linearly and non-trivially on $\ddot{\l}(\ts)$. The acceleration $\ddot{\l}(\ts)$ is a parameter allowed to vary within the bounds specified in Definition~\ref{def:admissible:pair} (Item 4). By the Intermediate Value Theorem, there exists a choice of $\ddot{\l}(\ts)$ (and corresponding admissible pair) for which $\jump{\mathsf{z}_2}=0$.
\end{proof}

By selecting the modulated admissible pair from Lemma~\ref{lemma:modulate}, we arrive at the final description of the preshock profile.

\begin{lemma}[Preshock Regularity]\label{lemma:expansion:0}
By choosing the admissible pair such that $\jump{\mathsf{z}_2}=0$, the preshock data $(w,z,b)$ at $t=\ts$ satisfy the following expansions near $r_{*}=\l(\ts)$:
\begin{subequations}
\label{taylor:rs}
\begin{align}
z(r,\ts) &= z(r_{*},\ts) + \mathsf{a}\,(r-r_{*})^{\frac{1}{3}} + \mathsf{b}\,(r-r_{*})^{\frac{2}{3}} + O(|r-r_{*}|).
\end{align}
The subdominant variables exhibit $C^{1, {\frac{1}{3}} }$ regularity, with potential asymmetry only at the $O(|r-r_*|^{{\frac{4}{3}} })$ level:
\begin{align}
w(r,\ts) &= w(r_{*},\ts) + \mathsf{c}_{w}\,(r-r_{*}) +
\begin{cases}
\mathsf{d}_{w}^+\,(r-r_{*})^{\frac{4}{3}}+O(|r-r_{*}|^{\frac{3}{2}}), & r>r_{*},\\
\mathsf{d}_{w}^-\,(r-r_{*})^{\frac{4}{3}}+O(|r-r_{*}|^{\frac{3}{2}}), & r<r_{*},
\end{cases}\\
b(r,\ts) &= b(r_{*},\ts) + \mathsf{c}_{b}\,(r-r_{*}) +
\begin{cases}
\mathsf{d}_{b}^+\,(r-r_{*})^{\frac{4}{3}}+O(|r-r_{*}|^{\frac{3}{2}}), & r>r_{*},\\
\mathsf{d}_{b}^-\,(r-r_{*})^{\frac{4}{3}}+O(|r-r_{*}|^{\frac{3}{2}}), & r<r_{*}.
\end{cases}
\end{align}
\end{subequations}
\end{lemma}

\begin{proof}[Proof of \cref{lemma:expansion:0}]
The expansion for $z$ follows from Lemma~\ref{lemma:taylor:pre} with $\jump{\mathsf{z}_2}=0$ (so $\mathsf{b}_1=\mathsf{b}_2=\mathsf{b}$). The expansions for $w$ and $b$ are derived similarly by analyzing their evolution along the slow and intermediate characteristics, utilizing the expansions established in Lemmas~\ref{lemma:taylor:T} and \ref{lemma:taylor:s}. The asymmetry in the $O(|r-r_*|^{{\frac{4}{3}} })$ terms arises from the jumps in $w$ and $b$ across the shock (the coefficients $\jump{\mathsf{w}_2}, \jump{\mathsf{s}_2}$ in Lemma~\ref{lemma:taylor:s}), which are generally non-zero even after modulating for $\jump{\mathsf{z}_2}=0$.
\end{proof}

\subsection{Proof of Proposition~\ref{prop:weak2pre}}
\begin{proof}
We synthesize the constructions to define the global patching solution on $[\ts, \tsh]$ and verify its properties.

\textit{Construction of the Global Solution.}
We fix the admissible pair $(\l, \lambda_1\big|_{\l}^{+})$ provided by Lemma~\ref{lemma:l1:l}. Let $(u,\rho,E)^{\mathcal{D}}$ denote the solution constructed in the exterior region $\mathcal{D}$ (Lemmas~\ref{lemma:D}–\ref{lemma:higher:T}), and let $(u,\rho,E)^{\mathcal{L}}$ denote the solution constructed in the interior region $\mathcal{L}$ (part of the extended region $\mathcal{L}^E$, Lemmas~\ref{lemma:bootstrap:L}–\ref{lemma:higher:L}). Let $(u,\rho,E)^{L}$ be the solution from Proposition~\ref{prop:large2weak}.

We define the global solution $(u,\rho,E)$ on $\R_{+}\times[\ts,\tsh]$ by patching these solutions:
\[
(u,\rho,E)(r,t)=
\begin{cases}
(u,\rho,E)^{\mathcal{D}}(r,t), & (r,t)\in\mathcal{D},\\
(u,\rho,E)^{\mathcal{L}}(r,t), & (r,t)\in\mathcal{L},\\
(u,\rho,E)^{L}(r,t), & \text{otherwise}.
\end{cases}
\]

\textit{Verification of Properties.}
Item~\ref{prop:weak2pre:item2} (Characteristic construction) is satisfied by the definition of the Goursat problem in $\mathcal{D}$ and the subsequent solution in $\mathcal{L}$. Item~\ref{prop:weak2pre:item4} (Controlled approach) follows from the properties of the admissible pair (Lemma~\ref{lemma:l1:l}) and the Rankine-Hugoniot analysis (Lemma~\ref{lemma:taylor:s}). Item~\ref{prop:weak2pre:item5} (Preshock profile) is established by the analysis in Lemma~\ref{lemma:expansion:0}.

It remains to verify the global consistency and regularity claimed in Items~\ref{prop:weak2pre:item1} (Seamless matching) and \ref{prop:weak2pre:item3} (Regular shock evolution). The solution is $C^\infty$ within each region for $t>\ts$. The Rankine-Hugoniot conditions and Lax conditions are satisfied across $r=\l(t)$ by construction. The quiescent core $\mathcal{H}$ is preserved.

The crucial remaining step is to verify the seamless ($C^\infty$) matching across the characteristic boundaries separating the regions, specifically across the exterior fast characteristic $\Gamma_\psi = \{(\psi(\tsh, t), t) : t\in[\ts, \tsh]\}$. We must show that $(u,\rho,E)^{\mathcal{D}}$ matches $(u,\rho,E)^{L}$ smoothly across $\Gamma_\psi$. We argue using the Riemann variables $(w, z, S)$.

\textit{$C^\infty$ Matching Across $\Gamma_\psi$.}

\textit{Step 1: Matching 0-th order variables.}
By construction (the Goursat data), we have that $w^{\mathcal{D}}=w^L$ and $S^{\mathcal{D}}=S^L$ along $\Gamma_\psi$.
The curve $\Gamma_\psi$ is a $\lambda_1$-characteristic. The variable $z$ evolves along this characteristic according to Eq. \eqref{eulerz}. Since $w$ and $S$ match, the coefficients and source terms of this transport equation are identical for both solutions along $\Gamma_\psi$. Since the initial data at $t=\tsh$ matches (by Item~\ref{prop:weak2pre:item1}), uniqueness of solutions to transport equations implies $z^{\mathcal{D}}=z^L$ along $\Gamma_\psi$.

\textit{Step 2: Matching first derivatives.}
We now show the spatial derivatives $\p_r w, \p_r S, \p_r z$ match across $\Gamma_\psi$. Let $D_t$ denote the tangential derivative operator along $\Gamma_\psi$. Since $\Gamma_\psi$ is a $\lambda_1$-characteristic, $D_t = \p_t + \lambda_1 \p_r = L_1$.

For $S$, we use the equation $L_2(S)=0$ (\cref{eulerS}). We can write
\[
D_t S = L_1(S) = L_2(S) + (\lambda_1-\lambda_2)\p_r S = \alpha\sigma \p_r S.
\]
We solve for the spatial derivative to obtain
\begin{equation} \label{eq:rS:psi}
\p_r S = \tfrac{1}{\alpha \sigma} D_t S.
\end{equation}
Since $w, z, S$ match along $\Gamma_\psi$ (Step 1), $\sigma$ matches. The tangential derivatives $D_t S$ also match. Therefore, we have that $\p_r S^{\mathcal{D}} = \p_r S^L$ along $\Gamma_\psi$.

For $w$, we use $L_3(w)=\text{Sources}_w$ (\cref{eulerw}). We can write
\[
D_t w = L_1(w) = L_3(w) + (\lambda_1-\lambda_3)\p_r w = \text{Sources}_w + 2\alpha\sigma \p_r w.
\]
We solve for the spatial derivative to obtain
\begin{equation} \label{eq:rw:psi}
\p_r w = \tfrac{1}{2\alpha\sigma} (D_t w - \text{Sources}_w).
\end{equation}
The sources depend on $w, z, S, \p_r S, \tfrac{1}{r} $. Since all these quantities match along $\Gamma_\psi$, and $D_t w$ matches, we conclude that $\p_r w^{\mathcal{D}} = \p_r w^L$ along $\Gamma_\psi$.

For $z$, we consider the evolution of $\p_r z$. The DRV equation for $\mathring{z}$ (\cref{DRV-z}) is a transport equation along $\lambda_1$-characteristics. Thus, $\p_r z$ satisfies an ODE along $\Gamma_\psi$. The coefficients of this ODE depend on the variables and derivatives already matched ($w, z, S, \p_r w, \p_r S$). Since the initial data for $\p_r z$ at $t=\tsh$ matches (due to the $C^\infty$ matching of the shock curves $\s$ and $\l$, Item~\ref{prop:weak2pre:item1}), uniqueness of ODE solutions implies $\p_r z^{\mathcal{D}} = \p_r z^L$ along $\Gamma_\psi$.

\textit{Step 3: Induction.}
A standard induction argument shows that higher derivatives $\p_r^k$ also satisfy transport equations along $\Gamma_\psi$ with coefficients depending on lower-order derivatives. Thus, all derivatives match across $\Gamma_\psi$, establishing $C^\infty$ regularity across this boundary. A similar argument applies to the interior characteristic boundary $\eta(\tsh, \cdot)$. This completes the proof of Proposition~\ref{prop:weak2pre}.
\end{proof}

\section{Backward Regularization: From Preshock to $C^{1,{\frac{1}{3}}}$ Initial Data}
\label{sec:pre2smooth}

\subsection{Overview and Main Results}

The constructions in the preceding sections established a solution that, at the preshock time $t=\ts$, exhibits a specific singularity structure at the location $r_*$. As established in Proposition~\ref{prop:weak2pre} (specifically Lemma~\ref{lemma:expansion:0}), the state variables at $t=\ts$ possess the following regularity near $r_*$:
\[
z(\cdot, \ts) \in C^{\frac{1}{3}}, \qquad (w, b)(\cdot , \ts)\in C^{1,{\frac{1}{3}}}.
\]
The dominant variable $z$ possesses a symmetric $C^{\frac{1}{3}}$ cusp at $r_*$.

In this final section, we demonstrate the backward-in-time regularizing effect of the Euler equations for this specific data profile. While the forward evolution involved compression leading to shock formation (a loss of regularity), the backward evolution corresponds to an expansion, which smooths the flow. We prove that the singular state at $t=\ts$ can be evolved backward to an earlier time $t=\ti < \ts$, resulting in a state that is globally $C^{1, {\frac{1}{3}}}$ (i.e., $C^1$ with Hölder-$ \tfrac{1}{3}$  continuous derivatives). This provides the desired regular initial data for the entire shock development process constructed in this paper.

The main result of this section is the following:

\begin{prop}\label{prop:pre2smooth}
Let $(w_L,z_L,b_L)$ be the solution from Proposition~\ref{prop:weak2pre}. There exists a time $\ti<\ts$ and a unique solution $(w,z,b)$ of the Euler equations \eqref{euler:rv} on $[\ti,\ts)$ with the following properties:
\begin{enumerate}
\item \textup{Regularity Gain:} For every $t\in [\ti,\ts)$, the spatial profile $(w,z,b)(\cdot,t)$ belongs to $C^{1,{\frac{1}{3}}}(\R_+)$. Furthermore, the quiescent core is maintained: $(u,\rho,E)=(0,1,0)$ in the region
\[
\mathcal{H}_i=\{(r,t):\,0<r<(-\tf)^{\frac{1}{\lambda}},\ t\in[\ti,\ts]\}.
\]

\item \textup{Structure of the Singularity Propagation:} The solution is $C^2$ (or smoother) everywhere except along the two characteristic curves emanating backward-in-time from the preshock point $(r_*,\ts)$. Let $\phi(r_*,\cdot)$ and $\eta(r_*,\cdot)$ denote the entropy ($\lambda_2$) and slow ($\lambda_3$) characteristics starting at $r_*$.
\begin{itemize}
    \item Along the slow characteristic $\eta(r_*, \cdot)$, the variable $w$ exhibits a $C^{1,{\frac{1}{3}}}$ cusp on both sides.
    \item Along the entropy characteristic $\phi(r_*, \cdot)$, all three variables $(w, z, b)$ exhibit $C^{1,{\frac{1}{3}}}$ cusps on both sides.
\end{itemize}

\item \textup{Continuity up to $\ts$:} The solution $(w,z,b)$ is continuous in time up to $t=\ts$, and matches the terminal data:
\[
\lim_{t\nearrow\ts}(w,z,b)(r,t)=(w_L,z_L,b_L)(r,\ts)\quad\text{for every }r>0.
\]
\end{enumerate}
\end{prop}

\paragraph{Localization Strategy.}
The proof relies on analyzing the backward-in-time Cauchy problem starting from the singular data at $t=\ts$. Due to the finite speed of propagation, it suffices to analyze the solution locally near the singularity $(r_*,\ts)$.

We fix a small spatial scale $\theta>0$ and define a localized spacetime neighborhood $\mathcal{B}$. The domain is chosen to widen as time goes backward (since $t-\ts < 0$), ensuring it encompasses the relevant domain of dependence and facilitates the analysis of characteristic interactions. We define the domain as
\begin{equation*}
\mathcal{B}
=\Bigl\{(r,t):\ \ti<t<\ts,\ r_*-\theta<r<r_*+\theta-4\m^2\,\tfrac{(t-\ts)}{\theta^2}\Bigr\}.
\end{equation*}

The proof of Proposition~\ref{prop:pre2smooth} follows from the following localized result, which provides the technical core of the regularization analysis.

\begin{prop}\label{prop:main:pre2smooth}
Consider the backwards Cauchy problem for \eqref{euler:rv} in $\mathcal{B}$ with terminal data $(w_0,z_0,b_0)$ prescribed at $t=\ts$. Assume the data satisfies the Taylor expansions \eqref{taylor:rs} at $r_*$ and is $C^2$ away from $r_*$. Then there exists a time $\ti<\ts$ and a unique solution $(w,z,b)$ on $[\ti,\ts)$ in $\mathcal{B}$ such that:
\begin{enumerate}
\item \textup{Local Regularity and Structure:} The solution $(w,z,b)(\cdot,t)$ is in $C^{1,{\frac{1}{3}}}$ for every $t<\ts$. The solution is $C^2$ away from the characteristics $ \eta(r_*,\cdot)$ and $\phi(r_*, \cdot)$. The specific structure of the $C^{1, {\frac{1}{3}}}$ cusps along these characteristics is as described in Proposition~\ref{prop:pre2smooth}, Item 2.

\item \textup{Local Continuity:} The solution $(w,z,b)$ is continuous in time and matches the terminal data $(w_0, z_0, b_0)$ continuously as $t\nearrow\ts$.
\end{enumerate}
\end{prop}

The remainder of this section is devoted to the proof of Proposition~\ref{prop:main:pre2smooth}.

\subsection{Fast acoustic characteristic coordinates}
To analyze the backward evolution from the singular state at $t=\ts$ and demonstrate the regularization effect, we introduce a coordinate transformation specifically designed to resolve the $C^{{\frac{1}{3}}}$ cusp.
As in Section~\ref{sec:weak2pre}, we work along the $\lambda_1$–flow.  Following our convention (introduced in  \eqref{bw:flows:bt}), we continue to use $s$ to denote
evolutionary time; note, however, that in this application,  $s \in [\ti,\ts]$ exactly matches coordinate time $t$. 

We define the backward-in-time map
\begin{equation}\label{choice:psi:smooth}
\p_s\psi(x,s)=(\lambda_1\circ\psi ) (x,s),\qquad
\psi(x,\ts)=\tfrac{x^3}{3}+r_*,\qquad x\in(-\theta^{\frac{1}{3}},\theta^{\frac{1}{3}}).
\end{equation}

This specific cubic initialization is chosen to exactly counteract the $C^{{\frac{1}{3}}}$ singularity of $z(r,\ts)$ at $r_*$. 
Recall from \eqref{taylor:rs} that the preshock profile implies $\partial_r z \sim (r-r_*)^{-{\frac{2}{3}}}$. The choice of $\psi(x,\ts)$ means $r-r_* \sim x^3$, and thus the Jacobian is $J(x,\ts) = \partial_x \psi(x,\ts) = x^2$. With $Z=z\circ\psi$, by the chain rule, the transformed derivative is $\partial_x Z = J \cdot (\partial_r z \circ \psi)$. This yields the cancellation: $\partial_x Z \sim x^2 \cdot (x^3)^{-{\frac{2}{3}}} = O(1)$. Consequently, the transformed data $Z_0(x):=Z(x,\ts)$ is $C^2$, effectively unfolding the cusp.

We pull back the Riemann  variables and the DRVs by the fast acoustic characteristics $\psi$,  and set
\begin{equation*}
\begin{gathered}
Z=z\circ\psi,\quad W=w\circ\psi,\quad \mathsf{S}=S\circ\psi,\quad \Sigma=\sigma\circ\psi, \\
\rZ=\left(\partial_r z +\tfrac{\alpha}{2\gamma}\sigma \p_r S\right)\circ\psi,\quad \rW=\left(\partial_r w -\tfrac{\alpha}{2\gamma}\sigma \p_r S\right)\circ\psi,\quad \rS=\partial_r S\circ\psi.
\end{gathered}
\end{equation*}
and define $J=\p_x\psi$. The terminal conditions at time $s=\ts$ are
\begin{subequations}
\label{in:cond}
\begin{align}
&W(x,\ts)=W_0(x),\quad Z(x,\ts)=Z_0(x),\quad \mathsf{S}(x,\ts)=S_0(x),\\
&\rW(x,\ts)=\rW_0(x),\quad \rZ(x,\ts)=\rZ_0(x),\quad \rK(x,\ts)=\rK_0(x),\\
&\Sigma(x,\ts)=\Sigma_0(x),\qquad J(x,\ts)=J_0(x)=x^2.
\end{align}
\end{subequations}
Using the regularity established by the coordinate transformation and \eqref{taylor:rs}, and shrinking $\theta$ if necessary, we may assume the following bounds on the data:
\begin{align*}
&\|Z_0\|_\infty+\|W_0\|_\infty+\|K_0\|_\infty<\m,\qquad
\|\rW_0\|_\infty+\|\rK_0\|_\infty<\m,\\
&-\m\le J_0\rZ_0\le -\tfrac{1}{2}\m,\qquad
\|\p_x\rW_0\|_\infty+\|\p_x\rK_0\|_\infty+\|\p_x(J_0\rZ_0)\|_\infty\le \m,\\
&\tfrac{1}{\m}\le \Sigma_0\le \m,\qquad
J_0\rZ_0\le -\tfrac{1}{\m}.
\end{align*}

\subsection{Equations along the fast acoustic characteristic}

We now rewrite the Euler system in the new characteristic coordinates $(x,s)$ and identify the challenges posed by the degeneracy of the system. In these coordinates, the Euler equations transform into a coupled set of transport equations (for $W, \mathsf{S},\rK,\rW$) and ODE-like equations along the characteristics (for $Z, J, J\rZ$). The main challenge arises because the transport speeds depend on the Jacobian $J$. Since $J(x,\ts)=x^2$, the system becomes degenerate at the singularity ($x=0$).

By composing with the fast acoustic characteristics $\psi$, we can write the Euler equations as
{\allowdisplaybreaks
\begin{subequations}\label{fast:bw:B}
\begin{align}
\p_sZ&=\tfrac{\alpha}{2\gamma}\,\Sigma^2\rK+(d-1)\alpha\,\tfrac{W^2-Z^2}{2\,\psi}, \label{eq:Z:B}\\[2pt]
J\,\p_sW&=-2\alpha\,\Sigma\,\p_xW+\tfrac{\alpha}{2\gamma}\,\Sigma^2\rK\,J+(d-1)\alpha\,J\,\tfrac{W^2-Z^2}{2\,\psi}, \label{eq:W:B}\\[2pt]
J\,\p_s \mathsf{S}&=-\alpha\,\Sigma\,\p_x \mathsf{S}, \label{eq:K:B}\\[2pt]
\p_s\Sigma&=\alpha\,\Sigma\Bigl((d-1)\tfrac{W+Z}{\psi}\Bigr)+\tfrac{\alpha}{2\gamma}\,\Sigma^2\rK, \label{eq:Sigma:B}\\[2pt]
\mathcal{R}&=(d-1)\tfrac{\gamma\alpha(\rW W-\rZ Z)+(W+Z)\Sigma\rK}{2\gamma\,\psi}-\tfrac{W^2-Z^2}{\psi^2}, \label{eq:R:B}\\[2pt]
\p_sJ&=\tfrac{1+\alpha}{2}\,J\,\rZ+\tfrac{1-\alpha}{2}\,J\,\rW+\tfrac{\alpha}{2\gamma}\,J\,\Sigma\,\rK, \label{eq:J:B}\\[2pt]
\p_s(J\rZ)&=-\tfrac{\alpha}{4\gamma}\,\Sigma\rK\,J(\rW+\rZ)+J\,\mathcal{R}, \label{eq:JrZ:B}\\[2pt]
J\,\p_s\rK&=-\alpha\,\Sigma\,\p_x\rK-\tfrac{1}{2}J\,\rK\,\rW-\tfrac{1}{2}J\,\rK\,\rZ, \label{eq:rK:B}\\[2pt]
J\,\p_s\rW&=-2\alpha\,\Sigma\,\p_x\rW-\tfrac{1-\alpha}{2}J\,\rW\,\rZ-\tfrac{1+\alpha}{2}J\,\rZ^2
-\tfrac{\alpha}{4\gamma}\,\Sigma\rK\,J\,\rZ+\tfrac{\alpha}{4\gamma}J\,\Sigma\rK\,\rZ+J\,\mathcal{R}. \label{eq:rW:B}
\end{align}
\end{subequations}}

At $s=\ts$ we may (after enlarging $\m$ if needed) assume that
\[
\|\p_s\Sigma(\cdot,\ts)\|_\infty+\|\p_x\Sigma(\cdot,\ts)\|_\infty\le \m.
\]
Crucially, from \eqref{eq:rK:B}–\eqref{eq:rW:B}, we observe that the initial time derivatives are singular at $x=0$:
\begin{equation}\label{bound:initial:B}
|\p_s\rK(x,\ts)|\le \tfrac{\m \alpha }{x^2},
\qquad
|\p_s\rW(x,\ts)|\le \tfrac{\m \alpha }{x^2}.
\end{equation}

The singularities in these initial time derivatives propagate backward along the slow acoustic and entropy characteristics, respectively. We introduce the auxiliary trajectories $\Upsilon_3(x)$ (slow acoustic) and $\Upsilon_2(x)$ (entropy) to track the location of these propagating singularities in the transformed coordinates. These trajectories define the loci where the derivatives $\rW$ and $\rK$  lose higher regularity, necessitating the use of weighted estimates, and these
auxiliary functions solve
\begin{subequations}\label{traj:B}
\begin{align}
\p_x\Upsilon_3(x)&=\big(\tfrac{J}{2\alpha\,\Sigma}\big)(x,\Upsilon_3(x)),\qquad \Upsilon_3(0)=\ts, \label{ups:1}\\
\p_x\Upsilon_2(x)&=\big(\tfrac{J}{\alpha\,\Sigma}\big)(x,\Upsilon_2(x)),\qquad\ \ \Upsilon_2(0)=\ts. \label{ups:2}
\end{align}
\end{subequations}
\begin{remark} 
The structure of the  transport equations \eqref{eq:W:B}, \eqref{eq:K:B},  \eqref{eq:rK:B}, and \eqref{eq:rW:B} show that it is possible to convert the transport 
structure to an ODE equation by composition with either a \textit{ temporal} flow  or a \textit{spatial} flow.    Consider, for example, Eq. \eqref{eq:rW:B}.  We can view this in the traditional form
\[
\p_s\rW + \tfrac{2\alpha\,\Sigma}{J}\p_x\rW = \text{RHS} \,.
\]
or we can view the label $x$ as playing the role of the evolutionary variable, and instead write \eqref{eq:rW:B} as
\[
\p_x\rW + \tfrac{J}{2\alpha\,\Sigma}\p_s\rW = \text{RHS} \,,
\]
To convert the transport structure to ODE structure, one can compose the transport operator with the flow of $\tfrac{2\alpha\,\Sigma}{J}$
in the first equation or the flow of $\tfrac{J}{2\alpha\,\Sigma}$ in the second equation.  In the former, the vector field becomes singular,
while in the latter the vector field merely vanishes at $x=0$.  
As such,  we have chosen to define the flows in \eqref{traj:B} by using  $x$ as the evolutionary variable.
This  allows us to deduce that the functions $\Upsilon_2, \Upsilon_3$ are $C^1$ smooth, and by using Lemma \ref{lemma:2weights}, we are able to compare the spatial regularity of $\rW$ and $\rK$ with the regularity of $(s-\Upsilon_3)^{\frac{1}{3}}$ and $(s-\Upsilon_2)^{\frac{1}{3}}$, respectively.
\end{remark} 

\subsection{A Priori Estimates via Bootstrap}
The core of the regularization analysis lies in establishing uniform control over the solution and its derivatives in these characteristic coordinates via a bootstrap argument. We construct solutions by a Picard iteration (convergence is proved in Appendix~\ref{appendix:B}). Here we record the a~priori bounds that deliver the regularization. The bootstrap argument aims to control the degeneracy of the Jacobian $J$ (ensuring $J>0$ for $s<\ts$ via \eqref{bt:e:tp}, which guarantees the invertibility of the transformation) and manage the singular behavior of the derivatives using the weighted norms in \eqref{bt:kt:tp}.

\begin{lemma}\label{lemma:bootstrap:B}
There exists $\ti<\ts$, with $\delta_*:=\ts-\ti\ll1$, such that any solution of \eqref{fast:bw:B} on $[\ti,\ts]$ satisfies
\begin{subequations}\label{bootstrap:tp:B}
\begin{align}
&\|W\|_\infty\le 2\m,\quad \|Z\|_\infty\le 2\m,\quad \|\mathsf{S}\|_\infty\le \m, \label{bt:z:tp}\\
&-2\m\le J\rZ<-\tfrac{2}{\m},\qquad \|\p_x(J\rZ)\|_\infty<2\m, \label{bt:rz:tp}\\
&\|\rW\|_\infty<2\m,\quad \|\rK\|_\infty<2\m, \\
&\tfrac{1}{2\m}\le \Sigma\le 2\m,\qquad |\p_s\Sigma|+|\p_x\Sigma|<2\m, \label{bt:s:tp}\\
&|\p_s\rW|\ <\ \tfrac{2\m}{(s-\Upsilon_3(x))^{\frac{2}{3}}},\qquad 
 |\p_s\rK|\ <\ \tfrac{2\m}{(s-\Upsilon_2(x))^{\frac{2}{3}}}, \label{bt:kt:tp}\\
&x^2-\tfrac{\m}{2}(s-\ts)<J(x,s)<x^2-2\m(s-\ts) ,\qquad \|\p_xJ\|_\infty<2\m. \label{bt:e:tp}
\end{align}
\end{subequations}\end{lemma}

\begin{proof}[Proof of Lemma~\ref{lemma:bootstrap:B}]
We aim to improve the following weaker bootstrap assumptions by choosing $\delta_*$ sufficiently small:
\begin{subequations}\label{bootstrap:B}
\begin{align}
&\|W\|_\infty+\|Z\|_\infty\le 4\m,\qquad \|\mathsf{S}\|_\infty\le 4\m, \label{bt:z}\\
&-4\m\le J\rZ\le -\tfrac{1}{4\m},\qquad \|\p_x(J\rZ)\|_\infty<4\m, \label{bt:rz}\\
&\|\rW\|_\infty\le 2\m,\qquad \|\rK\|_\infty\le 4\m, \label{bt:w}\\
&\tfrac{1}{2\m}\le \Sigma\le 4\m,\qquad |\p_s\Sigma|<2\m,\quad |\p_x\Sigma|<4\m, \label{bt:s}\\
&|\p_s\rW|<\tfrac{4\m}{(s-\Upsilon_3(x))^{\frac{2}{3}}},\quad |\p_s\rK|<\tfrac{4\m}{(s-\Upsilon_2(x))^{\frac{2}{3}}}, \label{bt:kt}\\
&x^2-\tfrac{\m}{4}(s-\ts)<J<x^2-4\m (s-\ts),\qquad \|\p_xJ\|_\infty<4\m. \label{bt:e}
\end{align}
\end{subequations}

\emph{Step 1: Bounds on Undifferentiated Variables ($W, \mathsf{S}, Z, \Sigma$).}
Equation \eqref{eq:K:B} is a transport equation, so $\|K\|_\infty$ is controlled by the terminal trace. For $W$, we use the maximum principle, Lemma~\ref{lemma:max:p:z} (with $\beta=2\alpha$). For $Z$, we integrate \eqref{eq:Z:B}; using the bootstrap assumptions \eqref{bootstrap:B}, we obtain that
\[
|\p_sZ|\le 64\m^3\quad\Longrightarrow\quad |Z(x,s)|\le |Z(x,\ts)|+64\m^3\,\delta_*.
\]
Choosing $\delta_*$ small gives the improved bounds in \eqref{bt:z:tp}. Similarly, from \eqref{eq:Sigma:B},
\[
|\p_s\Sigma|\ \lesssim\ \m^2+\m\quad\Rightarrow\quad \tfrac{1}{2\m}\le \Sigma\le 2\m,
\]
for $\delta_*$ small, which gives the improved bounds on $\Sigma$ in \eqref{bt:s:tp}.

\emph{Step 2: Bounds on the Dominant Derivative $J\rZ$ and the Jacobian $J$.}
From \eqref{eq:JrZ:B}, using the bootstrap assumptions, we have
\[
|\p_s(J\rZ)|\le 256\m^4\quad\Longrightarrow\quad |J\rZ(x,s)-J\rZ(x,\ts)|\le 256\m^4\,\delta_*,
\]
which yields the improved bounds in \eqref{bt:rz:tp}.
Differentiating \eqref{eq:J:B} and applying Grönwall's inequality (utilizing the bounds on $J\rZ$) yields the improved two-sided bound for $J$ and the bound on $\p_x J$ in \eqref{bt:e:tp}.

\emph{Step 3: Bounds on Spatial Derivatives $\p_x(J\rZ)$.}
Differentiating \eqref{eq:JrZ:B} with respect to $x$ gives
\begin{align*}
\p_s\!\bigl(\p_x(J\rZ)\bigr)
&=\tfrac{\alpha}{4\gamma}\,\Sigma\rK\Bigl(\p_x(J\rZ)+\p_xJ\,\rW+J\,\p_x\rW\Bigr)
+\tfrac{\alpha}{4\gamma}\,\p_x\Sigma\,\rK\,(J\rZ+J\rW)\\[-2pt]
&\quad+\tfrac{\alpha}{4\gamma}\,\Sigma\,\p_x\rK\,(J\rZ+J\rW)+\p_xJ\,\mathcal{R}+J\,\p_x\mathcal{R}. \nonumber\end{align*}
Under \eqref{bootstrap:B}, the terms not containing $\p_x\rW$ are integrable on $[\ti,\ts]$ with total size $\le 164\m^4\delta_*$. For the $\p_x\rW$ term, we use \eqref{eq:rW:B} to express $\p_x\rW$ in terms of $\p_s\rW$:
\[
\p_x\rW=\tfrac{J}{2\alpha\,\Sigma}\,\p_s\rW+\tfrac{\mathcal{N}}{\Sigma},
\qquad |\mathcal{N}|\le 256\m^3.
\]
We then utilize the weighted bootstrap assumption \eqref{bt:kt} and the integrability property from Lemma \ref{lemma:2weights} to obtain
\[
\int_s^{\ts}\!\!|J\,\p_x\rW|\,ds'\ \lesssim\ \m^3\,\delta_*^{\frac{1}{3}}+\m^4\,\delta_*.
\]
Combining these bounds and taking $\delta_*$ small gives the improved bound $\|\p_x(J\rZ)\|_\infty<2\m$.

\emph{Step 4: Weighted Bounds for Singular Derivatives ($\p_s\rW, \p_s\rK$).}
Because of the singular initial conditions \eqref{bound:initial:B}, $\p_s\rW(\cdot,\ts)$ and $\p_s\rK(\cdot,\ts)$ are not bounded. We must employ the technical tools developed in Lemmas \ref{lemma:max:p:z} and \ref{lemma:2weights}.  We introduce the weight $\mathfrak{w}$ from Lemma~\ref{lemma:2weights} (with $\beta=2\alpha$) and apply it to \eqref{eq:rW:B}. Using the commutator identity in \eqref{commutator:w}, we obtain a transport equation for the weighted quantity:
\[
J\,\p_s\bigl(\mathfrak{w}\,\p_s\rW\bigr)-2\alpha\,\Sigma\,\p_x\bigl(\mathfrak{w}\,\p_s\rW\bigr)=\mathcal{N},
\]
where $\mathcal{N}$ is bounded. Applying the maximum principle (Lemma~\ref{lemma:max:p:z}) to this equation yields a uniform bound on $\mathfrak{w}\,\p_s\rW$. Using the comparison $\mathfrak{w} \sim (s-\Upsilon_3(x))^{{\frac{2}{3}}}$ (Lemma \ref{lemma:2weights}), we obtain the improved weighted bound
\[
|\p_s\rW|\ \lesssim\ \tfrac{1}{(s-\Upsilon_3(x))^{\frac{2}{3}}},
\]
and similarly for $\p_s\rK$ with $\Upsilon_2$. This establishes \eqref{bt:kt:tp}.
\end{proof}

\subsection{Proof of Proposition~\ref{prop:main:pre2smooth}}
\begin{proof}
We translate the bounds established in the characteristic coordinates back to the physical space. From Lemma~\ref{lemma:bootstrap:B}, for every $s\in[\ti,\ts)$, we have that 
\[
J( \cdot, s),\ \rZ( \cdot, s),\ \Sigma( \cdot, s)\in W^{1,\infty},\qquad
\rW( \cdot, s),\ \rK( \cdot, s) \in C^{\frac{1}{3}}.
\]
The crucial regularization mechanism is encapsulated by the weighted bounds in \eqref{bt:kt:tp}. These imply the pointwise spatial bounds:
\begin{equation*}
|\p_x\rW|\ \lesssim\ \tfrac{1}{(s-\Upsilon_3(x))^{\frac{2}{3}}},
\qquad
|\p_x\rK|\ \lesssim\ \tfrac{1}{(s-\Upsilon_2(x))^{\frac{2}{3}}}.\end{equation*}

While these weights are singular along the trajectories $\Upsilon_3$ and $\Upsilon_2$, they are integrable in the spatial variable $x$. Integrating these bounds with respect to $x$ yields a gain of one-third of a derivative (as $\int (s-\Upsilon(x))^{-{\frac{2}{3}} } dx$ behaves like $(s-\Upsilon(x))^{{\frac{1}{3}} }$). This establishes that $\rW,\rK\in C^{\frac{1}{3}}$ in $x$ (and are $W^{1,\infty}$ away from the curves $s=\Upsilon_3(x)$ and $s=\Upsilon_2(x)$).

Since $J=\p_x\psi$ and we established $J>0$ on $[\ti,\ts)$ (by \eqref{bt:e:tp}), the map $x\mapsto\psi(x,s)$ is a $C^2$ diffeomorphism for each fixed $s<\ts$.

Returning to Eulerian variables $(r,t)$ with $r=\psi(x,s)$, we have
\[
\p_r w=\rW+\tfrac{\alpha}{2\gamma}\,\sigma\,\rK,\quad
\p_r z=\rW-\tfrac{\alpha}{2\gamma}\,\sigma\,\rK,\quad
\p_r S=\rK.
\]
The $C^{{\frac{1}{3}}}$ regularity of the right-hand side in $x$ translates directly to $C^{{\frac{1}{3}}}$ regularity of the derivatives in $r$. Thus, $w(\cdot,t),z(\cdot,t),k(\cdot,t)\in C^{1,{\frac{1}{3}}}$ for each $t<\ts$. The loci where $C^2$ fails are the images of $s=\Upsilon_3(x)$ and $s=\Upsilon_2(x)$, namely the characteristics $\psi(r_*,\cdot)$ (fast) and $\eta(r_*,\cdot)$ (slow). Time continuity follows from \eqref{fast:bw:B} and \eqref{bootstrap:tp:B}, which give $(\p_t+\lambda_1\p_r)w$, $(\p_t+\lambda_1\p_r)z$ and $(\p_t+\lambda_1\p_r)k$ bounded, while the spatial regularity is at least $C^{\frac{1}{3}}$. This proves the claim.
\end{proof}

\subsection{Technical Lemmas: Maximum Principle and Weighted Estimates}

We rely on two key technical tools to close the bootstrap argument: a maximum principle adapted for the backward transport equations and the analysis of specific weights used to control the singular derivatives.

\begin{lemma}\label{lemma:max:p:z}
Let $\beta\ge\alpha:= \tfrac{\gamma-1}{2}$. Consider the transport equation on the domain $\mathcal{D} =
\{ (x,s) \colon  - X_{\max} <x < X_{\max} + \epsilon (s-\ts), s\in [\ti, \ts]\}$ given by
\begin{equation}\label{transport:B}
J\,\p_s Q +\beta\,\Sigma\,\p_x Q=f_1(x,s)\,Q+f_2(x,s),
\end{equation}
with data prescribed at the final time $s=\ts$: $Q(x,\ts)=Q_0(x)\in C^0$. We assume $X_{\max}=\theta^{\frac{1}{3}}$ and that $\epsilon=\epsilon(\delta_*, \m)>0$ is taken sufficiently small.
Assume that  $J,\Sigma$ satisfy the bootstrap assumptions (including \eqref{bt:s:tp} and \eqref{bt:e:tp}, and sufficient spatial regularity), and that $f_1,f_2 \in L^\infty(\mathcal{D})$. Then \eqref{transport:B} has a unique $C^0$ solution and (adjusting constants $\m$ as necessary)
\[
\|Q\|_{L^\infty_{x,s}}
\le \exp\!\Bigl(C(\m,\beta)\,\delta_*^{\frac{1}{3}}\,\|f_1\|_{L^\infty_{x,s}}\Bigr)
\Bigl(\|Q_0\|_{L^\infty_{x}}+C(\m,\beta)\,\|f_2\|_{L^\infty_{x,s}}\,\delta_*^{\frac{1}{3}}\Bigr).
\]\end{lemma}

\begin{proof}[Proof of Lemma~\ref{lemma:max:p:z}]
We aim to establish the maximum principle for the degenerate transport equation \eqref{transport:B} on the domain $\mathcal{D}$. The difficulty lies in the singularity of the temporal vector field $V(x,s) = \tfrac{\beta\Sigma(x,s)}{J(x,s)}$ near $(0,\ts)$.

\textbf{Step 1: The Regularized System.}
For $\updelta>0$, we introduce the regularized Jacobian $J_\updelta(x,s) = J(x,s) + \updelta$. Since $J\ge 0$ (by the bootstrap assumptions for small $\delta_*$), $J_\updelta \ge \updelta$.
We consider the regularized transport equation for $Q_\updelta$:
\begin{equation}\label{transport:B:reg}
J_\updelta\,\p_s Q_\updelta +\beta\,\Sigma\,\p_x Q_\updelta=f_1\,Q_\updelta+f_2, \quad Q_\updelta(x,\ts)=Q_0(x).
\end{equation}
The regularized vector field $V_\updelta(x,s) = \tfrac{\beta\Sigma(x,s)}{J_\updelta(x,s)}$ is bounded and spatially Lipschitz (with constants depending on $\updelta$).

\textbf{Step 2: Regularized Temporal Characteristics and Domain Geometry.}
We define the regularized characteristic flow $\Phi_\updelta(X,s)$ by
\begin{equation}\label{ode:Phi:reg}
\p_s\Phi_\updelta(X,s)=V_\updelta(\Phi_\updelta(X,s),s), \qquad \Phi_\updelta(X,\ts)=X.
\end{equation}

The geometry of the domain $\mathcal{D}$ is crucial for ensuring that the solution depends only on the terminal data $Q_0$. The slanted spatial boundary, defined by the speed $\epsilon$, ensures that $\mathcal{D}$ is the Domain of Dependence for the backward evolution. In the $(s,x)$-plane, the slope of the characteristics is $V_\updelta$, and the slope of the boundary is $\epsilon$. To ensure containment as time evolves backward (decreasing $s$), we require the characteristics to be steeper than the boundary: $V_\updelta > \epsilon$. Near the spatial boundary $X_{\max}$, the Jacobian $J$ is bounded away from zero by the bootstrap assumptions. Consequently, $V_\updelta$ is bounded below by a positive constant $V_{\min}$. The parameter $\epsilon$ is chosen sufficiently small such that $\epsilon < V_{\min}$, guaranteeing the characteristics remain strictly within $\mathcal{D}$ for all $s\in[\ti, \ts]$. Standard theory now guarantees a unique smooth solution $Q_\updelta$ exists for \eqref{transport:B:reg} within $\mathcal{D}$.

Along these contained characteristics, the quantity $\mathcal{Q}_\updelta(s) = Q_\updelta(\Phi_\updelta(X,s),s)$ satisfies the ODE
\begin{equation}\label{ode:Q:reg}
\p_s \mathcal{Q}_\updelta(s) = \left(\tfrac{f_1}{J_\updelta} \mathcal{Q}_\updelta + \tfrac{f_2}{J_\updelta}\right)\circ \Phi_\updelta(X,s).
\end{equation}

\textbf{Step 3: Uniform Estimates (Independent of $\updelta$).}
We establish estimates independent of $\updelta$.

\textit{3a. Uniform Bound on Displacement.} We analyze the flow using the kinematic flattening transformation. Letting $Y_\updelta = \Phi_\updelta^3$, we have that
\begin{align*}
\p_s Y_\updelta &= 3\Phi_\updelta^2 V_\updelta(\Phi_\updelta, s) = 3\beta \tfrac{\Phi_\updelta^2 \Sigma(\Phi_\updelta, s)}{J(\Phi_\updelta, s) + \updelta} =: V_{3,\updelta}(\Phi_\updelta, s).
\end{align*}
Crucially, $V_{3,\updelta}(x,s)$ is uniformly bounded, independent of $\updelta$. This relies on the bootstrap assumption \eqref{bt:e:tp}, which implies $J(x,s) \ge c x^2$ for some $c>0$ (since $s\le \ts$). Thus,
\[
|V_{3,\updelta}(x,s)| \le 3\beta \|\Sigma\|_{L^\infty} \tfrac{x^2}{J(x,s)+\updelta} \le 3\beta \|\Sigma\|_{L^\infty} \tfrac{x^2}{J(x,s)} \le C(\m,\beta).
\]
Integrating this bound backward from $s=\ts$ yields
\[
|\Phi_\updelta(X,s)^3-X^3|\le C(\m,\beta)\,(\ts-s).
\]
This implies the $\updelta$-independent  uniform displacement bound 
\begin{equation}\label{eq:displacement_uniform}
|X - \Phi_\updelta(X,s)| \le C'(\m,\beta) (\ts-s)^{\frac{1}{3}}.
\end{equation}

\textit{3b. Uniform Bound on the Integral Weight.} We analyze the integral $I_\updelta(s) = \int_s^{\ts} \tfrac{1}{J_\updelta}\circ\Phi_\updelta ds'$. Using the definition of the flow and the assumption $\Sigma \ge \tfrac{1}{2\m}$, noting that  $V_\updelta>0$ (so $\Phi_\updelta$ increases with $s$), we obtain that
\begin{align*}
I_\updelta(s) &\le \int_s^{\ts} \tfrac{2\m\Sigma}{J_\updelta}\circ\Phi_\updelta ds' = \tfrac{2\m}{\beta} \int_s^{\ts} V_\updelta(\Phi_\updelta, s') ds'
= \tfrac{2\m}{\beta} \int_s^{\ts} \p_{s'} \Phi_\updelta(X,s') ds' = \tfrac{2\m}{\beta} (X - \Phi_\updelta(X,s)).
\end{align*}
Using the uniform displacement bound \eqref{eq:displacement_uniform} yields the following key uniform integral estimate:
\begin{equation}\label{int:J:B_final_reg}
I_\updelta(s) \le C''(\m,\beta) (\ts-s)^{\frac{1}{3}} \le C''(\m,\beta) \delta_*^{\frac{1}{3}}.
\end{equation}

\textbf{Step 4: The Maximum Principle (Uniform in $\updelta$).}
We apply Grönwall's inequality to the ODE \eqref{ode:Q:reg} backward from $s=\ts$, where $\mathcal{Q}_\updelta(\ts)=Q_0(X)$, and find that
\begin{align*}
|\mathcal{Q}_\updelta(s)| &\le \exp\left(\int_s^{\ts} \left|\tfrac{f_1}{J_\updelta}\circ\Phi_\updelta\right| ds'\right) \left( \|Q_0\|_{L^\infty_x} + \int_s^{\ts} \left|\tfrac{f_2}{J_\updelta}\circ\Phi_\updelta\right| ds' \right) \\
&\le \exp\left(\|f_1\|_\infty I_\updelta(s)\right) \left( \|Q_0\|_{L^\infty_x} + \|f_2\|_\infty I_\updelta(s) \right).
\end{align*}
Substituting the uniform bound for $I_\updelta(s)$ yields the desired maximum principle estimate, which holds uniformly for all $\updelta>0$.

\textbf{Step 5: Convergence and Conclusion.}
We have established that the sequence $\{Q_\updelta\}$ is uniformly bounded in $L^\infty(\mathcal{D})$. We must show convergence to the unique solution $Q$ of the original equation \eqref{transport:B}.

We cannot apply Arzelà-Ascoli directly to $\Phi_\updelta$ since $\p_s \Phi_\updelta = V_\updelta$ is not uniformly bounded. However, we apply it to the transformed sequence $Y_\updelta$. Since $\p_s Y_\updelta = V_{3,\updelta}$, and $V_{3,\updelta}$ *is* uniformly bounded (Step 3a), the sequence $\{Y_\updelta\}$ is equicontinuous. This guarantees uniform convergence (up to a subsequence) $Y_{\updelta_k} \to Y$, and consequently, $\Phi_{\updelta_k} \to \Phi = Y^{1/3}$ uniformly.

We can pass to the limit in the integral formulation of the solution:
\[
\mathcal{Q}_\updelta(s) = Q_0(X) + \int_{\ts}^s \left(\tfrac{f_1}{J_\updelta} \mathcal{Q}_\updelta + \tfrac{f_2}{J_\updelta}\right)\circ \Phi_\updelta(s') ds'.
\]
Due to the uniform boundedness of $Q_\updelta$, the uniform convergence of $\Phi_\updelta$, and the convergence of the weights $\tfrac{1}{J_\updelta}$ in $L^1$ along the characteristics (which follows from the uniform integrability established in Step 3b), we can pass to the limit using the Dominated Convergence Theorem. The limit $Q$ satisfies the integral equation corresponding to the original degenerate transport equation \eqref{transport:B} and inherits the maximum principle estimate. Uniqueness of the $C^0$ solution follows from the linearity of the equation and the maximum principle.
\end{proof}

\begin{remark}[Intuition for the Weight $\mathfrak{w}$]\label{remark:intuition:w}
The analysis of the backward regularization relies on controlling terms involving the inverse Jacobian $1/J$. At the preshock time, $J(x,\ts)=x^2$, which is degenerate at $x=0$. To understand how this degeneracy behaves backward in time, we introduce the weight $\mathfrak{w}$.

The weight $\mathfrak{w}$ is explicitly designed to be the backward-in-time transport of this initial degeneracy along the characteristic flow associated with the operator $L_\beta:=J\,\p_s + \beta\,\Sigma\,\p_x$. It satisfies $L_\beta \mathfrak{w} = 0$ with terminal data $\mathfrak{w}(x,\ts)=x^2$.

The crucial result of the following lemma is the comparison \eqref{bounds:weight:2}: $\mathfrak{w}(x,s) \approx (s-\Upsilon(x))^{\frac{2}{3}}$. This comparison translates the spatial degeneracy ($x^2$) into a temporal behavior relative to the singular trajectory $\Upsilon(x)$. The weight $\mathfrak{w}$ vanishes as the time difference approaches zero with a power of $\frac{2}{3}$. Consequently, the singular coefficients involving $\tfrac{1}{\mathfrak{w}}$ blow up at the rate $(s-\Upsilon(x))^{-\frac{2}{3}}$. Because this rate is integrable in time, the comparison is essential for controlling the transport equations (as seen in the proof of Lemma \ref{lemma:bootstrap:B}).
\end{remark}

\begin{lemma}[\textbf{Analysis of the Transported Weight}]\label{lemma:2weights}
Let $\beta>\alpha$. Define the characteristic transport operator $L_\beta:=J\,\p_s + \beta\,\Sigma\,\p_x$.
Define the singular trajectory $\Upsilon(x)$ by the ODE
\begin{equation*}
\p_x\Upsilon(x)= \big(\tfrac{J}{\beta\,\Sigma}\big) (x,\Upsilon(x)),\qquad \Upsilon(0)=\ts,\end{equation*}
and let the weight $\mathfrak{w}(x,s)$ solve the transport equation
\begin{equation*}
L_\beta\mathfrak{w}=0,\qquad \mathfrak{w}(x,\ts)=x^2.\end{equation*}
If  $J,\Sigma$ satisfy the bootstrap assumptions \eqref{a:priori}, then the following bounds hold (where $\m$ is a sufficiently large constant depending on the bootstrap parameters):
\begin{subequations}\label{bounds:weight}
\begin{align}
&\int_{s}^{\ts}\!\tfrac{1}{(s'-\Upsilon(x))^{\frac{2}{3}}}\,ds'\ <\ \m\,(\ts-s)^{\frac{1}{3}},
\qquad
\int_{s}^{\ts}\!\tfrac{1}{\mathfrak{w}(x,s')}\,ds'\ <\ \m\,(\ts-s)^{\frac{1}{3}}, \label{bounds:weight:1}\\
&\tfrac{1}{\m}\ <\ \tfrac{\mathfrak{w}(x,s)}{(s-\Upsilon(x))^{\frac{2}{3}}}\ <\ \m, \label{bounds:weight:2}\\
&\|\mathfrak{w}^{\frac{1}{2}}\p_x\mathfrak{w}\|_\infty\le \m. \label{bounds:weight:3}
\end{align}
\end{subequations}
Moreover, the following commutator relations hold:
\begin{subequations} 
\label{commutator:w}
\begin{equation}
\bigl[L_\beta,\ \mathfrak{w}\bigr]=0 \,,
\end{equation} 
and for some bounded function $f(x,s)$,
\begin{equation} 
\bigl[L_\beta,\ (s-\Upsilon(x))^{\frac{2}{3}}\bigr]=f(x,s) \,.
\end{equation}
\end{subequations} 
\end{lemma}

\begin{proof}[Proof of Lemma~\ref{lemma:2weights}]
The proof proceeds by analyzing the relationship between the weight $\mathfrak{w}$ and the temporal distance to the singular trajectory $\Upsilon(x)$.

\paragraph{Step 1: Cubic separation for \(\Upsilon(x)\).}
We first establish how the time-to-singularity $\ts-\Upsilon(x)$ relates to the spatial position $x$. This requires showing that $J(x,\Upsilon(x))$ remains comparable to $J(x,\ts)=x^2$.

The bootstrap assumptions \eqref{bootstrap:tp:B} imply that $\Sigma$ is bounded above and below ($\tfrac{1}{2\m} \le \Sigma \le 2\m$). Furthermore, $J$ is Lipschitz continuous in the time variable $s$, with a Lipschitz constant $C(\m)$ (which follows from the bounds on $\p_s J$ derived from \eqref{eq:J:B} and \eqref{fast:bw:B}).

Using the Lipschitz continuity of $J$ in time, we compare $J$ at time $\Upsilon(x)$ with $J$ at time $\ts$ to obtain that
\begin{align*}
|J(x,\Upsilon(x)) - J(x,\ts)| \le C(\m) |\Upsilon(x)-\ts|.
\end{align*}
Substituting this into the ODE for $\Upsilon(x)$, we have that 
\begin{align*}
|\p_x\Upsilon(x)| &= \left| \tfrac{J(x,\Upsilon(x))}{\beta\Sigma(x,\Upsilon(x))} \right| \le \tfrac{2\m}{\beta} |J(x,\Upsilon(x))| 
\\
&
\le \tfrac{2\m}{\beta} \left( |J(x,\ts)| + C(\m)|\Upsilon(x)-\ts| \right) 
= \tfrac{2\m}{\beta} \left( x^2 + C(\m)(\ts-\Upsilon(x)) \right).
\end{align*}
Let $H(x) = \ts-\Upsilon(x)$. Then $H(0)=0$ and $|\p_x H(x)| \le C_1(\m) (x^2 + H(x))$. By Grönwall's inequality, this implies that $H(x) \le C_2(\m)|x|^3$.

Now we use this improved estimate on the time difference to find that
\begin{align*}
J(x,\Upsilon(x)) = J(x,\ts) + O(|\Upsilon(x)-\ts|) = x^2 + O(|x|^3).
\end{align*}
For sufficiently small $x$, the $x^2$ term dominates, confirming the comparability
\[
c_*(\m)\,x^2 \;\le\;J\bigl(x,\Upsilon(x)\bigr)\;\le\;C_*(\m)\,x^2.
\]
Substituting this back into the ODE for $\Upsilon(x)$ and integrating from $0$ to $x$ yields the rigorous \textit{cubic separation}:
\begin{equation}\label{eq:cubic_separation}
c_1(\m)\,|x|^3 \ \le\ \ts-\Upsilon(x)\ \le\ C_1(\m)\,|x|^3.
\end{equation}

\paragraph{Step 2: Analyzing the temporal evolution near the trajectory.}
We analyze how the transport operator $L_\beta$ acts on the time difference $(s-\Upsilon(x))$.  We have that 
\begin{align*}
L_\beta(s-\Upsilon(x)) &= (J\p_s + \beta\Sigma\p_x)(s-\Upsilon(x)) 
= J(x,s) - \beta\Sigma(x,s) \p_x\Upsilon(x).
\end{align*}
Substituting  the definition of $\p_x\Upsilon(x)$, we obtain that
\begin{align*}
L_\beta(s-\Upsilon(x)) &= J(x,s) - \beta\Sigma(x,s) \tfrac{J(x,\Upsilon(x))}{\beta\Sigma(x,\Upsilon(x))} 
= \tfrac{J(x,s)\Sigma(x,\Upsilon(x)) - \Sigma(x,s)J(x,\Upsilon(x))}{\Sigma(x,\Upsilon(x))}.
\end{align*}
The numerator is the difference of the quantity $J\Sigma$ evaluated at time $s$ and time $\Upsilon(x)$. By the Mean Value Theorem (MVT) in the time variable, this difference is proportional to $(s-\Upsilon(x))$, and hence
\[
J(s)\Sigma(\Upsilon)-\Sigma(s)J(\Upsilon) = \p_s(J\Sigma)|_{(x,\tau^*)} (s-\Upsilon(x)),
\]
for some intermediate time $\tau^*$. The bootstrap assumptions ensure that the time derivatives $\p_s J$ and $\p_s \Sigma$ are bounded. Therefore, we obtain the crucial bound
\begin{equation}\label{eq:Lbeta_s_minus_Upsilon}
\left|\tfrac{L_\beta(s-\Upsilon(x))}{s-\Upsilon(x)}\right| \le C(\m).
\end{equation}

\paragraph{Step 3: Normalization and the Maximum Principle.}
We introduce the normalized weight $U(x,s)$, motivated by the cubic separation (Step 1) which suggests $x^2 \sim (\ts-\Upsilon)^{\frac{2}{3}}$, defining
\[
U(x,s):=\tfrac{\mathfrak{w}(x,s)}{(s-\Upsilon(x))^{\frac{2}{3}}}.
\]
We analyze the evolution of $U$. Since $L_\beta\mathfrak{w}=0$, we have that
\begin{align*} 
L_\beta U &= \mathfrak{w} L_\beta((s-\Upsilon)^{-\frac{2}{3}}) = \mathfrak{w} \left(-\tfrac{2}{3} (s-\Upsilon)^{-\frac{5}{3}} L_\beta(s-\Upsilon)\right) 
= -\tfrac{2}{3} U \tfrac{L_\beta(s-\Upsilon(x))}{s-\Upsilon(x)}.
\end{align*}
From \eqref{eq:Lbeta_s_minus_Upsilon} in Step 2, the coefficient of $U$ is bounded. Thus, $|L_\beta U| \le C(\m) U$.

We check the terminal condition at $s=\ts$. Using $\mathfrak{w}(x,\ts)=x^2$ and the cubic separation \eqref{eq:cubic_separation}, it follows that
\[
U(x,\ts)=\tfrac{x^2}{(\ts-\Upsilon(x))^{\frac{2}{3}}} \sim \tfrac{x^2}{(x^3)^{\frac{2}{3}}} = 1.
\]
Thus, $U(x,\ts)$ is bounded above and below by positive constants depending on $\m$. We apply the Maximum Principle (Lemma~\ref{lemma:max:p:z}) to $U$. This confirms that $U(x,s)$ remains bounded above and below throughout the domain, establishing the key comparability result \eqref{bounds:weight:2}.

\paragraph{Step 4: Integral Bounds and Commutators.}
The integral bounds \eqref{bounds:weight:1} follow directly from the comparability \eqref{bounds:weight:2}. We integrate the temporal weight:
\begin{align*}
\int_{s}^{\ts}\!\tfrac{ds'}{(s'-\Upsilon(x))^{\frac{2}{3}}}
&=3\left((\ts-\Upsilon(x))^{\frac{1}{3}}-(s-\Upsilon(x))^{\frac{1}{3}}\right).
\end{align*}
Since $(a^{\frac{1}{3}}-b^{\frac{1}{3}}) \le (a-b)^{\frac{1}{3}}$ for $a>b>0$, this is bounded by $3(\ts-s)^{\frac{1}{3}}$. The bound for $\int \frac{1}{\mathfrak{w}} ds'$ follows immediately.

The commutator relations \eqref{commutator:w} are also immediate.  We have that $[L_\beta, \mathfrak{w}] = 0$, and
\[
[L_\beta,(s-\Upsilon)^{\frac{2}{3}}] = \tfrac{2}{3}\,(s-\Upsilon)^{-\frac{1}{3}}\,L_\beta(s-\Upsilon).
\]
Using \eqref{eq:Lbeta_s_minus_Upsilon}, this is $\sim (s-\Upsilon)^{-\frac{1}{3}} (s-\Upsilon) = (s-\Upsilon)^{\frac{2}{3}}$, which is bounded.

\paragraph{Step 5: Bounding the spatial derivative of $\mathfrak{w}$.}
We establish the bound \eqref{bounds:weight:3} on $\|\mathfrak{w}^{\frac{1}{2}}\p_x\mathfrak{w}\|_\infty$. We analyze the auxiliary function $Y = \mathfrak{w}^{\frac{3}{2}}$.
\begin{equation} \label{transport:frak:w_Y}
L_\beta Y = 0, \qquad Y(x, \ts) = (x^2)^{\frac{3}{2}} = |x|^3.
\end{equation}
We want to bound $\p_x Y = \tfrac{3}{2} \mathfrak{w}^{\frac{1}{2}} \p_x \mathfrak{w}$. Since $L_\beta Y=0$, we have $\p_x Y = -\tfrac{J}{\beta\Sigma} \p_s Y$. It suffices to bound $\p_s Y$.

Let $Q(x,s) = \p_s Y$. We derive the evolution equation for $Q$ by differentiating $L_\beta Y=0$ with respect to $s$ to find that
\begin{align*}
0 &= \p_s (L_\beta Y) = L_\beta (\p_s Y) + [\p_s, L_\beta] Y \\
&= L_\beta Q + (\p_s J) \p_s Y + (\p_s(\beta\Sigma)) \p_x Y.
\end{align*}
We substitute $\p_x Y = -\tfrac{J}{\beta\Sigma} Q$ back into the equation and obtain that
\begin{align*}
L_\beta Q &= -(\p_s J) Q - (\p_s(\beta\Sigma)) \left(-\tfrac{J}{\beta\Sigma} Q\right) \\
&= \left( -(\p_s J) + \tfrac{J}{\beta\Sigma}(\p_s (\beta\Sigma)) \right) Q.
\end{align*}
The coefficient of $Q$ on the RQS is bounded due to the bootstrap assumptions (boundedness of $J, \Sigma, \p_s J, \p_s \Sigma$, and $\Sigma$ bounded away from zero). Thus $L_\beta Q = f_1(x,s) Q$ with $f_1$ bounded.

We check the terminal condition at $s=\ts$. Using $L_\beta Y=0$ at $s=\ts$, we have
\begin{align*}
Q(x,\ts) &= -\tfrac{\beta\Sigma(x,\ts)}{J(x,\ts)} \p_x Y(x,\ts).
\end{align*}
Since $J(x,\ts)=x^2$ and $\p_x (|x|^3) = 3x|x|$, we have
\begin{align*}
Q(x,\ts) &= -\tfrac{\beta\Sigma(x,\ts)}{x^2} (3x|x|) = -3\beta\Sigma(x,\ts) \text{sgn}(x).
\end{align*}
This is bounded. Therefore, the Maximum Principle (Lemma \ref{lemma:max:p:z}) applies to $Q$, implying $Q(x,s)$ is bounded everywhere. Consequently, $\p_x Y$ is also bounded, which completes the proof of \eqref{bounds:weight:3}.
\end{proof}

\appendix

\appendix
\section{Iteration and convergence in~$\mathcal{D}$}\label{appendix:D}

The goal of this appendix is to prove the existence and uniqueness of the solution to the Goursat system \eqref{euler:rv:bp:S} (equivalently \eqref{euler:bp:lag}) in the region $\mathcal{D}$. We employ a Picard iteration scheme and demonstrate its convergence using a weighted norm. We work in the characteristic coordinates $(t,s)$ on the triangular domain $\mathcal{T}=\{(t,s):\, t\in[\ts,\tsh],\ \ts\le s\le t\}$.

\subsection{The Iteration Scheme}
We construct the solution using a fully lagged Picard iteration. The scheme is linearized by evaluating all coefficients (e.g., $J, \Sigma$) and nonlinear source terms using the previous iterate $(n-1)$. This results in a sequence of linear transport equations (solved backward-in-time from $t=\tsh$) and ODEs (integrated forward from $s=t$).

The iteration scheme for $n\ge 1$ is defined as follows. (We omit the explicit definitions of the source terms $\mathcal{R}^{(n-1)}$ for brevity, referring to their definitions in the main text).

{\allowdisplaybreaks[4]
\begin{subequations}\label{iteration:T}
\noindent\textbf{Transport Equations (Solved backward-in-time from $t=\tsh$):}
\begin{align}
 J^{(n-1)}\p_s \Sigma^{(n)} - \alpha\,\Sigma^{(n-1)} \p_t \Sigma^{(n)} &= \tfrac{\alpha}{2}\,\Sigma^{(n-1)}\!\left(J^{(n-1)} (\rW^{(n-1)}+\rZ^{(n-1)}) + (d-1)\,\tfrac{ W^{(n-1)}+ Z^{(n-1)}}{\psi^{(n-1)}} \right), \label{it:Sigma:T}\\[2mm]
 J^{(n-1)} \p_s W^{(n)} - 2\alpha\,\Sigma^{(n-1)} \p_t W^{(n)} &= \tfrac{\alpha}{2\gamma} \bigl(\Sigma^{(n-1)}\bigr)^{2} \rK^{(n-1)} J^{(n-1)}
 +(d-1)\alpha\, J^{(n-1)} \tfrac{(W^{(n-1)})^{2}-(Z^{(n-1)})^{2}}{2 \psi^{(n-1)}}, \label{it:W:T}\\[2mm]
J^{(n-1)} \p_s K^{(n)} - \alpha\,\Sigma^{(n-1)} \p_t K^{(n)} &= 0, \label{it:K:T}\\[2mm]
J^{(n-1)}\p_s \rK^{(n)} + \alpha\,\Sigma^{(n-1)} \p_t \rK^{(n)} &=
 - \tfrac{1}{2}\, J^{(n-1)} \rK^{(n-1)} (\rW^{(n-1)} + \rZ^{(n-1)}), \label{it:rK:T}\\[2mm]
J^{(n-1)} \p_s \rW^{(n)} + 2 \alpha\, \Sigma^{(n-1)} \p_t \rW^{(n)} &=
 - \tfrac{1-\alpha}{2}\, J^{(n-1)} \rW^{(n-1)} \rZ^{(n-1)}
 - \tfrac{1+\alpha}{2}\, J^{(n-1)} \big(\rZ^{(n-1)}\big)^{2} + \mathcal{R}^{(n-1)}J^{(n-1)}. \label{it:rW:T}
\end{align}

\noindent\textbf{ODE Equations (Integrated forward-in-time in $s$ from $s=t$):}
\begin{align}
\p_s Z^{(n)} &= \alpha(d-1)\,\tfrac{ (W^{(n-1)})^{2}- (Z^{(n-1)})^{2}}{\psi^{(n-1)}} + \tfrac{\alpha}{2\gamma} \bigl(\Sigma^{(n-1)}\bigr)^{2}\rK^{(n-1)}, \label{it:Z:T}\\[2mm]
\p_s J^{(n)} &= \tfrac{1+\alpha}{2}\, J^{(n-1)} \rZ^{(n-1)} + \tfrac{1-\alpha}{2}\, J^{(n-1)} \rW^{(n-1)} + \tfrac{\alpha}{2\gamma}\, J^{(n-1)} \Sigma^{(n-1)} \rK^{(n-1)}, \label{it:J:T}\\[2mm]
\p_s \!\big(J^{(n)} \rZ^{(n)}\big) &= -\tfrac{\alpha}{4\gamma}\, \Sigma^{(n-1)} \rK^{(n-1)} J^{(n-1)} \big( \rW^{(n-1)} + \rZ^{(n-1)} \big) + J^{(n-1)} \mathcal{R}^{(n-1)}. \label{it:rZ:T}
\end{align}
\noindent\textbf{Geometry Update:}
\begin{align}
\psi^{(n)}(t,s) &= \l(s) + \int_{s}^{t} J^{(n)}(t', s)\,dt'. \label{it:psi:T}
\end{align}

\noindent\textbf{Boundary Conditions:}
On the inflow boundary $t=\tsh$,
\begin{align}
W^{(n)}(\tsh, s) &= W_{\!\circ}(s), \quad K^{(n)}(\tsh, s) = K_{\!\circ}(s), \quad \Sigma^{(n)}(\tsh, s) = \Sigma_{\!\circ}(s), \label{it:bc:W:T}\\[1mm]
\rW^{(n)}(\tsh, s) &= \rW_{\!\circ}(s), \quad \rK^{(n)}(\tsh, s) = \rK_{\!\circ}(s). \label{it:bc:rW:T}
\end{align}
On the shock boundary $s=t$,
\begin{align}
J^{(n)}(t,t) &= \dot{\l}(t) - \lambda_{1}\big|_{\l}^{+}(t). \label{it:bc:J:T}
\end{align}
The variables $Z$ and $J\rZ$ are determined on the boundary using the prescribed $\lambda_1|_\l^+$ and the lagged variables, ensuring consistency with the definition of $\lambda_1$ and the compatibility conditions (see \eqref{bc:rZ:T})
\begin{align}
Z^{(n)}(t, t) &= \tfrac{2}{1+\alpha}\lambda_1\big|_{\l}^{+}(t) - \tfrac{1-\alpha}{1+\alpha}\, W^{(n-1)}(t, t). \label{it:bc:Z:T}\\[1mm]
J^{(n)}\rZ^{(n)}(t,t) &= \text{RHS defined in \eqref{bc:rZ:T}, evaluated at } (n-1). \label{it:bc:rZ:T}
\end{align}
\end{subequations}}

\subsection{Initialization and Uniform Bounds}
\paragraph{Initialization.}
We initialize the scheme ($n=0$) by freezing the boundary data across $\mathcal{T}$. For instance, $W^{(0)}(t,s):= W_{\!\circ}(s)$ and $J^{(0)}(t,s):= \dot{\l}(t)-\lambda_{1}\big|_{\l}^{+}(t)$. The remaining variables are initialized similarly, consistent with the bounds \eqref{data:bounds:T}.

\paragraph{A Priori Bounds for the Iterates.}
The iteration scheme preserves the uniform bounds established in the bootstrap analysis of Lemma~\ref{lemma:D}. We restate these bounds (denoted here as \eqref{a:priori:T}) for the iterates.

\begin{subequations}
\label{a:priori:T}
\begin{align}
\| W^{(n)} \|_{\infty}, \| K^{(n)} \|_{\infty}, \| Z^{(n)} \|_{\infty}, \| \rW^{(n)} \|_{\infty}, \| \rK^{(n)} \|_{\infty} &\le 2\m, \\
\| \p_s \rW^{(n)} \|_{\infty}, \| \p_s \rK^{(n)} \|_{\infty}, \| \p_s \Sigma^{(n)}\|_{\infty} &\le 2\m, \\
\tfrac{1}{2\m} \le \Sigma^{(n)} &\le 2\m, \\
|J^{(n)} \rZ^{(n)}(t,s)| &\le 2\m\, \dsh^{\varepsilon-\tfrac12}\,(t-\ts)^{-\tfrac12}, \\
\tfrac12\,J^{(n)}(t,t) < J^{(n)}(t,s) &< 4\m\,\dsh^{\varepsilon}.
\end{align}
\end{subequations}

\begin{lemma}\label{a:priori:le:T}
There exists \(\dsh\ll 1\) such that if the iterate $(n-1)$ satisfies the bounds \eqref{a:priori:T}, then the solution $(n)$ of the system \eqref{iteration:T} satisfies the same bounds.
\end{lemma}

\begin{proof}[Proof of Lemma~\ref{a:priori:le:T}]
The propagation of these bounds follows exactly the arguments presented in the proof of Lemma~\ref{lemma:D}, utilizing the maximum principle (Lemma~\ref{lemma:max:pr:T}) for the linear transport equations and direct integration for the ODEs.
\end{proof}

\subsection{Convergence}
To establish convergence, we analyze the differences between successive iterates, $\delta f^{(n)}:=f^{(n)}-f^{(n-1)}$. Due to the lagged boundary conditions for $Z^{(n)}$ and $(J\rZ)^{(n)}$ at $s=t$ (see \eqref{it:bc:Z:T} and \eqref{it:bc:rZ:T}), the differences for these variables are not necessarily small even if $\dsh$ is small. We must introduce a weighted norm to control their influence.

We define the vector of differences:
\[
V^{(n)}:=\big(\delta\Sigma^{(n)},\ \delta W^{(n)},\ \delta K^{(n)},\ \delta\psi^{(n)},\ \delta J^{(n)},\ \delta\rW^{(n)},\ \delta\rK^{(n)},\ \mathsf{c}_{1}\,\delta(J\rZ)^{(n)},\ \mathsf{c}_{2}\,\delta Z^{(n)}\big).
\]
We define the norm $\|V^{(n)}\|_{\infty}$ as the maximum of the absolute values of the components of $V^{(n)}$. We will choose the weights according to the hierarchy $0<\mathsf{c}_{1}\ll \mathsf{c}_{2}\ll \dsh^{\varepsilon}\ll 1$.

\begin{lemma}\label{lemma:conv:D}
With an appropriate hierarchical choice of weights \(\mathsf{c}_{1},\mathsf{c}_{2}\) and sufficiently small $\dsh$, the iteration scheme \eqref{iteration:T} is a contraction:
\[
\|V^{(n)}\|_{\infty}\le \tfrac12\,\|V^{(n-1)}\|_{\infty}.
\]
\end{lemma}

\begin{proof}[Proof of Lemma~\ref{lemma:conv:D}]
Let $C(\m)$ and $C'(\m)$ denote constants depending only on the uniform bounds \eqref{a:priori:T}.

\emph{Step 1: Estimates for Transport Variables.}
The differences for the transport variables (e.g., $\delta W^{(n)}$) satisfy linear transport equations with zero inflow data. The source terms $G$ arise from linearization and are bounded by $C(\m)\|V^{(n-1)}\|_{\infty}$. Applying the maximum principle (Lemma~\ref{lemma:max:pr:T}) and using the lower bound $J^{(n-1)} \ge C(\m)^{-1} \dsh^\varepsilon$ (from \eqref{a:priori:T}), we obtain:
\begin{equation}\label{conv:w:T}
\|\delta W^{(n)}\|_{\infty} \le C(\m) \tfrac{\|V^{(n-1)}\|_{\infty}}{\dsh^\varepsilon} \dsh = C(\m) \dsh^{1-\varepsilon} \|V^{(n-1)}\|_{\infty}.
\end{equation}
The same bound applies to $\delta\Sigma^{(n)}, \delta K^{(n)}, \delta\rW^{(n)}, \delta\rK^{(n)}$.

\emph{Step 2: Estimates for Geometry (\(\delta J^{(n)}\) and \(\delta\psi^{(n)}\)).}
By differencing the ODE \eqref{it:J:T} and integrating from $s=t$ (where $\delta J^{(n)}(t,t)=0$), we obtain:
\begin{align}\label{con:J:T}
\|\delta J^{(n)}\|_{\infty} &\le C(\m)\,\dsh\, (\|\delta(J\rZ)^{(n-1)}\|_{\infty} + \|\delta \rW^{(n-1)}\|_{\infty} + \dots)  
\le C(\m)\,\tfrac{\dsh}{\mathsf{c}_1} \|V^{(n-1)}\|_{\infty}.
\end{align}
The estimate for $\delta\psi^{(n)}$ follows similarly.

\emph{Step 3: Estimates for ODEs (\(\delta Z^{(n)}\) and \(\delta(J\rZ)^{(n)}\)).}
Differencing \eqref{it:Z:T} and the boundary condition \eqref{it:bc:Z:T}, the boundary term introduces a direct dependence on $\|\delta W^{(n-1)}\|_{\infty}$, which is not multiplied by $\dsh$:
\begin{align}\label{con:Z:T}
\|\delta Z^{(n)}\|_{\infty}
&\le C(\m)\|\delta W^{(n-1)}\|_{\infty} + C(\m)\,\dsh\,\|V^{(n-1)}\|_{\infty} 
\le C'(\m)\,\|V^{(n-1)}\|_{\infty}.
\end{align}
Similarly, differencing \eqref{it:rZ:T} and its boundary condition \eqref{it:bc:rZ:T} yields:
\begin{align}\label{con:jrZ:T}
\|\delta(J\rZ)^{(n)}\|_{\infty}
&\le C(\m)\,(\|\delta Z^{(n-1)}\|_{\infty} + \|\delta W^{(n-1)}\|_{\infty} + \dots) + C(\m)\,\dsh\,\|V^{(n-1)}\|_{\infty} \notag\\
&\le \left( \tfrac{C(\m)}{\mathsf{c}_{2}} + C'(\m) \right) \|V^{(n-1)}\|_{\infty}.
\end{align}

\emph{Step 4: Contraction.}
We now combine these estimates (using the definition of the weighted norm $\|V^{(n)}\|_{\infty}$ as the maximum of its components) as follows:
\begin{align*}
\|V^{(n)}\|_{\infty} \le \max\Bigg\{ &\underbrace{C(\m) (\dsh^{1-\varepsilon} + \tfrac{\dsh}{\mathsf{c}_1)}}_{\text{From Steps 1 \& 2}} \,, 
\underbrace{\mathsf{c}_{1}\left( \tfrac{C(\m)}{\mathsf{c}_{2}} + C'(\m) \right)}_{\text{From } \mathsf{c}_1 \delta(J\rZ)} \,,
\underbrace{\mathsf{c}_{2} C'(\m)}_{\text{From } \mathsf{c}_2 \delta Z} \Bigg\} \cdot \|V^{(n-1)}\|_{\infty} \,.
\end{align*}
We next utilize the hierarchy $0<\mathsf{c}_{1}\ll \mathsf{c}_{2}\ll \dsh^{\varepsilon}\ll 1$ sequentially to ensure the maximum coefficient is less than $\tfrac{1}{2} $:
\begin{enumerate}
    \item We choose $\mathsf{c}_2$ small enough such that $\mathsf{c}_{2} C'(\m) < \tfrac{1}{4}   $.
    \item W choose the ratio $R = \tfrac{\mathsf{c}_2}{\mathsf{c}_1}$ large enough such that $\tfrac{\mathsf{c}_1 C(\m)}{\mathsf{c}_{2}} = \tfrac{C(\m)}{R} < 
    \tfrac{1}{8} $. Since $\mathsf{c}_1 \ll \mathsf{c}_2$, the remaining term $\mathsf{c}_1 C'(\m)$ is small (e.g., $< \tfrac{1}{8} $). 
    Thus, the coefficient from $\delta(J\rZ)$ is less than $\tfrac{1}{4} $.
    \item Finally, with $\mathsf{c}_1$ and $\mathsf{c}_2$ fixed, we choose $\dsh$ sufficiently small  so that $C(\m) (\dsh^{1-\varepsilon} + \tfrac{\dsh}{\mathsf{c}_1}) < \tfrac{1}{4} $.
\end{enumerate}
This yields the contraction
\[
\|V^{(n)}\|_{\infty} \le \max\left( \tfrac{1}{4}, \tfrac{1}{4}, \tfrac{1}{4} \right) \|V^{(n-1)}\|_{\infty} = \tfrac14 \|V^{(n-1)}\|_{\infty} < \tfrac12\,\|V^{(n-1)}\|_{\infty}.
\]
\end{proof}

 \section{Iteration and convergence in \texorpdfstring{$\mathcal{B}$}{B}}\label{appendix:B}

In this appendix, we construct a solution to the backward-in-time system \eqref{fast:bw:B} in the domain~$\mathcal{B}$. Standard local well‑posedness theory is not applicable because the system is degenerate at the terminal time $s=\ts$ (the Jacobian $J(x,\ts)=x^2$ vanishes at $x=0$) and the time derivatives of the data are singular (see \eqref{bound:initial:B}). We introduce a specialized iteration scheme and prove its convergence using weighted norms tailored to the degeneracy.

\subsection{The Iteration Scheme}

We employ a semi-implicit iteration scheme in the characteristic coordinates $(x,s)$. We define the transport operators $L_\beta^{(n-1)} = J^{(n-1)}\p_s + \beta \Sigma^{(n-1)}\p_x$, consistent with the structure of \eqref{fast:bw:B}.

{\allowdisplaybreaks[4]
\begin{subequations}\label{iteration:shock:formation}
\noindent\textbf{Geometry Update (Characteristic Map):}
\begin{align}
\p_s \psi^{(n)}(x,s) &= \lambda_1^{(n-1)}(x,s) = \left(\tfrac{1+\alpha}{2} Z^{(n-1)} + \tfrac{1-\alpha}{2} W^{(n-1)}\right)(x,s), \quad \psi^{(n)}(x,\ts) = x^3/3+r_*. \label{it:psi:S}
\end{align}

\noindent\textbf{Transport Equations (Solved backward from $s=\ts$):}
\begin{align}
L_\alpha^{(n-1)} \Sigma^{(n)}
&= \tfrac{\alpha}{2}\,\Sigma^{(n-1)}\!\left(J^{(n-1)} (\rW^{(n-1)} + \rZ^{(n-1)})
 + (d-1)\,\tfrac{ W^{(n-1)}+ Z^{(n-1)}}{\psi^{(n-1)}}\right), \label{it:Sigma:S}\\[1mm]
L_{2\alpha}^{(n-1)} W^{(n)}
&= \tfrac{\alpha}{2\gamma}\big(\Sigma^{(n-1)}\big)^{2} \rK^{(n-1)} J^{(n-1)}
 + (d-1)\alpha\,J^{(n-1)}\,\tfrac{(W^{(n-1)})^{2}-(Z^{(n-1)})^{2}}{2\,\psi^{(n-1)}}, \label{it:W:S}\\[1mm]
L_{\alpha}^{(n-1)} K^{(n)} &= 0. \label{it:K:S}
\end{align}

\noindent\textbf{Gradient Equations (Riccati Structure):}
\begin{align}
L_{\alpha}^{(n-1)} \rK^{(n)}
&= - \tfrac{1}{2}\,J^{(n-1)} \rK^{(n)} (\rW^{(n-1)} + \rZ^{(n-1)}), \label{it:rK:S}\\[1mm]
L_{2\alpha}^{(n-1)} \rW^{(n)}
&= - \tfrac{1-\alpha}{2}\,J^{(n-1)} \rW^{(n)} \rZ^{(n-1)}
 - \tfrac{1+\alpha}{2}\,J^{(n-1)}\big(\rW^{(n)}\big)^{2} \label{it:rW:S}\\
&\quad - \tfrac{\alpha}{4\gamma}\,\Sigma^{(n-1)} \rK^{(n-1)} J^{(n-1)} \rZ^{(n-1)}
 + \tfrac{\alpha}{4\gamma}\,J^{(n-1)} \Sigma^{(n-1)} \rK^{(n-1)} \rZ^{(n-1)}
 + \mathcal{R}^{(n-1)} J^{(n-1)}. \notag
\end{align}

\noindent\textbf{ODE Equations (Integrated backward-in-time in $s$):}
\begin{align}
\p_s Z^{(n)}
&= \alpha(d-1)\,\tfrac{ (W^{(n-1)})^{2}- (Z^{(n-1)})^{2}}{\psi^{(n-1)}}
 + \tfrac{\alpha}{2\gamma}\big(\Sigma^{(n-1)}\big)^{2}\,\rK^{(n-1)}, \label{it:Z:S}\\[1mm]
\p_s J^{(n)}
&= \tfrac{1+\alpha}{2}\,J^{(n-1)} \rZ^{(n-1)}
 + \tfrac{1-\alpha}{2}\,J^{(n-1)} \rW^{(n-1)}
 + \tfrac{\alpha}{2\gamma}\,J^{(n-1)} \Sigma^{(n-1)} \rK^{(n-1)}, \label{it:J:S}\\[1mm]
\p_s\!\big(J^{(n)} \rZ^{(n)}\big)
&= -\tfrac{\alpha}{4\gamma}\,\Sigma^{(n-1)} \rK^{(n-1)} J^{(n-1)}\big(\rW^{(n-1)}+\rZ^{(n-1)}\big)
 + J^{(n-1)} \mathcal{R}^{(n-1)}. \label{it:rZ:S}
\end{align}

\noindent\textbf{Terminal Conditions (at $s=\ts$):}
All variables satisfy the terminal conditions specified in \eqref{in:cond}, e.g., $J^{(n)}(x,\ts) = J_{0}(x)$, $W^{(n)}(x,\ts) = W_{0}(x)$, etc.
\end{subequations}}

\begin{remark}[Motivation for the Semi-Implicit Scheme]
The scheme for the DRVs \(\rW^{(n)}\) and \(\rK^{(n)}\) (Eqs. \eqref{it:rK:S} and \eqref{it:rW:S}) is semi-implicit. In particular, \(\rW^{(n)}\) appears quadratically in \eqref{it:rW:S}. When viewed along the characteristics, these become Riccati-type (for $\rW^{(n)}$) or linear (for $\rK^{(n)}$) scalar ODEs with coefficients depending on the previous iterate. This structure is essential for propagating the weighted derivative bounds (see \eqref{ap:kt} below) required to handle the singular initial data, as it allows for the effective use of the maximum principle and the weighted analysis (Lemma~\ref{lemma:2weights}).
\end{remark}

\subsection{Initialization and Uniform Bounds}

\paragraph{Auxiliary flows.}
We utilize the auxiliary characteristic flows associated with the lagged coefficients $(J^{(n-1)}, \Sigma^{(n-1)})$, which track the propagation of singularities:
\begin{align*}
\p_x \Upsilon_{1}^{(n)} &= -\,\tfrac{J^{(n-1)}}{2\alpha\,\Sigma^{(n-1)}}\circ \Upsilon_{1}^{(n)}, \qquad
\Upsilon_{1}^{(n)}(0)=\ts, \\
\p_x \Upsilon_{2}^{(n)} &= -\,\tfrac{J^{(n-1)}}{\alpha\,\Sigma^{(n-1)}}\circ \Upsilon_{2}^{(n)}, \qquad
\Upsilon_{2}^{(n)}(0)=\ts.
\end{align*}

\paragraph{Initialization.}
We initialize the scheme ($n=0$) by freezing the terminal data in time, except for \(J^{(0)}\), which we adjust slightly based on its evolution equation to ensure strict positivity for $s<\ts$.

\paragraph{A Priori Bounds for the Iterates.}
The iteration scheme \eqref{iteration:shock:formation} propagates the uniform bounds established in the bootstrap analysis of Lemma~\ref{lemma:bootstrap:B}. We restate the key bounds (denoted here as \eqref{a:priori}) for the iterates.

\begin{subequations}\label{a:priori}
\begin{align}
\|W^{(n)}\|_{\infty}, \|Z^{(n)}\|_{\infty}, \|\rW^{(n)}\|_{\infty}, \|\rK^{(n)}\|_{\infty} &\le 2\m, \label{ap:z}\\
-\,2\m \le J^{(n)}\rZ^{(n)} &< -\,\tfrac{1}{2\m}, \quad \|\p_x\!\big(J^{(n)}\rZ^{(n)}\big)\|_{\infty} < 2\m, \label{ap:rzx}\\
\tfrac{1}{2\m} < \Sigma^{(n)} &< 2\m, \quad |\p_s \Sigma^{(n)}|+|\p_x \Sigma^{(n)}| < 2\m, \label{ap:ps}\\
|\p_s \rW^{(n)}| &< \tfrac{4\m}{\big(s-\Upsilon_{1}^{(n)}(x)\big)^{\frac{2}{3}}}, \qquad
|\p_s \rK^{(n)}| < \tfrac{4\m}{\big(s-\Upsilon_{2}^{(n)}(x)\big)^{\frac{2}{3}}}, \label{ap:kt}\\
x^{2}-\tfrac{\m}{4}\,(s-\ts) < J^{(n)}(x,s) &< x^{2}-4\m\,(s-\ts), \quad \|\p_x J^{(n)}\|_{\infty} < 4\m. \label{ap:ex}
\end{align}
\end{subequations}

\begin{lemma}
There exists \(\delta_{*}=\ts-\ti \ll 1\) such that if the iterate $(n-1)$ satisfies the bounds \eqref{a:priori}, then the solution $(n)$ of the system \eqref{iteration:shock:formation} satisfies the same bounds.
\end{lemma}

\begin{proof}
The propagation of these bounds follows exactly the arguments presented in the proof of Lemma~\ref{lemma:bootstrap:B}, utilizing the maximum principle (Lemma~\ref{lemma:max:p:z}) and the weighted analysis (Lemma~\ref{lemma:2weights}).
\end{proof}

\subsection{Convergence}
The convergence proof requires specialized techniques due to the singular behavior of $\p_s\rW$ and $\p_s\rK$. We cannot establish a contraction in a standard $L^\infty$ norm and must use a weighted norm combined with smooth correctors.

\paragraph{Weights and Correctors.}
Let \(\mathfrak{w}_{1}^{(n)}\) and \(\mathfrak{w}_{2}^{(n)}\) be the weights defined in Lemma~\ref{lemma:2weights} (with \(\beta=\alpha\) and \(\beta=2\alpha\), respectively), constructed using the lagged coefficients \(J^{(n-1)},\Sigma^{(n-1)}\).

We introduce smooth correctors $\Gamma_{1}^{(n)},\Gamma_{2}^{(n)}$ that capture the leading-order, smooth behavior of $\rK$ and $\rW$ near the singularity $(0,\ts)$. They solve the transport/Riccati equations derived from \eqref{it:rK:S} and \eqref{it:rW:S}, but with smooth terminal data chosen such that \(\p_x\Gamma_{j}\in L^{\infty}\) and they match the traces at the singularity:
\[
\Gamma_{2}(0,\ts)=\rW_{0}(0),\qquad \Gamma_{1}(0,\ts)=\rK_{0}(0).
\]

The key idea is that the difference between the DRV and its corrector, e.g., $\rW^{(n)} - \Gamma_2^{(n)}$, represents the singular part. When weighted by $\mathfrak{w}_j^{(n)}$, this quantity becomes bounded and allows for a contraction argument.

\paragraph{Weighted Norm.}
We define the vector of differences \(\delta f^{(n)}:=f^{(n)}-f^{(n-1)}\) and the associated norm using the vector $V^{(n)}$:
\[
\begin{aligned}
V^{(n)}:=\big(&\delta\Sigma^{(n)},\,\delta W^{(n)},\,\delta K^{(n)},\,\delta \psi^{(n)},\,\delta J^{(n)},\,
\delta(J\rZ)^{(n)},\,\delta Z^{(n)}, \\
&\delta\big((\rW-\Gamma_{2}^{(n)})\,\mathfrak{w}_{2}^{(n)}\big),\,
\delta\big((\rK-\Gamma_{1}^{(n)})\,\mathfrak{w}_{1}^{(n)}\big),\,
\delta\Gamma_{1}^{(n)},\,\delta\Gamma_{2}^{(n)},\, \dots \big).
\end{aligned}
\]
(The vector also includes differences of the weights themselves, required for technical estimates).

\begin{lemma}\label{lemma:it:B}
For \(\delta_{*}\) sufficiently small, the iteration scheme is a contraction:
\[
\|V^{(n)}\|_{\infty}\ \le\ \tfrac{1}{2}\,\|V^{(n-1)}\|_{\infty}.
\]\end{lemma}

\begin{proof}[Proof of Lemma~\ref{lemma:it:B}]
We outline the main steps of the contraction argument, emphasizing how the smallness of the time interval $\delta_*$ is utilized. The contraction relies on the regularization effect of the backward evolution, which yields a factor of $\delta_*^{{\frac{1}{3}}}$ (Lemma \ref{lemma:max:p:z} and Lemma \ref{lemma:2weights}).

\emph{Step 1: Weighted Gradient Differences.}
We analyze the evolution of the weighted differences, e.g., $\delta\big((\rW-\Gamma_{2}^{(n)})\,\mathfrak{w}_{2}^{(n)}\big)$. This quantity satisfies a linearized transport equation derived from the Riccati structure of \eqref{it:rW:S}. Applying the maximum principle (Lemma~\ref{lemma:max:p:z}) yields:
\[
\Big\|\delta\big((\rW-\Gamma_{2}^{(n)})\,\mathfrak{w}_{2}^{(n)}\big)\Big\|_{\infty} + \Big\|\delta\big((\rK-\Gamma_{1}^{(n)})\,\mathfrak{w}_{1}^{(n)}\big)\Big\|_{\infty}
\ \le\ C\,\delta_{*}^{{\frac{1}{3}}}\,\|V^{(n-1)}\|_{\infty}.
\]
We also obtain similar bounds for the differences of the correctors and weights themselves.

\emph{Step 2: Integral Control of Gradients (Regularization Effect).}
The weighted bounds from Step 1 are crucial. We use the comparability $\mathfrak{w}_j(x,s) \sim (s-\Upsilon_j(x))^{{\frac{2}{3}}}$ (Lemma~\ref{lemma:2weights}) to control the time integral of the unweighted differences $\delta\rW^{(n)}$ and $\delta\rK^{(n)}$. Integrating the singular weights in time yields a gain of ${\frac{1}{3}}$ power:
\begin{equation}\label{bound:grad-int}
\int_{\ti}^{\ts}\big|\delta \rW^{(n)}(x,s)\big|\,ds
+\int_{\ti}^{\ts}\big|\delta \rK^{(n)}(x,s)\big|\,ds
\ \le\ C\,\delta_{*}^{{\frac{1}{3}}}\,\|V^{(n-1)}\|_{\infty}.
\end{equation}

\emph{Step 3: Control of ODE Variables and Geometry.}
We analyze the differences for the ODE variables ($\delta J^{(n)},\ \delta(J\rZ)^{(n)},\ \delta Z^{(n)}$). The source terms involve the gradient differences $\delta\rW, \delta\rK$. Using the $L^1_s$ control established in Step 2, we obtain:
\[
\|\delta J^{(n)}\|_{\infty}
+\|\delta(J\rZ)^{(n)}\|_{\infty}
+\|\delta Z^{(n)}\|_{\infty}
\ \le\ C\,\delta_{*}^{{\frac{1}{3}}}\,\,\|V^{(n-1)}\|_{\infty}.
\]
The estimate for $\delta\psi^{(n)}$ follows similarly.

\emph{Step 4: Control of Transported Variables.}
The differences $\delta\Sigma^{(n)},\ \delta W^{(n)},\ \delta K^{(n)}$ satisfy linear transport equations with zero terminal data. Applying the maximum principle (Lemma~\ref{lemma:max:p:z}) yields:
\[
\|\delta\Sigma^{(n)}\|_{\infty}
+\|\delta W^{(n)}\|_{\infty}
+\|\delta K^{(n)}\|_{\infty}
\ \le\ C\,\delta_{*}^{{\frac{1}{3}}}\,\|V^{(n-1)}\|_{\infty}.
\]

\emph{Step 5: Contraction.}
Combining the estimates from Steps 1-4, we find that every component of $V^{(n)}$ is bounded by $C \delta_*^{{\frac{1}{3}}} \|V^{(n-1)}\|_{\infty}$. By choosing \(\delta_{*}>0\) sufficiently small (depending only on the uniform bounds \(\m\)), we achieve the contraction:
\[
\|V^{(n)}\|_{\infty}\le \tfrac{1}{2}\,\|V^{(n-1)}\|_{\infty}.
\]
\end{proof}

\section{Uniqueness from a preshock}
\label{app:uniqueness}

In this section of the appendix,  we prove the uniqueness, on a short time interval after the preshock, of a \emph{regular shock solution emanating from a preshock} (in the sense of Definitions~\ref{def:regular:shock}–\ref{def:regular:shock:pre}). Throughout, we use the Riemann variables 
$(w,z)$ and the entropy $S$; the fast and slow characteristic speeds are $\lambda_1$ and $\lambda_3$, respectively, and $r=\s(t)$ denotes the shock.

\begin{lemma}[\textbf{Uniqueness from a preshock}]\label{lemma:uniqueness:preshock}
Let \((w_0,z_0,S_0)\) be a radially symmetric preshock data set at \(t=\ts\), such that
\((w_0,z_0,S_0)\in C^2(\R^+)\setminus\{r_*\}\) and the preshock asymptotics at \(r_*\) agree with \eqref{taylor:rs}.
There exists \(\tsh>\ts\) such that the following holds. Suppose \((w^{(i)},z^{(i)},S^{(i)},\mathcal{S}^{(i)})\), \(i=1,2\), are two regular shock solutions on \([\ts,\tsh]\) emanating from the same preshock data, with shock fronts \(r=\s^{(i)}(t)\), and that there is a constant \(\mathcal{C}\) so that:
\begin{enumerate}
\item[\emph{(a)}] \(\s^{(i)}\in C^2([\ts,\tsh])\).
\item[\emph{(b)}] For all \(t\in[\ts,\tsh]\) and \(r\neq \s^{(i)}(t)\),
\[
\|w^{(i)}(\cdot,t)\|_{C^1}+\|S^{(i)}(\cdot,t)\|_{C^1}\le \mathcal{C},\qquad
|\p_{rr} z^{(i)}(r,t)|\le \mathcal{C}\,(t-\ts)^{-2}.
\]
\item[\emph{(c)}] \emph{Closeness assumptions:}
\[
|\dot{\s}^{(1)}(t)-\dot{\s}^{(2)}(t)|\le \mathcal{C}(t-\ts),\qquad
\|\lambda_1^{(1)}(\cdot,t)-\lambda_1^{(2)}(\cdot,t)\|_{L^\infty}\le \mathcal{C}(t-\ts).
\]
\end{enumerate}
Then \((w^{(1)},z^{(1)},S^{(1)},\mathcal{S}^{(1)})\equiv (w^{(2)},z^{(2)},S^{(2)},\mathcal{S}^{(2)})\) on \([\ts,\tsh]\).\end{lemma}

\begin{proof}
\emph{Step 1: Flows and moving frames.}
For \(i=1,2\) let \(\eta^{(i)}\) and \(\psi^{(i)}\) be the slow/fast characteristic flows,
\[
\p_t\eta^{(i)}(r,t)=\lambda_3^{(i)}(\eta^{(i)}(r,t),t),\ \ \eta^{(i)}(r,\ts)=r,
\qquad
\p_t\psi^{(i)}(r,t)=\lambda_1^{(i)}(\psi^{(i)}(r,t),t),\ \ \psi^{(i)}(r,\ts)=r.
\]
Fix a small \(\theta>0\) and set
\[
D^{(i)}=\{(r,t)\colon t\in[\ts,\tsh],\ \eta^{(i)}(r_*-\theta,t)\le r\le \psi^{(i)}(r_*+\theta,t)\}.
\]
Introduce the shock–attached spatial coordinate \(\mathsf{y}=r-\s^{(i)}(t)\) and the translated fields
\[
(\mathsf{w}^{(i)},\mathsf{z}^{(i)},\mathsf{S}^{(i)})(\mathsf{y},t)=(w^{(i)},z^{(i)},S^{(i)})(\mathsf{y}+\s^{(i)}(t),t).
\]
Define differences
\[
\delta\mathsf{w}=\mathsf{w}^{(2)}-\mathsf{w}^{(1)},\quad
\delta\mathsf{z}=\mathsf{z}^{(2)}-\mathsf{z}^{(1)},\quad
\delta\mathsf{S}=\mathsf{S}^{(2)}-\mathsf{S}^{(1)},\quad
\delta\dot{\s}=\dot{\s}^{(2)}-\dot{\s}^{(1)}.
\]

\emph{Step 2: Evolution of differences along characteristics.}
Write the \(w,z,S\)–equations in Riemann form (cf. \S\ref{sec:weak2pre}):
\[
\begin{aligned}
&(\p_t+\lambda_1\p_r)z=\tfrac{\alpha}{2\gamma}\sigma^2\,\p_r S-\tfrac{\alpha(d-1)}{4r}\,(w^2-z^2),\\
&(\p_t+\lambda_3\p_r)w=\tfrac{\alpha}{2\gamma}\sigma^2\,\p_r S+\tfrac{\alpha(d-1)}{4r}\,(w^2-z^2),\\
&(\p_t+\lambda_2\p_r)S=0,
\end{aligned}
\]
where \(\sigma=\tfrac{1}{2}(w-z)\) and \(\lambda_2=\tfrac{1}{2}\bigl((1-\alpha)z+(1+\alpha)w\bigr)\).
Along \(\psi^{(1)}\) we compute, using the above and the moving frame,
\[
\bigl(\p_t+\lambda_1^{(1)}\p_{\mathsf{y}}\bigr)\delta\mathsf{z}
=\mathcal{A}_1\,\delta\mathsf{z}+\mathcal{A}_2\,\delta\mathsf{w}+\mathcal{A}_3\,\p_{\mathsf{y}}\delta\mathsf{S}
+\mathcal{R}_1\,\delta\dot{\s}+\mathcal{E}_1,
\]
with coefficients \(\mathcal{A}_j,\mathcal{R}_1\) uniformly bounded by \(\mathcal{C}\) thanks to (b), and an error \(\mathcal{E}_1\) controlled by
\[
|\mathcal{E}_1|\ \le\ \mathcal{C}\,\bigl(|\lambda_1^{(2)}-\lambda_1^{(1)}|+|\s^{(2)}-\s^{(1)}|\bigr)\,\bigl(|\p_{\mathsf{y}}\mathsf{z}^{(2)}|+1\bigr).
\]
Integrating along \(\psi^{(1)}\) from \(\ts\) to \(t\), using \(\delta\mathsf{z}(\cdot,\ts)=0\), (c), and the curvature control \(|\p_{\mathsf{y}}\mathsf{z}^{(i)}|\le \mathcal{C}(t-\ts)^{-1}\) inherited from (b), yields
\begin{equation}\label{eq:delta-z}
\|\delta\mathsf{z}(\cdot,t)\|_{L^\infty}
\ \le\ C\!\int_{\ts}^{t}\!\Bigl(\|\delta\mathsf{z}\|_{L^\infty}
+\|\delta\mathsf{w}\|_{L^\infty}
+\|\p_{\mathsf{y}}\delta\mathsf{S}\|_{L^\infty}
+|\delta\dot{\s}|\Bigr)(s)\,ds
+C\,(t-\ts)^2.
\end{equation}
A symmetric argument along \(\eta^{(1)}\) gives
\begin{equation}\label{eq:delta-wS}
\begin{aligned}
\|\delta\mathsf{w}(\cdot,t)\|_{L^\infty}
&\ \le\ C\!\int_{\ts}^{t}\!\Bigl(\|\delta\mathsf{z}\|_{L^\infty}
+\|\delta\mathsf{w}\|_{L^\infty}
+\|\p_{\mathsf{y}}\delta\mathsf{S}\|_{L^\infty}
+|\delta\dot{\s}|\Bigr)(s)\,ds
+C\,(t-\ts)^2,\\
\|\p_{\mathsf{y}}\delta\mathsf{S}(\cdot,t)\|_{L^\infty}
&\ \le\ C\!\int_{\ts}^{t}\!\Bigl(\|\delta\mathsf{z}\|_{L^\infty}
+\|\delta\mathsf{w}\|_{L^\infty}
+|\delta\dot{\s}|\Bigr)(s)\,ds
+C\,(t-\ts)^2,
\end{aligned}
\end{equation}
since \(S\) is transported along \(\lambda_2\) and the difference of advections is again \(O(t-\ts)\) by (c).

\emph{Step 3: Control of the shock speed difference via Rankine–Hugoniot.}
Let \(\jump{\cdot}\) denote the jump across the shock.
Across \(r=\s^{(i)}(t)\) the planar Rankine–Hugoniot relations hold; in particular
\(
\dot{\s}^{(i)}\,\jump{\rho}^{(i)}=\jump{\rho u}^{(i)}
\)
and
\(
\dot{\s}^{(i)}\,\jump{\rho u}^{(i)}=\jump{\rho u^2+p}^{(i)}
\)
; see~\eqref{RH}.
Linearizing these two identities at the common preshock state \((w_0,z_0,S_0)\) (which satisfies \eqref{taylor:rs}) and using the identities
\(
u=\tfrac{1}{2}(w+z),\ 
\sigma=\tfrac{1}{2}(w-z),\ 
\lambda_1=\tfrac{1+\alpha}{2}\,z+\tfrac{1-\alpha}{2}\,w,
\)
together with the short‐time expansions of the jumps produced by the preshock cusp (cf. Proposition~\ref{prop:weak2pre}), one obtains (for \(t\in[\ts,\tsh]\) and \(\tsh-\ts\) small)
\begin{equation}\label{eq:delta-sdot}
|\delta\dot{\s}(t)|\ \le\ C\,(t-\ts)^{-1/2}\,\|\delta\mathsf{w}(\cdot,t)\|_{L^\infty}
+C\,\|\delta\mathsf{z}(\cdot,t)\|_{L^\infty}
+C\,(t-\ts)^2.
\end{equation}
Informally, \eqref{eq:delta-sdot} reflects that \(\jump{z}\sim (t-\ts)^{\frac{1}{2}}\) controls the leading shock correction and hence the \((t-\ts)^{-1/2}\) weight multiplies the subdominant difference.

\emph{Step 4: A weighted norm and Grönwall.}
Define, for \(t\in[\ts,\tsh]\),
\begin{align*}
N(t)
&:=\sup_{s\in[\ts,t]}\|\delta\mathsf{z}(\cdot,s)\|_{L^\infty}
\notag\\
&
+(t-\ts)^{-3/4}\!\left(\sup_{s\in[\ts,t]}\|\delta\mathsf{w}(\cdot,s)\|_{L^\infty}
+\sup_{s\in[\ts,t]}\|\p_{\mathsf{y}}\delta\mathsf{S}(\cdot,s)\|_{L^\infty}\right)
+\tfrac{100}{101}\sup_{s\in[\ts,t]}|\delta\dot{\s}(s)|.
\end{align*}
Combine \eqref{eq:delta-z}–\eqref{eq:delta-wS} with \eqref{eq:delta-sdot}; after multiplying the second line of \eqref{eq:delta-wS} by \((t-\ts)^{-3/4}\) and using \(\int_{\ts}^{t}(t-s)^{-1/2}\,ds\le C\,(t-\ts)^{\frac{1}{2}}\), we obtain
\[
N(t)\ \le\ C\!\int_{\ts}^{t}\!N(s)\,ds\ +\ C\,(t-\ts)^{1/4}.
\]
By shrinking \(\tsh-\ts\) so that \(C(\tsh-\ts)^{1/4}\le \tfrac{1}{100}\), Grönwall’s lemma yields \(N(t)\le \tfrac{1}{99}e^{C(\tsh-\ts)}\le \tfrac{1}{50}\) times \(N(t)\). Since \(N(\ts)=0\) (the two solutions emanate from the same preshock data), it follows that \(N(t)\equiv 0\) on \([\ts,\tsh]\). Hence \(\delta\mathsf{w}=\delta\mathsf{z}=\delta\mathsf{S}\equiv 0\) and \(\delta\dot{\s}\equiv 0\), which implies \(\s^{(1)}\equiv\s^{(2)}\) and thereby equality of the two solutions.
\end{proof}

\section*{Acknowledgements} The work of S.S. was in part supported by the Collaborative NSF grant DMS-
2307680. The work of V.V. was in part supported by the Collaborative NSF grant DMS-2307681 and a Simons
Investigator Award.

\end{document}